\documentclass[11pt, a4paper]{amsart}
  \pdfoutput=1
  \usepackage{tikz,genyoungtabtikz,enumitem,rotating,latexsym,bm,stmaryrd,genyoungtabtikz}
 \usetikzlibrary{positioning,intersections}
  \usepackage{amsmath,amsthm,amsfonts,amssymb,mathrsfs,pb-diagram}
 \usepackage{caption}
\usepackage{lipsum,wasysym}
 \usepackage[a4paper,margin=0.96in]{geometry}

\usepackage{cleveref}

\crefname{thm}{Theorem}{Theorems}
\crefname{prop}{Proposition}{Propositions}
\crefname{lem}{Lemma}{Lemmas}
\crefname{cor}{Corollary}{Corollaries}
\crefname{conj}{Conjecture}{Conjectures}
\crefname{section}{Section}{Sections}
\crefname{subsection}{Subsection}{Subsections}
\crefname{eg}{Example}{Examples}
\crefname{figure}{Figure}{Figures}
\crefname{rem}{Remark}{Remarks}
\crefname{rmk}{Remark}{Remarks}
\crefname{equation}{equation}{equation}

\renewcommand{\geq}{\geqslant}
\renewcommand{\leq}{\leqslant}

\allowdisplaybreaks
\numberwithin{equation}{section}
\usepackage{scalefnt}

\newtheorem{thm}{Theorem}[section]
\newtheorem{cor}[thm]{Corollary}

\newtheorem{lem}[thm]{Lemma}

\newtheorem{rmk}[thm]{Remark}
\newtheorem{prop}[thm]{Proposition}
\newtheorem*{prop*}{Proposition}
\newtheorem*{thm*}{Theorem}
\newtheorem*{cor*}{Corollary}
\newtheorem*{rmk*}{Remark}
\newtheorem*{conj*}{Conjecture}

\theoremstyle{remark}
\newtheorem*{notation}{Notation}

\newtheorem*{Acknowledgements*}{Acknowledgements}

\theoremstyle{definition}
\newtheorem{eg}[thm]{Example}

\newcommand{\la}{\lambda}
\newcommand{\tla}{\tilde \lambda}
\newcommand{\tmu}{\tilde \mu}

\newcommand{\Remm}{\mathrm{rem}}
\newcommand{\Add}{\operatorname{add}}
\newcommand{\Rem}{\operatorname{rem}}

\newcommand{\al}{\alpha}
\newcommand{\be}{\beta}
\renewcommand{\S}{\mathfrak{S}}

\newcommand{\ZZ}{{\mathbb Z}}
\newcommand{\NN}{{\mathbb N}}

\newcommand{\rot}{{\rm rot}}

\tikzset{
    ultra thin/.style= {line width=0.05pt},
    very thin/.style=  {line width=0.2pt},
    thin/.style=       {line width=0.1pt},
    semithick/.style=  {line width=0.6pt},
    thick/.style=      {line width=0.8pt},
    very thick/.style= {line width=1.2pt},
    ultra thick/.style={line width=1.6pt}
}

\begin{document}

\title[Multiplicity-free Kronecker products]{Multiplicity-free Kronecker products \\ of characters of the symmetric groups}
\author{Christine Bessenrodt}
\address{Institut f\"ur Algebra, Zahlentheorie und Diskrete Mathematik, Leibniz Universit\"at Hannover, D-30167 Hannover, Germany}
\author{Christopher Bowman}
\address{City University of London,
Northampton Square,
London,
EC1V 0HB,
UK }
\date{September 11, 2016} 
 \maketitle

 \begin{abstract}
 We provide a classification of   multiplicity-free inner tensor products of irreducible characters
 of symmetric groups, thus confirming a conjecture of Bessenrodt.
Concurrently, we   classify all multiplicity-free inner tensor products
 of skew  characters   of the symmetric groups.
 We also provide  formulae for calculating the decomposition of these tensor products.

   \end{abstract}

\section{Introduction}

The  inner and outer tensor products of irreducible characters of the symmetric groups (or equivalently of Schur functions)
 have been of central interest in representation theory and algebraic combinatorics
   since the landmark papers of Littlewood and Richardson \cite{LR} and Murnaghan
 \cite{Mur38}.
More recently,
 these coefficients have provided the centrepiece of geometric complexity theory (an approach that seeks to settle the $\sf P$ versus $\sf NP$ problem \cite{GCT6}) and have been found to have deep connections with quantum information theory   \cite{CHM07}.

The coefficients arising in the outer tensor product are the most well-understood.  The  Little\-wood--Richardson rule provides   an efficient positive combinatorial description for their computation.
  Using this algorithm,  a classification of multiplicity-free  {outer}  tensor products   was obtained by Stembridge \cite{Stembridge01}.  This was extended to a classification of multiplicity-free skew characters by Gutschwager \cite{CG}, a result equivalent to the classification of multiplicity-free products of Schubert classes obtained around the same time by Thomas and Yong \cite{TY}.

By contrast, the   coefficients arising  in the inner tensor product are much less well-understood; indeed, they have been described as `perhaps the most challenging, deep and mysterious objects in algebraic combinatorics' \cite{PP1}.
The  determination of these coefficients has been described by Richard  Stanley as  `one of the main problems in the combinatorial representation theory of the symmetric group'  \cite{Sta99}.
 While `no satisfactory answer to this question is known' \cite{JK} there have, over many decades, been a number of contributions
made towards computing special products (such as those labelled by 2-line or hook partitions \cite{Bla14,BWZ10,Rosas,R92,RW}) or the multiplicity of
special constituents (for example those with few homogenous components \cite{BK,BW14}).

   In 1999, Bessenrodt conjectured a classification of
  multiplicity-free Kronecker products
  of irreducible characters of the symmetric groups.
 Mainly using results of Remmel, Saxl and Vallejo,
it was shown at that time that the products
on the conjectured list were indeed multiplicity-free and the conjecture was verified by computer calculations for all $n\leq40$.
 Since then, multiplicity-free Kronecker products have been studied in \cite{BaOr,BWZ10,CG1,Man10}.
 In this paper we    prove that the classification list is indeed
complete for all~$n\in \mathbb{N}$ and hence confirm the conjecture,
that is, we have the following result:

\begin{thm}\label{thm:classification}
Let $\lambda, \mu$ be partitions of $n\in \mathbb{N}$.
Then the Kronecker product $[\lambda]\cdot [\mu]$ of the irreducible characters $[\lambda],[\mu]$ of $\S_n$
 is multiplicity-free if and only if
 the partitions $\lambda, \mu$ satisfy one of the following conditions (up to conjugation of one or both of the partitions):
\begin{enumerate}
 \item One of the partitions is $(n)$, and the other one is arbitrary;
 \item one of the partitions is $(n-1,1)$, and the other one is a fat hook
 (here, a fat hook is a partition with at most two different parts,
 i.e.\ it is of the form $(a^b,c^d)$, $a \geq c$);
 \item $n=2k+1$ and $\lambda=(k+1,k)=\mu$, or $n=2k$ and $\lambda=(k,k)=\mu$;
 \item $n=2k$, one of the partitions is $(k,k)$, and the other one is one of $(k+1,k-1),(n-3,3)$ or a hook;
 \item one of the partitions is a rectangle, and the other one is one of $(n-2,2), (n-2,1^2)$;
 \item the partition pair is one of the pairs
 $((3^3),(6,3))$,  $((3^3),(5,4))$, and  $((4^3),(6^2))$.
 \end{enumerate}
\end{thm}

We also provide the explicit combinatorial formulae for calculating  any multiplicity-free Kronecker product in   \cref{sec:mf-products}.
Using this we can then easily prove the following consequence of
\cref{thm:classification}:

\begin{thm}\label{thm:3classification}
Let $\la, \mu, \nu$ be partitions of $n\in \mathbb{N}$, all
different from $(n)$ and $(1^n)$.
Then the Kronecker product $[\lambda]\cdot [\mu]\cdot [\nu]$ of the irreducible characters $[\la],[\mu], [\nu]$ of $\S_n$ is not multiplicity-free.
\end{thm}

Assuming the classification of multiplicity-free Kronecker
products for a symmetric group $\S_n$,
with some further work the complete list
of multiplicity-free products involving {\em skew characters} of $\S_n$ is obtained;
we state this below.
Indeed, this will be an important tool in the inductive proof of Theorem~\ref{thm:classification}.
A proper skew diagram is one that is not the diagram of a partition up to rotation,
 the corresponding skew character has two distinct irreducible constituents  by \cite[Lemma 4.4]{BK};
 we shall refer to such a character  as a {\em proper skew character}.

\begin{thm}\label{thm:classification-skew}
 No product of two proper skew characters is multiplicity-free.  Now,
 let $\alpha$ be a partition of  $n$ and let $\chi$ denote a proper skew character of $\mathfrak{S}_n$.
 The product $\chi \cdot [\alpha]$ is multiplicity-free
if and only if one of the following holds
(up to conjugation of one of the characters):
\begin{enumerate}
\item $\chi$ is a multiplicity-free skew character,
and $[\alpha]$ is a linear
character;
\item
$n=ab$, $a,b\geq 2$: $\alpha=(a^b)$, $\chi=[(n,1)/ (1)] = [n]+[n-1,1]$;
 \item
$n=2k$, $k\geq 2$, $\alpha=(k,k)$, $\chi=[ (k+1,k)/ (1)] = [k+1,k-1]+[k,k]$.
 \end{enumerate}
 \end{thm}

The layout of the paper is as follows.
 In Section~\ref{sec:background}, we recall the results concerning Kronecker and Littlewood--Richardson coefficients which will be useful for the remainder of the paper, chief among these are Dvir recursion and Manivel's semigroup property.  We also explain our methodology and the intersection diagrams which will be essential in   the bulk of the paper.
  In Section~\ref{sec:mf-products}, we verify that the products on our list are indeed multiplicity-free and provide formulae for decomposing these inner tensor products;
  using some of these,
  we also show how to deduce \cref{thm:3classification} from
  \cref{thm:classification}.
  Sections~\ref{sec:warm-up} to \ref{sec:finale} are dedicated to proving the converse, namely that any product $[\lambda]\cdot [\mu]$ such that the pair $(\lambda,\mu)$ is  not on the  list in \cref{thm:classification}, contains multiplicities.
  Section~\ref{sec:warm-up} serves as a gentle introduction to the techniques which will be used in Sections~\ref{sec:rectangle}, \ref{sec:fatty}, and \ref{sec:finale}; here we consider tensor squares, products involving a hook, and products involving a 2-line partition.
  In Section~\ref{sec:if1then2}, we show that if \cref{thm:classification} has been proven to be true for all partitions of degree less than or equal to $n$, then \cref{thm:classification-skew} is also true for all skew-partitions  of degree less than or equal to $n$.
 We then begin our inductive proof of  \cref{thm:classification} in earnest.
  In Sections~\ref{sec:rectangle} and~\ref{sec:fatty}  we consider
 products involving either a character labelled by a  rectangle  or fat hook partition; such products are the most difficult to tackle using
 Dvir recursion and the semigroup property
 as one is more likely to reduce to a  multiplicity-free product.
  Finally, in  Section~\ref{sec:finale} we  prove that if
 \cref{thm:classification} and thus also \cref{thm:classification-skew}  are true for all partitions of degree  less than or equal to $n-1$, then they also hold  true for any product
 involving partitions of degree $n$.
 The hard work in earlier sections  has
  a surprising pay-off: the large reduction from arbitrary tensor products to those involving a fat hook is much simpler than one would expect.
   The main technique in the final section is to reduce to a product involving a fat hook or a rectangle and to appeal to the earlier sections.

 \section{Background  and useful results }\label{sec:background}

 \subsection{Symmetric group combinatorics}

We let $\mathfrak{S}_n$ denote the symmetric group on $n$ letters.
The combinatorics underlying the representation theory of the symmetric group  is based on  partitions.  A \emph{partition} $\lambda$ of $n$, denoted $\lambda\vdash n$, is defined to be a weakly decreasing sequence $\lambda=(\lambda_1,\lambda_2,\dots,\lambda_\ell)$ of non-negative integers such that the sum $|\lambda|=\lambda_1+\lambda_2+\dots +\lambda_\ell$ equals~$n$.  The {\em length} of a partition $\lambda\vdash n$ is the number of nonzero parts, we denote this by $\ell(\lambda)$.
The {\em width} of a partition $\lambda\vdash n$ is the size of the first part and is denoted $w(\lambda)=\lambda_1$.
The {\em depth} of a partition $\lambda\vdash n$ is
$n-\lambda_1$.

We identify a  partition, $\lambda$, with its associated   \emph{Young diagram}, that  is the set of nodes
\[ \left\{(i,j)\in\ZZ_{>0}^2\ \left|\ j\leq \lambda_i\right.\right\}.\]
A node $(i,\lambda_i)$  of $\lambda$ is \emph{removable} if it can be removed from the diagram of $\lambda$ to leave the  diagram of a partition, while a node not in the diagram of $\lambda$ is an \emph{addable} node of $\lambda$ if it can be added to the diagram of $\lambda$ to give the   diagram of a  partition.
The set of removable (respectively addable) nodes of a partition, $\lambda$, is denoted by $\Remm(\lambda)$ (respectively $\Add(\lambda)$).  Given $A\in \Remm(\lambda)$ (respectively $A\in \Add(\lambda)$) we let $\lambda_A$ (respectively $\lambda^A$) denote the partition obtained by removing the node $A$ from (respectively adding the node $A$ to)  the partition $\lambda$.

   Given $\lambda  \vdash n$, we define the \emph{conjugate} or \emph{transpose} partition, $\lambda^t$, to be equal to
 the partition obtained from $\lambda$ by reflecting its Young diagram through the $45^\circ$ diagonal.
   The Durfee length of $\la$ is the diagonal length of the Young diagram of $\lambda$, and thus gives the side lengths of the largest square which fits into the Young diagram of $\lambda$.

 Given  $\mu $ and $\lambda $  partitions such that $\mu _i \leq \lambda _i$ for all $i\geq 1$, we write   $\mu  \subseteq \lambda $.  If $\mu  \subseteq \lambda $, then the \emph{skew partition} or \emph{skew Young diagram}  (denoted $\lambda /\mu $) is simply the set
difference between the Young diagrams of $\lambda $ and $\mu $.
If $n= |\lambda|-|\mu|$ then we say that $\lambda /\mu $ is a skew partition of $n$.
We let $\gamma^{{\rm rot}}$ denote the diagram obtained by rotating the Young diagram of
$\gamma$ through ${180^\circ}$.
 We say that a skew diagram $\gamma$ is a \emph{proper skew diagram} if  neither
 $\gamma$ nor  $\gamma^{{\rm rot}}$  is the diagram of a partition.
We say that  a skew diagram $\lambda/\mu$ is {\em basic} if it does not contain empty rows or columns, in other words
$\mu_i<\la_i$, $\mu_i\leq \la_{i+1}$ for each $1\leq i \leq \ell(\la)$.

Over the complex numbers, the  irreducible  characters, $[\lambda]$, of $\S_n$ are indexed by the  partitions, $\lambda \vdash n$.
Given a skew partition $\lambda/ \mu$ of $n$, we have an associated  {\em skew character} $[\lambda/\mu]$ of $\mathfrak{S}_n$, see \cite[Section 2.4]{JK} for more details.
 For the corresponding definitions of Schur and skew Schur functions, see \cite{Sta99}.

\subsection{Multiplicity-free skew characters}\label{subsec:mf-skew}
 We recall the classification of multiplicity-free outer products of irreducible characters  and multiplicity-free skew characters of symmetric groups as in
 \cite{Stembridge01} and \cite{CG,TY}, respectively.

\begin{thm}[Multiplicity-free outer products of irreducible characters \cite{Stembridge01}\label{thm:mf-outer}]
A complete list of
multiplicity-free outer products of two irreducible
characters of symmetric groups is given as follows:
\begin{itemize}[leftmargin=0pt,itemindent=1.5em]
\item $[\text{rectangle}] \boxtimes [\text{rectangle}]$;
\item $[\text{rectangle}] \boxtimes [\text{near-rectangle}]$;
\item $[\text{2-line rectangle}]\boxtimes [\text{fat hook}]$;
\item $[\text{linear}]\boxtimes [\text{anything}]$.
\end{itemize}
Here, a linear partition (2-line rectangle) means a partition with one row or one column  (two rows or two columns).
A near-rectangle is obtained from a rectangle by adding a single row or column
to a rectangle, so a near-rectangle is a special fat hook.
\end{thm}

Generalising this result, Gutschwager \cite{CG}
classified the basic skew partitions
giving multiplicity-free skew characters;
this is closely connected to the classification of
multiplicity-free products of Schubert classes given by Thomas and Yong \cite{TY}.

Let $\rho/\sigma$ be a basic skew diagram; it may be connected or decompose into two or more pieces (where two adjacent pieces only meet in a point).
We define two paths along the rim of $\rho/\sigma$.
The \emph{inner path} starts in the lower left corner with an upward segment,
follows the shape of $\sigma$ and ends with a segment to the right in the upper right corner; here, by a \emph{segment} we mean the maximal pieces of the path where the direction doesn't change.
The \emph{outer path} starts in the lower left corner with a segment to the right,
follows the shape of $\rho$ and ends with an upward segment in the upper right corner.

We let $s_{in}$ and $s_{out}$ denote
the length of the shortest straight segment of the inner
path and of the outer path, respectively.
  Figure~\ref{Basic skew diagrams} depicts several basic skew diagrams, where
the partition $\rho$ is shown embedded in a rectangle, with
complementary partition $\tau$. In the middle picture, the
skew diagram $\rho/\sigma$ decomposes into two pieces $\delta'$ and $\delta''$.

  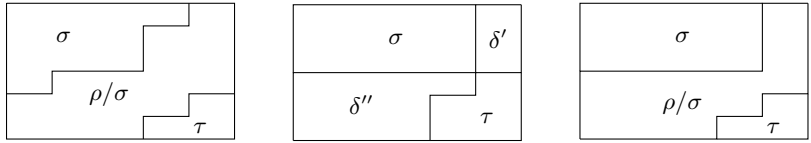
\begin{figure}[ht!]
  $$\scalefont{0.8}
       \begin{minipage}{120mm}
       \begin{tikzpicture}[scale=0.6]
      \draw
    (0,0)--(5,0)--(5,-3)--(0,-3)--(0,0);
      \draw
    (4,0)--(4,-0.5)--(3,-0.5)--(3,-1.5)--(1,-1.5)--(1,-2)--(0,-2);
       \draw
    (5,-2)--(4,-2)--(4,-2.5)--(3,-2.5)--(3,-3);
    \draw(1.25,-0.75) node {$\sigma$};
        \draw(2.25,-2) node {$\rho/\sigma$};
        \draw(4.25,-2.75) node {$\tau$};
      \end{tikzpicture}
      \qquad
 \begin{tikzpicture}[scale=0.6]
      \draw  (0,-1.5)--(0,0)--(4,0); \draw (5,-1.5)--(5,-3)--(3,-3);
      \draw   (5,0)--(4,0)--(4,-1.5)--(0,-1.5)--(0,-3);
       \draw   (5,0)--(5,-1.5)--(4,-1.5)--(4,-2)--(3,-2)--(3,-3)--(0,-3);
    \draw(2.25,-0.75) node {$\sigma $};
        \draw(1.5,-2.2) node {$\delta'' $};
        \draw(4.5,-0.75) node {$\delta' $};
        \draw(4.25,-2.5) node {$\tau $};
      \end{tikzpicture}
      \qquad
      \begin{tikzpicture}[scale=0.6]
      \draw  (0,-1.5)--(0,0)--(4,0); \draw (5,-2)--(5,-3)--(3,-3);
      \draw   (5,0)--(4,0)--(4,-1.5)--(0,-1.5)--(0,-3);
       \draw  (5,0)--(5,-2)--(4,-2)--(4,-2.5)--(3,-2.5)--(3,-3)--(0,-3);
    \draw(2.25,-0.75) node {$\sigma $};
        \draw(2.25,-2.2) node {$\rho/\sigma $};
        \draw(4.25,-2.75) node {$\tau $};
      \end{tikzpicture}
       \end{minipage}
  $$
  \caption{Basic skew diagrams}
  \label{Basic skew diagrams}
  \label{fomentation}
\end{figure}

Before we state the classification of the basic
skew diagrams labelling multiplicity-free skew characters,
we recall that the character associated to a skew diagram
is homogeneous if and only if the diagram is a partition diagram
up to a possible rotation by $180^\circ$;
 in which case it is already irreducible (see \cite{BK, vW_schur}).
Thus, the skew diagram is proper if and only if the corresponding
skew character is proper, i.e., it has at least two different constituents.

 \begin{thm}[Multiplicity-free outer products of skew characters]\cite{Stembridge01,CG,TY}
\label{thm:mf-basic-skew}
Let $D$ be a basic proper skew diagram.
Then the skew character $[D]$ is multiplicity free if and only if up to rotation of $D$ by $180^\circ$, we have $D=\rho/\sigma$ with $\sigma$ a rectangle, and additionally one of the following conditions holds:
\begin{enumerate}
\item $s_{in}=1$;
\item $s_{in}=2$, $|\Remm(\la)|=3$;
\item $s_{out}=1$, $|\Remm(\la)|=3$;
\item $|\Remm(\la)|=2$.
\end{enumerate}
\end{thm}

\begin{rmk}
We emphasise that Theorem~\ref{thm:mf-basic-skew} covers
all cases of multiplicity-free proper skew characters; in particular,
the skew character $[\rho/\sigma]$ is not multiplicity-free when
the diagram $\rho/\sigma$ decomposes
into more than two connected components, or if it decomposes into two components and one of them is a proper skew partition.

 Furthermore, note that in the cases (2)-(4) described above, the complementary partition $\tau$ to $\rho/\sigma$ (in the pictures above) is a (rotated) fat hook, as in \cref{Basic skew diagrams}.
\end{rmk}

Assuming that the two pictures to the right in Figure~\ref{Basic skew diagrams} are scaled such that the short segments on the outer path are of length~1,
the theorem tells us that these skew diagrams correspond to multiplicity-free characters, whereas the
skew diagram in the left picture certainly does not as both $\sigma$ and $\tau$ are not rectangular.

\subsection{The semigroup property for Kronecker coefficients}
  We now recall Manivel's semigroup property for Kronecker coefficients~\cite{Man11}.
 This  will be one of the two main tools used in   proving the classification theorem.

Let $\lambda,\mu, \nu$ be partitions of $n$.
We define the Kronecker coefficients $g(\lambda,\mu,\nu)$
to be the coefficients in the expansion
$$[\lambda]\cdot [\mu] = \sum_{\nu \vdash n} g(\lambda,\mu,\nu) \,  [\nu]\:.$$
In principle, they may be computed via the scalar product, in other words,
$$
g(\lambda,\mu,\nu) = \langle [\lambda]\cdot [\mu], [\nu]\rangle
= \frac{1}{n!}\sum_{g\in \S_n} [\la](g) [\mu](g) [\nu](g),
$$
from which it also shows that the Kronecker coefficients are symmetric in $\la, \mu,\nu$.
 For  $\la, \mu \vdash n$  we also define
$$
g(\la,\mu)= \max\{g(\lambda,\mu,\nu), \nu \vdash n\},
$$
so that the Kronecker product $[\la]\cdot[\mu]$ is multiplicity-free
if and only if $g(\la,\mu)=1$.

\begin{prop}\label{prop:monotonicity}
Let $\alpha,\beta,\gamma \vdash n_1$ and $\lambda,\mu,\nu\vdash n_2$.
If both $g(\alpha,\beta,\gamma)>0$ and $g(\lambda,\mu,\nu)>0$ then
$$g(\lambda+\alpha,\mu+\beta,\nu+\gamma)\geq \max\{
g(\lambda,\mu,\nu),g(\alpha,\beta,\gamma)\}\:.$$
In particular,
$$g(\la+\alpha,\mu+\beta)\geq
\max\{g(\lambda,\mu),g(\alpha,\beta)\}\:.$$
\end{prop}

\begin{rmk}
 We will often use this as a reduction procedure, in particular by removing rows and columns from two partitions under consideration.

As $g(\la,\mu)=g(\la^t,\mu)=g(\la^t,\mu^t)$, we can conjugate one or both of the partitions in the result above. This means that for the inequality, we do not have to take both partitions away from rows at the top but may take off one (or both) from
columns at the bottom.

 \end{rmk}

For a given partition $\nu$ and
$I \subset \{1,\ldots,\ell(\nu)\}$,
 we let
 $\nu_I= (\nu_{i_1},\nu_{i_2},\ldots )_{i_k\in I}
 $
 and
  $\nu^I= (\nu_{j_1},\nu_{j_2},\ldots )_{j_k\not\in I}
 $.

\begin{cor}\label{semigroup}
Let $\lambda,\mu$ be partitions  of $n$,
and suppose there  exist some $I$ and $J$ such that
 $|\lambda_I|=|\mu_J|$.
 Then
 $$g(\la,\mu)\geq \max\{g(\la_I,\mu_J),g(\la^I,\mu^J)\}\:.$$
 In particular, if  either $g(\lambda_I,\mu_J)>1$ or $g(\lambda^I,\mu^J)>1$, then it follows that $g(\lambda,\mu)>1$ also.
   \end{cor}

 \begin{proof}
 Suppose that $|\lambda_I|=|\mu_J|=k$.
 Let $\sigma \vdash k$
   and $\tau \vdash n-k$
   be partitions such that
$g(\lambda_I,\mu_J,\sigma)=g(\la_I,\mu_J),g(\lambda^I,\mu^J,\tau)=g(\la^I,\mu^J)$.
 We have that
 $$
g(\la,\mu)\geq g(\lambda,\mu,\sigma+\tau)
 \geq \max\{g(\lambda_I,\mu_J,\sigma),g(\lambda^I,\mu^J,\tau)	 \},
 $$ and so the result follows.
  \end{proof}

   \begin{notation}
 If $\lambda = \mu+\nu$, we say that
 $\lambda / \nu$ is an $(SG)$-\emph{removable} (or {\em semigroup removable}) skew partition.
See \cref{hjflsdahjlfakds} for an example of how one can use this procedure to prove that a     product  contains multiplicities.
   \end{notation}

\subsection{Dvir recursion}
 We now recall      Dvir's recursive   approach to calculating the value of a given Kronecker coefficient.
This is the second  main tool which we shall use in our proof of the classification theorem.

In the following, if $\la=(\la_1,\la_2,\dots,\la_\ell)$ is a partition,
we set $\hat\la = (\la_2, \la_3 , \ldots,\lambda_\ell)$.
\begin{thm}\label{theo:ThDvir1}  (\cite[1.6 and 2.4]{Dvir}, \cite[1.1 and 2.1(d)]{CM})
  \label{CDvir2}\label{dvirmaximal}
Let $\la$, $\mu$  be partitions of $n$.
Then
$$\max\{ \nu_1 \mid  g(\la,\mu,\nu)>0 \}
    = |\mu \cap \la|$$
Let $\mu,  \nu,  \la \vdash n$  and set
$\beta = \mu \cap \la$.
If  $ \nu_1 = |\mu \cap \la|$, then
$$ g(\la,\mu,\nu )
= \langle [\la/ \beta] \cdot [\mu / \beta] , [\hat{\nu}]\rangle \; .$$
\end{thm}

\begin{rmk*}\label{divr0}In the situation above, note that by the Littlewood-Richardson rule
and the bound on the width given above,
any constituent $[\alpha]$ of $[\la/ \beta] \cdot [\mu / \beta]$
has width at most $m$, so
$\al=(\al_1,\al_2,\ldots)$ can always be extended to
a partition $(m,\al)=(m,\al_1,\al_2,\ldots)$, giving a constituent in $[\la]\cdot[\mu]$.
\end{rmk*}

Since skew characters of $\S_n$ can  be decomposed
into  irreducible characters using the Littlewood-Richardson rule,
the following theorem provides a
recursive formula for the coefficients $g(\lambda,\mu,\nu)$.

\begin{thm}\label{theo:ThDvir2} \cite[2.3]{Dvir}.
Let $\la,\mu$ and $\nu=(\nu_1,\nu_2,\dots)$ be partitions of~$n$.
Define
$$Y(\nu) = \{ \eta=(\eta_1,\ldots) \vdash n \mid
  \eta_i \geq \nu_{i+1} \geq \eta_{i+1} \mbox{ for all } i \geq 1\}\:,$$
i.e., $Y(\nu)$ is the set of partitions obtained from $\hat\nu$ by adding a horizontal strip of size~$\nu_1$.
Then
$$ g(\la,\mu,\nu)
= \sum_{{\alpha \vdash \nu_1}\atop {\alpha \subseteq \la \cap \mu}}
 \langle [\la / \alpha] \cdot [\mu/\alpha] , [\hat{\nu}]\rangle
  - \sum_{{{\eta \in Y(\nu) \atop {\eta \neq \nu}} \atop
  \eta_1 \leq |\la \cap \mu|} }  g( \la,\mu,\eta) \; .$$
\end{thm}

This is crucial for the following result that will be useful later.
\begin{lem} \cite[Lemma 4.6]{BK}, \cite[Lemma 2.6]{BW14}
\label{lem:prepalmext}
Let $   \la, \mu  \vdash n$ be partitions not of the form $(n)$ or $(n-1,1)$ up to conjugation.
Set $\beta=\la \cap \mu \vdash m$.
Assume that $\lambda/\beta$ is a single row and that
$[\mu/\beta]$ is an irreducible character $[\alpha ]$,
with a partition $\al$.
Then
we have $g(\la,\mu,(m,\al))>0$.
Furthermore, we define the virtual character
\begin{equation}\label{chi}
\chi=
\sum_{A\in \Remm (\beta)} [   \la/ \beta_A] \cdot [ \mu / \beta_A]
-  \sum_{B \in \Add(\alpha)} \alpha^B
\:.
\end{equation}
Then if $\langle\chi,[\kappa]\rangle>0$, for $\kappa\vdash n-m+1$,
then $\nu=(m-1,\kappa)$ is a partition of $n$, and
$ g(\la,\mu,\nu) =  \langle\chi,[\kappa]\rangle$.
\end{lem}

\subsection{Terminology, notation, and methods }
 We shall frequently use the following terms:
 \begin{itemize}[leftmargin=0pt,itemindent=1.5em]
\item {\sf linear partition} (or {\sf linear character}) to mean a partition of the form $(k )$ or $(1^k)$ (or the corresponding character $[k]$ or $[1^k]$) for some $k\geq 1$;
\item  the {\sf natural character} to mean the character $[k-1,1]$ for some $k\geq 3$;
 \item {\sf 2-line partition}  to mean a partition, $\lambda$, such that $\ell(\lambda)=2$ or $w(\lambda)=2$;
\item  {\sf proper hook} to mean a  partition  of the form $(n-a,1^a)$ for $1\leq a< n-1$;
\item {\sf fat rectangle} to mean a rectangle which is not linear or a 2-line rectangle;
\item {\sf  proper fat hook} to mean a fat hook which is not equal to a rectangle, hook, or 2-line partition;
 \item  {\sf proper skew partition} to mean a skew partition, $\lambda$,  such that neither $\lambda$ nor  $\lambda^{\rm rot}$ is a proper partition.
\end{itemize}
Given  $\lambda, \mu \vdash n$, we shall refer to the {\sf diagram} for
this pair of partitions to be the diagram obtained by placing the partitions $\la$ and $\mu$
on top of one another so that one can see the intersection of these partitions
(usually denoted $\beta=\la \cap \mu$) and the
 set differences   $\mu / (\la \cap \mu)$ and $\la / (\la \cap \mu)$  explicitly, see for example  \cref{3aa3}.

\begin{eg}\label{hjflsdahjlfakds}
Suppose we wish to show that the tensor square   of the character $[ a^3 ]$ contains multiplicities.
 We    do this by considering the possible ways in which
 we can reduce our problem (using Dvir recursion or the semigroup property) to a problem for a pair of smaller partitions.
 We have that $\lambda=\mu = (3^3) + ( {(a-3)}^3)$ and
  $$g(\la, \mu)\geq  g((3^3),(3^3)) >1,$$
 by the semigroup property, as required (for example,   the coefficient $g((3^3),(3^3),(5,2,2))=2$).

   Alternatively, one can prove that   $[\la] \cdot [\mu]$ contains multiplicities (for $a \geq 3$) as follows.
   If $a \leq 6$ then the result can be verified by direct computation.
   For $a>6$,  we can conjugate and obtain
   $g((a^3),(a^3)) = g((3^a),(a^3)) \geq  g((3^{a-3}),((a-3)^3))  $  by Dvir recursion.   The result then follows by induction.

  \end{eg}
 \begin{eg}
 Suppose we wish to show that the product $[11,10^3,6,5,2^4,1] \cdot  [11,7^3,6,5^4,2,1]$ contains multiplicities.
 The diagram   is the rightmost depicted in  \cref{3aa3}.
 We have that $\gamma=\delta=(3^3)$ and so
   $$g(\la, \mu) \geq g(\gamma,\delta) = g((3^3),(3^3)) >1,$$
   using Dvir recursion, as required.  Alternatively, one can  use \cref{semigroup} to  remove all rows and columns which are common to both  $\la$ and $\mu$ to obtain the pair of partitions $\tla=(3^6)$ and $\tmu=(6^3)$.  The result then follows from the previous example.

    \end{eg}

 \begin{figure}[ht!]
 $$
\scalefont{0.6}
     \begin{minipage}{38mm}\begin{tikzpicture}[scale=0.5]
  \draw[very thick]
  (0,0)--(1.5,0)--(1.5,-1.5)--(0,-1.5)--
(0,0)
  ;
   \draw[] (1.5,0)--(2.5,0)
   (1.5,-1.5)--(2.5,-1.5)
   (4,0)--(5.5,0)--(5.5,-1.5)--(4,-1.5);
          \draw(4.75,-0.75) node {$\gamma$};
  \draw[dotted]
   (2.5,-1.5)--(4,-1.5)
   (2.5,0)--(4,0);
  \draw (0,-1.5)--(0,-2.5)  (1.50,-1.5)--(1.50,-2.5)
 (0,-4)--(0,-5.5)--(1.5,-5.5)--(1.5,-4);
 \draw[dotted]
   (1.5,-2.5)--(1.5,-4)
   (0,-2.5)--(0,-4);

          \draw(0.75,-4.75) node {$\delta$};

       \end{tikzpicture}\end{minipage}\quad
       \quad
     \begin{minipage}{38mm}\begin{tikzpicture}[scale=0.5]
  \draw[very thick]
  (0,0)--(5.5,0)--(5.5,-0.5)--(3.5,-0.5)
  --(3.5,-2)  --(3,-2)  --(3,-2.5)  --(2.5,-2.5)--(2.5,-3)--(1,-3)--(1,-5)--(0.5,-5)--(0.5,-5.5)--(0,-5.5)--(0,0)
  ;
    \draw (5,-0.5)--(5,-2)--(3.5,-2);
   \draw (2.5,-3)--(2.5,-4.5)--(1,-4.5);
          \draw(4.25,-1.25) node {$\gamma$};
          \draw(1.75,-3.75) node {$\delta$};
       \end{tikzpicture}\end{minipage}
$$
\caption{The diagram  for the pairs of partitions  $(\la,\mu)=((3^a),(a^3))$ and
$(\la,\mu)=( (11,7^3,6,5^4,2,1),(11,10^3,6,5,2^4,1))$ }
\label{3aa3}
 \end{figure}
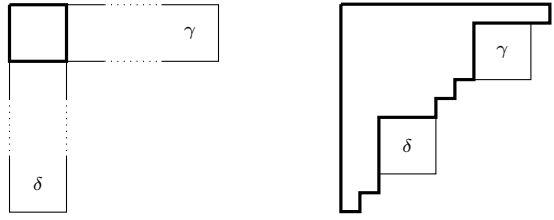

In \cref{sec:warm-up} we shall  first prove  \cref{thm:classification} in the case of  products $[\lambda]\cdot [\mu]$
 such that $\la=\mu$;  one of $\la$ or $\mu$ is a hook; and one of $\la$ or $\mu$ is a 2-line partition.
  This allows us to avoid the discussion of small critical cases in the later sections and
 serves as an introduction to the methods used.
In \cref{sec:if1then2},  we  shall then  show that if \cref{thm:classification} holds
 by induction on the degree, $n$, then so does \cref{thm:classification-skew}.

 We shall then begin our inductive proof of \cref{thm:classification}, assuming the validity of both \cref{thm:classification,thm:classification-skew} for partitions
 of strictly smaller degree.
 We shall then consider products with a rectangle, followed by products with a fat hook, and then finally arbitrary Kronecker products.  At each stage, our strategy will be to prove the result by
 using   the semigroup property and
 Dvir recursion to reduce the problem to
 $(i)$  a pair of partitions of strictly smaller degree and then appealing to our inductive proof, or
 $(ii)$  a pair of partitions  of degree $n$ which have already been considered.
 For example in  \cref{sec:rectangle} we shall reduce to pairs involving a  2-line or hook partition;   in \cref{sec:fatty} we shall reduce to pairs  involving a rectangle, or a  2-line, or hook partition.

\section{The  products on the list are multiplicity-free }
\label{sec:mf-products}

 Around the time of the classification conjecture, a number of formulae for
 special products and for constituents of small depth had already been
 obtained, notably by Jeff Remmel and his collaborators, as well as Jan Saxl and Ernesto Vallejo.
This   allowed Bessenrodt   to check, prior to making the conjecture, that all the products on the list were
indeed multiplicity-free.
  In this section we collect together the non-trivial formulae for the products on our list (up to conjugation).
 Some of these have appeared in the literature in the past years, and in these
 cases we refrain from giving proofs and provide references instead.

We start by recalling the products  with the character $[n-1,1]$, which
are easy to compute, and then the classification of such
 multiplicity-free products is not hard to deduce (see \cite{BK}).

\begin{lem}\cite[Lemma 4.1]{BK}\label{lem:LNat}
Let $n \geq 3$, and let $\mu$ be a partition of~$n$.
Let $r=|\Remm(\mu)|$.
Then
$$[\mu] \cdot [n-1,1] =
\left(\sum_{A\in \Remm(\la)} \sum_{B\in \Add(\mu_A)} [(\mu_A)^B]\right) - [\mu]=(r-1)[\mu] + \text{ other constituents }\:.$$
\end{lem}

Applying the formula above, the multiplicity-free products occurring below can easily be given explicitly in any concrete case.
We  set  $\chi_{(x>y)}=1$ if $x>y$, and~0 otherwise.
Similarly, we set $\chi_{(x>y>z)}=1$ if $x>y>z$, and~0 otherwise.  We extend this notation to other inequalities in the obvious fashion.

\begin{cor}\label{cor:mfNat}
Let $n \geq 3$, and let $\mu$ be a partition of~$n$.
Then
\begin{enumerate}
\item[{(i)}]
$[\mu] \cdot [n-1,1]$ is multiplicity-free if and only if $\mu$ is
a fat hook.
\item[{(ii)}]
$[\mu] \cdot [n-1]\uparrow^{\S_n}$ is multiplicity-free if and only if $\mu$ is a rectangle.
\end{enumerate}
In particular,  for $n>2$  we have that
$$[n-1,1]^2 = [n] + [n-1,1]+\chi_{(n>3)}[n-2,2]+ [n-2,1^2]\:.$$

\end{cor}

  The classification of Kronecker squares was
 also verified in the course of  making the
 classification conjecture in  1999
   using the formulae  stated below (which follow as special cases from \cite{RW,Rosas}) and
 work of  Saxl, Zisser and Vallejo \cite{Saxl,Z,V}.
  In the next section
we  will provide a short  proof  that   the square products in \cref{cor:mfNat} and \cref{prop:special-2part-prod}
constitute a complete list of nontrivial multiplicity-free square products (up to conjugation)
using the semigroup property.

\begin{prop} \label{prop:special-2part-prod}Let $k\in \mathbb{N}$.
 \begin{enumerate}
 \item[{(i)}] For $n=2k+1$, we have
$$[k+1,k]^2 = \sum_{\la\vdash 2k+1\atop {\ell(\lambda) \leq 4}} [\lambda]\:.$$
\item[{(ii)}]
Let $n=2k$, we let  $E(n)$ and $O(n)$ denote
the sets of partitions of $n$ into only even parts and only odd parts,
respectively, then
$$[k,k]^2 =\sum_{{\lambda \in E(n)} \atop {\ell(\lambda) \leq 4}} [\lambda] +
\sum_{{\lambda \in O(n)} \atop {\ell(\lambda) = 4}} [\lambda]\:.$$
\item[{(iii)}] Let $n=2k$, we have that
$$[k,k]\cdot [k+1,k-1]
=\sum_{{\la \vdash n, \lambda \not\in E(n)}\atop {\ell(\lambda) < 4}} [\lambda] +
\sum_{{\la \vdash n, \lambda \not\in O(n)\cup E(n)}\atop {\ell(\lambda) = 4}} [\lambda]
\:.$$
\end{enumerate}
\end{prop}
\begin{proof}
The
decompositions $(i)$, $(ii)$, $(iii)$ have  since appeared explicitly in
\cite{BWZ10,GWXZ,Man10},
so we refrain from elaborating on the proof.
\end{proof}

\begin{rmk}Note that the products in $(ii)$ and $(iii)$ are in the following sense
complementary; we have
$$[k,k] \cdot  ([k,k]+[k+1,k-1])
=\sum_{{\la \vdash 2k}\atop{\ell(\lambda) \leq 4}} [\lambda] \:.$$
\end{rmk}

The decomposition of the products of characters involving a   2-line partition and a hook partition
has been determined explicitly by Remmel \cite{R92} and Rosas \cite{Rosas}.
The formulae there are quite involved, but can be applied in our special case
to show
\begin{prop}\label{prop:kk-hook}
Let $n=2k$, and let $\mu \vdash n$ be a hook.
 Then $[k,k]\cdot [\mu]$ is multiplicity-free.
\end{prop}
\begin{proof}
By the formulae given in \cite{R92} or  \cite{Rosas} it is clear that no constituent to a partition of Durfee length~3 can appear, but only hooks and double-hooks.

From the formula in \cite[Theorem 2.2]{R92} for the multiplicity of hook constituents in the product, it is immediate that each of these can appear at most once (note that in Theorem 2.2(ii)(c) the second term can't appear for $(m,n)=(k,k)$).

For a double-hook $\nu$, we might use either \cite[Theorem 2.2]{R92} or \cite[Theorem 4]{Rosas} to deduce that  $g( (k,k) ,\mu,\nu)=0$ or 1.
Let $\mu=(n-b,1^{b})$ be our hook, and let $\nu$ be a double-hook that is not a hook, written as $\nu=(a_1,a_2,2^{b_2},1^{b_1})$ (here $a_1,a_2>0$,  $b_1,b_2\ge 0$); we may  assume
(by conjugation if necessary) that $a_1-a_2 \le b_1$.
 We recall the formula from \cite[Theorem 4]{Rosas}:
$$
g((k,k),\mu,\nu)
= X_1+X_2+X_3-X_4$$
where
\begin{equation}\label{anfmdsanfsd}
\begin{array}{llllllllll}
X_1 =
\chi_{(a_2 \le k-b_2-1 \le a_1)}\, \chi_{(b_1+2b_2 < b < b_1+2b_2+3)}, &
X_2  =
\chi_{(a_2 \le k-b_2 \le a_1)} \, \chi_{(b_1+2b_2 \le b \le b_1+2b_2+3)},
\\
X_3 =
\chi_{(a_2 \le k-b_2+1 \le a_1)} \, \chi_{(b_1+2b_2 < b < b_1+2b_2+3)},
&
X_4 =
\chi_{(a_2+b_2+b_1 =k)} \, \chi_{(b_1+2b_2+1 \le b \le b_1+2b_2+2)}.
\end{array}
\end{equation}

First we consider the case where $X_1=1=X_2$ and $X_3=0$. Then $a_1=k-b_2$, so $a_1+b_2=k=a_2+b_1+b_2$, and hence $X_4=1$.

If $X_1=0$ and $X_2=1=X_3$, then $a_2=k-b_2$, hence
$a_2+b_2=k=a_1+b_1+b_2\ge a_2+b_1+b_2$, so we must have $b_1=0$
and then $a_1=a_2$. But then $X_3=0$, a contradiction.

If $X_1=1=X_3$, then we also have $X_2=1$.
In this case, we must have $a_2\le k-b_2-1$ and $k-b_2+1\le a_1$.
By our assumption, $a_1-a_2\le b_1$, hence
$$k\le a_1+b_2-1\le a_2+b_1+b_2-1\:.$$
Since $2k=a_1+a_2+b_1+2b_2$, we obtain
$a_1+b_2+1 \le k$, and thus we have the contradiction
$$k-b_2+1\le a_1\le k-b_2-1\:.$$
Hence the multiplicity
$g((k,k),\mu,\nu)= X_1+X_2+X_3-X_4$
is always at most~1.
\end{proof}

\medskip

 We now provide explicit  formulae for the
 Kronecker products of small depth listed in  \cref{thm:classification}.
 Also these products were checked in the course of making the
 classification conjecture in 1999 using    \cite{Dvir,V}.
We use this opportunity to correct a small mistake in
the statement of the formula for the decomposition
given in \cite[Corollary 4.6]{BaOr}; this correction
is provided in case $(i)$ below.

\begin{prop}
\label{Products with a partition of small depth}
The  Kronecker products   involving a partition of  depth 2 or 3 listed in
 \cref{thm:classification} can be calculated as follows.
 Here we take the convention that if $\la$ is not a partition,
 then $[\la]$ is zero.
 \begin{itemize}[leftmargin=0pt,itemindent=1.5em]
 \item[$(i)$]
Let $n=ab\geq 6$,
 $\la=(a^b)$, with $a,b>1$.
Then the decomposition of the product $[n-2,2]\cdot [a^b]$ is as follows,
\begin{align*}
 & [a^b]+\chi_{(a>2)}[a^{b-1},a-1,1] +[a^{b-2},(a-1)^2,1^2] +\chi_{(b>3)}[(a+1)^2,a^{b-4},(a-1)^2]
\\
& +\chi_{(b>2)} [a+1,a^{b-2},a-1]+\chi_{(b>2)}[a+1,a^{b-3},(a-1)^2,1]+[a+2,a^{b-2},a-2]
\\
&+\chi_{(a>2)}[a+1,a^{b-2},a-2,1]+\chi_{(a>3)}[a^{b-1},a-2,2] \:.
\end{align*}
\item[$(ii)$] Let $n=ab$ and $\la=(a^b)$, with $a\geq b>1$.
Then the decomposition of the product $[n-2,1^2]\cdot [a^b]$ is as follows,
\begin{align*}
&\chi_{(b>2)}[a+2,a^{b-3},(a-1)^2]+[a+1,a^{b-2},a-1]+[a+1,a^{b-2},a-2,1]
+[a^{b-2},(a-1)^2,2]\\
&
+\chi_{(b>2)}[a+1,a^{b-3},(a-1)^2,1]+\chi_{(b>2)}[(a+1)^2,a^{b-3},a-2]
 + [a^{b-1},a-2,1^2]+[a^{b-1},a-1,1]
\:.
\end{align*}

\item[$(iii)$]  Let $n= 2k > 16$.
Then the decomposition of the product $[n-3,3]\cdot [k,k]$ is as follows,
\begin{align*}
 &[k+1,k-1]+[k+1,k-2,1]+[k,k-1,1] + [k,k-2,1^2]+[k,k-2,2]+[k,k-3,3]+\\
& [k-1,k-1,2]+[k-1,k-2,2,1] +[k+3,k-3]+[k+2,k-3,1]+[k+1,k-3,2] \:.
\end{align*}
For $6\leq n\leq 16$ the remaining  multiplicity-free products of the form
$[n-3,3]\cdot [\lambda]$   are   precisely those with $\lambda \in \{
(4,2), (4,1^2), (4,3), (3^3)\}$ (up to conjugation),
for the corresponding $n$.
 \end{itemize}
 \end{prop}
\begin{proof}
Case $(i)$ can be proved   directly
using Dvir's recursion and the Littlewood-Richardson rule.
Case $(ii)$ may be proved easily by applying Dvir's recursion formula,
computing the multiplicity of constituents $[\la]$ using $g((k,k),(n-2,1^2),\la)=g((k,k),\la,(n-2,1^2))$ and the
known formula for $g((k,k),\la,(n-1,1))$.
Case  $(iii)$ can be proved easily using \cite{V};
as it has since appeared in  \cite[Theorem 4.8]{BaOr}
we refrain from elaborating on the proof.
\end{proof}

Finally we show in this section that our main result implies that
no product of three non-linear irreducible characters of the symmetric groups
is multiplicity-free; hence at the end of this article also \cref{thm:3classification} is confirmed.
\begin{prop}\label{prop:3from2}
Assume that \cref{thm:classification} is true.
Then also \cref{thm:3classification} holds.
\end{prop}

\proof
Let $\la,\mu,\nu$ be partitions of $n$, all different from $(n)$ and $(1^n)$.
First we show that a product involving a square cannot be multiplicity-free.
Since no product of two non-linear irreducible characters of $\S_n$
is irreducible by \cite{BK} or \cite{Z}, we have
$$
\langle [\la][\la][\mu],[\mu] \rangle
= \langle [\la][\mu],[\la][\mu] \rangle
> 1,
$$
so we are done in this case.

Hence $\la,\mu,\nu$ have to be three different partitions with pairwise
multiplicity-free products.
To avoid discussion of small cases, for $n\le 12$ the assertion of \cref{thm:3classification} is checked by computer,
so we assume now $n\ge 13$.
By conjugating if necessary, we only have to consider the following triples
$(\la,\mu,\nu)$:
$((n-1,1),(k+1,k-1),(k,k))$,
$((n-1,1),(n-3,3),(k,k))$,
$((n-1,1),(n-a,1^a),(k,k))$ with $2< a \le \frac{n-1}2$,
$((n-1,1),(n-2,2),(a^b))$,
$((n-1,1),(n-2,1^2),(a^b))$.
In  all of these cases,
$[\la][\mu]$ is easily seen to have
a constituent $[\al]$ where $\al$ is not a fat hook by \cref{lem:LNat}.
It immediately follows from \cref{thm:classification}  that $[\al][\nu]$, and hence
$[\la][\mu][\nu]$, cannot be multiplicity-free.
\qed

\section{Squares, and products with a hook or with a 2-line partition}\label{sec:warm-up}
As a warm-up to the later sections, we shall now give a self-contained proof of the classification theorem for products $[\lambda] \cdot [\mu]$ involving a hook or 2-line partition or for which $\lambda=\mu$.

\subsection{Squares}\label{squares}

 We first consider products of the form $[\lambda]\cdot [\lambda]$.  We use the semigroup property to give a simple proof of the converse to
 \cref{prop:special-2part-prod} (that any product not on the list contains multiplicities).
 \begin{prop}  \label{selftensor}
 Let $\lambda$ be a partition of~$n$.
Then $[\la]\cdot [\la]$ is multiplicity-free
if and only if $\la$ or its conjugate is one of the following
$$(n), \; (n-1,1), \;  (\left\lceil\frac n2 \right\rceil , \left\lfloor  \frac n2 \right\rfloor)\:.$$
  \end{prop}

 \begin{proof}
 By  \cref{prop:special-2part-prod},
  it will suffice to show that any product not of the above form contains multiplicities.
  Suppose that $\lambda$ is a 2-line partition not of the above form.  Up to conjugation, we can assume that
$\lambda=(\lambda_1,\lambda_2)$ such that
$(i)$ $\lambda_2>1$ and $(ii)$ $\lambda_1-\lambda_2\geq 2$.
The smallest partition satisfying these properties is $ (4,2)$; this shall be our \emph{seed} and we shall \emph{grow} all other 2-line cases from this one.  Given any $\lambda$ of the above form, we have that
$$
\lambda = (4,2) + (\lambda_1-4,\lambda_2-2),
$$
where the latter term on the right-hand side  is a partition because of $(ii)$.
By \cref{prop:monotonicity}, we have that
$$
g(\lambda,\lambda) \geq g((4,2),(4,2))=2
$$
(for example,  $g((4,2),(4,2),(3,2,1))=2$) and so the product $[\lambda]^2$ is contains multiplicities.

It remains to consider the case in which  $\lambda$ is a partition with  $\ell(\lambda), w(\lambda) \geq 3$.
Set $I=\{1,2,3\}$. Then
$$g(\la,\la) \geq g(\la_I ,\la_I ) =
g((\la_I)^t ,(\la_I)^t) \geq g (((\la_I)^t)_I ,((\la_I)^t)_I) \:.
$$
Now $\tla= ((\la_I)^t)_I= \la^t \cap (3^3)$ is a partition with
$\ell(\tla), w(\tla) = 3$.
Up  to conjugacy we only need to consider
$(3,1^3),(3,2,1), (3,3,1), (3,3,2), (3^3)$, with $g(\tla, \tla)$ equal to 2, 5, 3, 3, 2, respectively.
\end{proof}

We will later use some more detailed information on squares.
By work of Saxl \cite{Saxl}
, Zisser \cite{Z}
, Vallejo \cite{V, V14}
we have the following result on constituents in squares.
We refer to  \cite[Section 2.3.17]{JK} for the definition of a  hook in a diagram.
\begin{prop}\label{prop:squares}
Let $\lambda \vdash n$, $\lambda \neq (n), (1^n)$.
Let
$h_k=\#\{k\text{-hooks in } \lambda\} $ for $
k=1,2,3$ and
$h_{21}=\#\{\text{non-linear 3-hooks $H$ in } \lambda   \}$. Then
$$\begin{array}{rl}
[\lambda]^2 & =  [n]+a_1[n-1,1] +a_2[n-2,2] +b_2[n-2,1^2]\\[5pt]
&
\: +a_3[n-3,3] +b_3[n-3,1^3]+c_3[n-3,2,1]  + \text{constituents of greater depth}
\end{array}$$
with
 $a_1=  h_1-1, \; b_2=(h_1-1)^2$,
$a_2=  h_2+h_1(h_1-2), \text{for } n\geq 4$,\\
$a_3=  h_1(h_1-1)(h_1-3)+h_2(2h_1-3)+h_3$,
$\text{for } n \geq 6$,\\[5pt]
$b_3=  h_1(h_1-1)(h_1-3)+(h_1-1)(h_2+1)+h_{21}$,
$ \text{for } n \geq 4$, \\[5pt]
\hbox{$c_3= 2h_1(h_1-1)(h_1-3)+h_2(3h_1-4)+h_1+h_{21}$},
$ \text{for } n \geq 5$.\\[1ex]
 In particular, for $n \ge 4$ we always have
$a_2 >0$.
\end{prop}
 \begin{rmk} \cref{prop:squares} was used to verify the classification of multiplicity-free  square products in \cref{selftensor} prior to 1999.
Applying this result to a partition $\la$ for which $[\la]^2$ is
multiplicity-free,  immediately
yields that $\la$ is a rectangle
or $(n-1,1)$ or $(k+1,k)$ (or conjugate).
Squares of rectangles can then be dealt with using Dvir recursion.
\end{rmk}

 \subsection{Hook partitions}\label{sectionhook!}

 We shall now cover the case of products  $[\lambda]\cdot [\mu]$ such that one of $\lambda$ or $\mu$ is a hook, different from $(n)$, $(n-1,1)$ and their conjugates,  and the other is an arbitrary partition.

\begin{prop}
Let $n \geq 5$, and let  $\mu=(n-a,1^a)$ with $2\leq a\leq n-3$.
Let $\lambda\vdash n$, $\la$ not equal to $(n)$ or $ (n-1,1)$ up to conjugation.
 If
  $[\mu] \cdot [\lambda]$ is   multiplicity-free
   then
   (up to conjugation of $\la,\mu$) we have that
   $\lambda $ is equal to     $(k,k)$ for $n=2k$,
   or $a=2$ and $\lambda$ is a rectangle.
\end{prop}

\begin{proof}
From our computational data, we know that the result holds for all $n\le 20$, so we may assume that $n\ge 21$.
We will proceed by induction, so we assume that the result holds
for products with hooks of size smaller than~$n$.
Furthermore, by conjugating if necessary,
we may (and will) assume that for both $\mu=(n-a,1^a)$ and $\la$ the length is at most as large as the width, so $a\le \frac{n-1}2$.
We have to show  $g(\la,\mu)>1$ for any $\la$  different from $(n),(n-1,1)$ and their conjugates, and with $(\mu,\la)$ not on the classification list above.

We start with the case in which $\ell(\la)=2$, so
$\lambda=(n-b,b)$ where by our assumptions  $n-b>b\geq 2$.
We remove the third column of $\la$, of height $h\in \{1,2\}$, to obtain
a partition $\tla$, and we set $\tmu=(n-a-h,1^a)$.
By our assumptions, $(\tla,\tmu)$ is a pair not on our classification list for $n-h$, hence by  \cref{semigroup} we conclude
$$g(\la,\mu)\ge g(\tla,\tmu) >1$$
and we are done in this case.

We now assume $\ell(\la)\geq 3$; since $(\la,\mu)$ is not on our classification
list, $\la$ is not a rectangle when $a=2$.
We remove the third row $\la_3$ from $\la$ to obtain $\tla$.
As $\la_3 \leq n/3$ and $a\le \frac{n-1}2$,
we have
$$n-a-\la_3 \ge n-a-\frac n3 \ge \frac 16 (n+3)\ge 4 \:.$$
Hence
$\tmu=(n-a-\la_3,1^a)$ still satisfies the conditions of the
proposition we want to prove.
Hence by induction and \cref{semigroup}, we have that
$ g(\la,\mu) \geq g(\tilde\mu,\tilde\la)>1$
unless
$ \tla=(m,m)$ for $m=\frac{n-\la_3}2$, or $a=2$ and $\tla$
is an arbitrary rectangle.

Indeed, both cases can only occur when $\la= (m,m,r)$, with $r\ge 1$,
and if $a=2$, we also have $m>r$ (note that $m\geq 7$ as $n\ge 21$).
In which case, we let $\tla$ be obtained by removing the second column from~$\la$, of height $h\in \{1,2,3\}$, and we set $\tmu=(n-a-h,1^a)$.
Then $\ell(\tla)=3$, and $\la$ is not a rectangle in case $a=2$, so
hence again we have
by induction and \cref{semigroup} that
$ g(\la,\mu) \geq g(\tilde\mu,\tilde\la)>1$.
\end{proof}

\subsection{2-line partitions}\label{section2parter}

We now consider products in which  one factor is labelled by a partition $\mu$ with two rows or two columns. Conjugating if necessary,
we may assume that $\mu$ has two rows.

\begin{lem} \label{short_second_row}
  Let $n\in \mathbb{N}$ and  $\lambda$ be a partition of  $n$,
  not equal to $(n)$ or $(n-1,1)$ up to conjugation.
  \begin{enumerate}
   \item
  Let $n\ge 4$. If the product
  $[n-2,2] \cdot [\lambda]$ is multiplicity-free, then
  $\lambda$ is a rectangle or $\la$ is equal to $(3,2)$ (up to conjugation).
  \item
  Let $n\ge 6$.  If the  product $[n-3,3] \cdot [\lambda]$
  is multiplicity-free, then $\lambda=(k,k)$ or
  $(4,1^2)$, $(4,2)$, $(4,3)$, $(3^3)$ (up to conjugation).
 \end{enumerate}
\end{lem}

\begin{proof}
 As the smaller cases hold by computer calculations,
 we may assume that $n > 17$.
 In both cases of the lemma we proceed by induction.
By \cref{sectionhook!}, we may assume that $\lambda$ is not a hook.
 Conjugating if necessary, we assume that $w(\lambda)\geq \ell(\lambda)$;
 thus $w(\lambda)\geq 5$  by our assumption on $n$.

We have $\mu=(\mu_1,\mu_2)$ with $\mu_2\in \{2,3\}$,
and we take a partition  $\la$ (satisfying the assumptions above)
that is not a rectangle when $\mu_2=2$,
and not $(k,k)$ when $\mu_2=3$.
 We remove the fourth column of $\la$, of height $h$ say, to obtain $\tla$,
and we set $\tmu=(\mu_1-h,\mu_2)$.
By our assumptions, $\tla$ is not a hook,
nor is  it  a rectangle, nor $(3,2)$ or its conjugate
when $\mu_2=2$,
and not of the form $(\tilde k,\tilde k)$
or one of the exceptional small partitions or one of their
conjugates
when $\mu_2=3$..
Hence by induction $g(\tla,\tmu)>1$, and we are done by
\cref{prop:monotonicity}.\end{proof}

\begin{lem} \label{two_equal_parts}
  Let $n=2k\geq 2$.
  Let $\lambda\vdash n$.   If  the product $[k,k] \cdot [\lambda]$ is multiplicity-free, then $\lambda$  is a hook, $(n-2,2), (n-3,3),(k,k),(k+1,k-1)$
  or $(4^3)$ (up to conjugation).
\end{lem}

\begin{proof}
As the smaller cases hold by computer calculations,
 we may assume that $n \ge 26$.  We set $\mu=(k,k)$.
We now assume that $\lambda$ is not one of the partitions listed
above that are already known to give a multiplicity-free product with $[k,k]$.
We shall again proceed by induction.
 We assume that $w(\lambda)\geq \ell(\lambda)$ and therefore (by our assumption on $n$) we conclude that $w(\lambda)\ge 6$.
If the fifth or sixth column is of even height (and equal to $2h$ say), remove this column from $\lambda$ to obtain $\tla$.
Otherwise, both the fifth and sixth columns are of odd height (and their sum is equal to $2h$ say); then remove both columns from $\lambda$ to obtain $\tla$
In both cases set $\tmu=(k-h,k-h)$.
We then have a pair of partitions $(\tmu,\tla)$ such that $g(\tmu,\tla)>1$ by induction (keeping in mind that $n\geq 26$),
 and hence, by Proposition~\ref{prop:monotonicity},
  $g(\mu,\la)>1$.
\end{proof}

\begin{prop}\label{prop:res2part}
  For $\mu$ a 2-line partition, a product $[\mu]\cdot[\lambda]$
  is multiplicity-free if and only if   the product is on the
  classification list of Theorem~\ref{thm:classification}.
\end{prop}

\begin{proof}
We may assume that $n \geq 26$ by our computational data, and we proceed again by induction.
By \cref{sec:mf-products} it is enough to prove that any product not on the list contains multiplicities.
By \cref{sectionhook!} we can assume that $\lambda$ is not a hook.
As before, we may (and will)
assume that $w(\lambda)\geq \ell(\lambda)$;
note that then $w(\la)\ge 6$ by our assumption on~$n$.
By  \cref{short_second_row,two_equal_parts} we may assume for $\mu=(\mu_1,\mu_2)$ that   $\mu_1 > \mu_2 > 3$; we can also
assume that $\la$ is not of the form dealt with in these lemmas.
By \cref{selftensor}, we may also assume that $\la\ne \mu$.

We first suppose that $\ell(\lambda)=2$.
For $\alpha=(\alpha_1,\alpha_2)$ with $\alpha_2>3$,  we define $ \alpha'$ by
$ \alpha'+(1^2)=\alpha$ if $\alpha_2\geq 5$,
and by $ \alpha'+(2)=\alpha $ otherwise.   With this notation in place,
$g(\la,\mu) \geq g(\la',\mu')>1$ (given our assumption on $n$).

We now assume $\ell(\lambda)\ge 3$.  Remove the fifth column of $\lambda$, of height $h$ say, to obtain $\tla$.
 We have three cases to consider:
$(i)$ $h< \mu_1-\mu_2$,
$(ii)$ $h= \mu_1 -\mu_2 +2h'$,
or $(iii)$ 	$h= \mu_1 -\mu_2 +2h' +1$ for $h'\geq 0$.
Corresponding to these cases
we write $(i)$ $\mu=\tmu +(h)$,
$(ii)$  $\mu=\tmu +(\mu_1-\mu_2 +h',h')$,
or $(iii)$ $\mu=\tmu +(\mu_1-\mu_2 +h',h'+1)$.
We hence obtain a pair of partitions $(\tla,\tmu)$ such that
by induction $g(\tla,\tmu)>1$ (keeping in mind that $n\geq 26$),
and hence we are again done by \cref{prop:monotonicity}.
\end{proof}

\section{Multiplicity-free products of skew characters}\label{sec:if1then2}

It is the aim of this section to show that if
Theorem~\ref{thm:classification}  holds for a fixed $n\in \NN$, then so does
Theorem~\ref{thm:classification-skew}.
In the final proof of Theorem~\ref{thm:classification} (and hence also
of Theorem~\ref{thm:classification-skew}) by induction we may
thus always assume that both
Theorem~\ref{thm:classification}
and \cref{thm:classification-skew} hold  for all symmetric groups of degree strictly less than $n$.

First we require some preparatory results on how (multiplicity-free) skew characters decompose into simple constituents.
The following observation by Gutschwager
will be a very useful tool later on.

\begin{lem} \cite{CGposet}
\label{skew-neighbours}
Any proper skew character of $\S_n$ has two neighbouring constituents, i.e.,
constituents $[\la],[\mu]$ such that $|\la\cap \mu|=n-1$.
\end{lem}

We now describe multiplicity-free proper skew characters with large maximal constituents (in  the lexicographic ordering of the partition labels).

\begin{lem}\label{lm:special-skew-mf}
Let $\chi$ be a  multiplicity-free proper skew character of $\S_n$.
\begin{enumerate}
\item
If $\chi$ has a constituent $[n]$, then
 $\chi= [n-k]\boxtimes [k]= \displaystyle \sum_{i=0}^k [n-k+i,k-i]$,
for some $0\leq k\leq n/2$.
\item
If $\chi$ has maximal constituent $[n-1,1]$, then we have one of the following:
\begin{itemize}[leftmargin=0pt,itemindent=1.5em]
\item $k\leq (n-2)/2$ and $$\chi= [n-k-1,1]\boxtimes [k]=\sum_{i=0}^k [n-1-i,1+i]+\sum_{i=0}^{k-1} [n-2-i,1+i,1];$$
\item  $a>b$, $m=\max(\lfloor (2b-a)/2\rfloor,0)$ and
$$\chi = [(a,b)/(b-1)]=\sum_{i=m}^{b-1} [a-b+1+i,b-i];$$
\end{itemize}
\item
If $\chi$ has maximal constituent $[n-2,2]$, and also  has
 $[n-2,1^2]$ appearing as a constituent, then we have one of the following:
\begin{itemize}[leftmargin=0pt,itemindent=1.5em]
\item
$\chi=[n-3,1]\boxtimes[1^2]=[n-2,2]+[n-2,1^2]+[n-3,1^3]$;
\item
$\chi=[(n-2,s,1)/(s-1)]=[((n-2)^2,s)/(n-3,s-1)]$,$1<s<n-2$;
\item
$\chi=[((n-2)^2,1)/(n-3)]$, $n>3$.
\end{itemize}
\item
If $\chi$ has maximal constituent $[n-2,2]$, and contains
also $[n-3,3]$,  but not $[n-2,1^2]$, then we have one of:
\begin{itemize}[leftmargin=0pt,itemindent=1.5em]
\item
$\chi=[n-k-2,2]\boxtimes [k]=[n-2,2]+[n-3,3]+... + [n-3,2,1]+...+[n-4,2^2]$;
\item
$\chi=[(n-2,s+2)/ (s)]=[n-2,2]+[n-3,3]+...
 =\displaystyle\sum_{a=2}^{\lfloor n/2\rfloor} [n-a,a]$.
 \end{itemize}
\end{enumerate}
\end{lem}

\begin{proof}
First we note that the diagram of the proper skew character $\chi=[\la/\mu]$
can have at most two components (as no outer product of three characters is multiplicity-free) by \cref{subsec:mf-skew}.
Our tactic for the proof will be to examine the maximal constituents of skew characters using the Littlewood--Richardson rule \cite{JAM} and
then considering which characters also satisfy the conditions of
\cref{subsec:mf-skew}.  We shall write $\chi = [\lambda /\mu]$ for the duration of the proof.

(1) Assuming that  $\chi$ contains $[n]$,
we immediately deduce (from the Littlewood--Richardson rule) that
the skew diagram $\la/\mu$ consists
solely of disconnected single rows.
By \cref{thm:mf-outer}, there are at most two such disconnected rows in $\chi$.
Hence the skew character is of the form
$\chi=[n-k]\boxtimes [k] = \sum_{i=0}^k [n-k+i,k-i]$, for some $k\leq n/2$.

(2) Assume now that the maximal constituent
of the skew character $\chi$ is $[n-1,1]$.
Again by the Littlewood--Richardson rule, we deduce that
 the skew diagram $\la/\mu$ then
must have one column  of length 2, and all others are of length 1.
As the skew diagram has at most two components,
this leaves only a few possibilities for the skew character, and
we can only have the two types
listed in (2), by~\cref{subsec:mf-skew}.

To (3) and (4).
We assume that $\chi$ has maximal constituent $[n-2,2]$.
Note that the skew diagram $\la/\mu$ then
must have two columns of length 2, and all others are of length~1,
and as before, it has at most two components;
if it is disconnected, then both components are of partition shape
(up to rotation).

We first consider  case $(3)$,  in which
$\chi$ contains   $[n-2,1^2]$ as a constituent.
Then, if the diagram has two components,
both components have a column of length~2 (by the Littlewood--Richardson rule).
 By Theorem~\ref{thm:mf-outer},
the only possibility is then that
$\chi=[n-3,1]\boxtimes[1^2]=[n-2,2]+[n-2,1^2]+[n-3,1^3]$.

Now assume that   $\la/\mu$ is connected and
 $[n-2,1^2]$ appears  as a constituent of $\chi$. In which case
the diagram of $(2^2)$ does not appear as a subdiagram  of $\lambda/\mu$ (by the Littlewood--Richardson rule).
 Now, the multiplicity-free
condition leaves only the possibilities
$\la/\mu=(r,s,1)/(s-1)$,
with $r\geq s>1$, or
$\la/\mu=(r^2,s)/(r-1,s-1)$,
with $r >s>1$, and
for $r>s$ the corresponding skew characters are equal
since the diagrams only differ by a rotation.
Since $|\la/\mu|=n$, we have $r=n-2$.

For case (4), we now assume that $\chi$ has maximal constituent $[n-2,2]$ and
contains $[n-3,3]$, but not $[n-2,1^2]$.
If the diagram is disconnected, the only possibility for $\chi$ is
$$\chi=[n-k-2,2]\boxtimes [k]=[n-2,2]+[n-3,3]+... + [n-3,2,1]
+\dots +[n-4,2,2].$$
When  $\lambda/\mu$  is connected, the
diagram of $(2^2)$ appears as a subdiagram of  that of  $\lambda/\mu$, by our assumption that $\chi$ does not contain $[n-2,1^2]$.
 Therefore
$$\chi=[(s+r,s+2)/(s)]=[n-2,2]+[n-3,3]+...
=\sum_{a=2}^{\lfloor n/2\rfloor} [n-a,a],$$ and $|\la/\mu|=n$ implies $s+r=n-2$.
\end{proof}

A first contribution towards classifying multiplicity-free
products of skew characters with irreducible characters
is contained in the following easy result.

\begin{lem}\label{mftimesfirsthook}
Let $\chi$ be a proper skew character of $\S_n$.
Then $\chi \cdot [n-1,1]$ is multiplicity-free if and only if $n=2$,
and the product is then
$([2]+[1^2])\cdot [1^2]=[1^2]+[2]$.
\end{lem}
\proof
By  \cref{skew-neighbours}, $\chi$ has two neighbouring constituents,
which we may write as $[\al^X]$ and $[\al^Y]$, for $\al$ a partition of $n-1$ and $X\neq Y$ addable nodes for $\al$.
If one of the two partitions $\al^X,\al^Y$ is not a rectangle, say $\al^X$, then
$([\al^X]+[\al^Y])\cdot [n-1,1]$ contains
$2[\al^X]$, and hence the product is not multiplicity-free.
But if both $\al^X$ and $\al^Y$ are rectangles, then
we have $\al=(1)$ and $\al^X$, $\al^Y$ are the rectangles
$(2)$, $(1^2)$.
\qed

For later usage, we now consider products of irreducible characters
with characters that will appear as subcharacters in certain skew characters.

\begin{lem}\label{lm:special-rect-prod}
Let $n\geq 4$ and  let $\al$ be a partition of $n$.
\begin{enumerate}
\item
Let $\chi=[n-2,2]+[n-2,1^2]$, and let $\al$ be a rectangle.
Then $\chi \cdot [\al]$ is multiplicity-free if and only if
$\al$ is  linear, or $n=4$ and $\al=(2^2)$.
In the latter case,
$\chi\cdot[2^2]=[4]+[3,1]+[2^2]+[2,1^2]+[1^4]$.

\item
Let $\chi=[n]+[n-2,2]$.
Then $\chi \cdot [\al]$ is multiplicity-free if and only if
$\al$ is   linear.
\end{enumerate}
\end{lem}

\begin{proof}
If $\al$ is a linear partition, both products $\chi\cdot [\al]$  are
clearly multiplicity-free.

(1)
The assertion is easily checked for $\al=(2^2)$.
 Summing  the formulas   in  $(i)$ an $(ii)$  of  \cref{Products with a partition of small depth}
 we immediately see that $[a^{b-1},a-1,1]$ appears with multiplicity 2 in $[\al] \cdot \chi$.

(2)
If $\al$ is not a linear partition, then
$[\al]^2$ contains $[n]+[n-2,2]$,
by Proposition~\ref{prop:squares}.
Thus
 $\langle \chi \cdot [\al],[\al]\rangle
= \langle [n]+[n-2,2], [\al]\cdot[\al]\rangle \geq 2 \:,
 $
and hence $\chi \cdot [\al]$ is not multiplicity-free.
\end{proof}

A crucial step towards Theorem~\ref{thm:classification-skew}, and
thus a contribution towards the induction strategy
mentioned previously, is contained in the next proposition.

\begin{prop}\label{prop:step1-to-skew}
Assume that Theorem~\ref{thm:classification} holds for a fixed $n\in \NN$.
Let $\chi$ be a proper skew character of $\S_n$, and let $[\al]$ be
an irreducible character of $\S_n$.
Then $\chi \cdot [\al]$ is multiplicity-free
if and only if $\chi$ and $\al$ are one of the following pairs
(up to multiplication of $\chi$ by a linear character
or conjugation of the partition $\al$):
\begin{enumerate}
\item $\chi$ is any multiplicity-free proper skew character,
and $\al$ is a linear partition;
\item
$n=ab$ for $a,b\geq 2$, $\al=(a^b)$, $\chi=[n-1]\boxtimes[1]=[n]+[n-1,1]$;
 \item
$n=2k$ for $k\geq 2$, $\al=(k,k)$, $\chi=[(k+1,k)/(1)]=[k+1,k-1]+[k,k]$.
\end{enumerate}
\end{prop}

\begin{proof}
We know that the products in the three situations
above are indeed multiplicity-free; for cases (2) and (3) this was
already covered in Section~\ref{sec:mf-products} (and (1) is obvious).

So we now assume that $\chi \cdot [\al]$ is multiplicity-free
and  that $[\al]$ is not linear (in other words
$\al$ is not a linear partition)  and hence we may assume $n>2$.
By Lemma~\ref{mftimesfirsthook} we already know
that $[n-1,1]\cdot \chi$ is not multiplicity-free.
Therefore we need  only consider  $\al \neq (n-1,1)$ (or its conjugate)
and hence we may assume that $n\geq 4$.

We have assumed that   the classification list in Theorem~\ref{thm:classification} is complete for our fixed $n\in \mathbb{N}$ and that $\chi \cdot [\alpha] $ is multiplicity-free.  Every partition on the list is a fat hook and so we deduce that
all constituents of $\chi$ are labelled by fat hooks.
Also, since $\chi$ has a non-linear constituent, $\al$ must be a fat hook.

Thus $\al$ is a fat hook different from $(n)$, $(n-1,1)$ (and their conjugates, by our assumption and \cref{mftimesfirsthook} respectively)
that has a multiplicity-free product with two neighbouring fat hooks
(because of \cref{skew-neighbours}).

We shall now consider the possible partitions $\al$
from the  list
in Theorem~\ref{thm:classification} satisfying these conditions. Case-by-case, we consider   $\alpha$ on the list
$(n-2,2)$, $(n-3,3)$, $(n-2,1^2)$, $(k+1,k)$, $(k,k)$, $(k+1,k-1)$,
  hooks and rectangles, and the few cases for $n\leq 12$.

For $\al=(n-2,2)$, we consider the possible constituents in $\chi$.
When $n>6$, the only non-trivial possible  constituents of $\chi$ such that $\chi\cdot [\alpha]$ is multiplicity-free are
  $(n-1,1)$ (and its conjugate)
or a rectangle.
Since $\chi$ has to have two neighbouring constituents, it must
contain $\chi_0=[n]+[n-1,1]$ (up to conjugating); but
$[n-2,2]\cdot \chi_0$ is not multiplicity-free (for $n\geq 5$).

 For $n=5$, the character $[3,2]$ has a multiplicity-free product with all $[\be]$,
$\be \vdash 5$, $\be \neq (3,1^2)$.
Given the previous arguments, we only have to consider the products with
the neighbour pair sums $[4,1]+[3,2]$ and $[3,2]+[2^2,1]$, and neither of these products
are   multiplicity-free.

  For $n=4$, the   proper skew partitions (up to conjugation) which give multiplicity-free  products with $\alpha$ are
\begin{align*}
[2,2]\cdot([3] \boxtimes [1])&=
   [3,1]+[2,2]+[2,1^2] \\
   [2,2]\cdot [(3,2)/(1)] &= [4]+[3,1]+[2,2]+[2,1^2]+[1^4]  \end{align*}
 and correspond to cases $(2)$ and $(3)$.  The remaining proper skew characters
 $ [2,1] \boxtimes [1]$
 and $[2] \boxtimes [2]$
 which are equal to $[(3,2)/(1)] +[2,1^2]$ and
 $[(3,2)/(1)] +[4]$ respectively and so their products with $[2,2]$ are not multiplicity-free.

Now we consider $\al=(n-3,3)$, and look again for the possible constituents in $\chi$ that have a multiplicity-free product with $[\al]$.
For  $n\geq 7$ the only possible constituents of $\chi$ whose product with $[\alpha]$ is multiplicity-free are
$(n)$, $(n-1,1)$, or $(k,k)$  and their conjugates; with the  exceptions of
  $(4,3)$ for $n=7$ (and their  conjugates).
Recall $\chi$ has a neighbouring pair of constituents, and therefore  must contain $\chi_0=[n]+[n-1,1]$ for all $n \geq 7$   up to conjugation.  However,
$[n-3,3]\cdot \chi_0$ is not multiplicity-free  for $n\geq 7$.

For $n=6$, the character $[3^2]$ has  multiplicity-free products   with $[6]$, $[5,1]$,  $[4,2]$, $[4,1^2]$ and $[3,3]$ (and their conjugates).
The only neighbour pair sums that have multiplicity-free products with $[\al]$
are $[6]+[5,1]=[5]\boxtimes [1]$
and $[3,3]+[4,2]=[(4,3)/(1)]$
(up to conjugation) which correspond to cases (2) and (3) of the proposition.
The first (respectively second)  skew character can only be extended to $[6]+[5,1]+[1^6]$ (respectively cannot be extended) so that the  product with $[3^2]$ remains multiplicity-free.
The former is not a skew character and so does not provide a counter example.
Hence, these considerations for $\al=(n-3,3)$ have only led to
the cases for $n=6$ in (2) and (3).

For $\al=(n-2,1^2)$, we consider the possible constituents in $\chi$.
The only possible constituents in $\chi$ are then
$(n)$, $(n-1,1)$ (and their conjugates) and rectangles.
As before, $\chi$ must then contain $\chi_0=[n]+[n-1,1]$ (up to conjugating),
except when $n=4$, where there are further possible neighbour pairs.
But $[n-2,1^2]\cdot \chi_0$ is not multiplicity-free, and
for $n=4$, no neighbour pair sum has a multiplicity-free
product with $[2,1^2]$.

 For $\al$ a hook partition not equal to $(n)$, $(n-1,1)$, $(n-2,1^2)$ up to conjugation  (which have already been considered) we consider the possible constituents in $\chi$.
The only possible constituents of $\chi$ are
 $(n)$, $(n-1,1)$ or $(k,k)$ (and their conjugates).
Again, $\chi$ must then contain $\chi_0=[n]+[n-1,1]$ (up to conjugating),
but $[\al]\cdot \chi_0$ is not multiplicity-free.

 For $\al=(k+1,k)$ for $k>3$, we consider the possible constituents in $\chi$.
 Then $\chi$ could only have $(n)$, $(n-1,1)$ or $(k+1,k)$ (and their conjugates)
as constituents, except for $n=9$, when also $(3^3)$ can  also appear.
As before, $\chi$ must then contain $\chi_0=[n]+[n-1,1]$ (up to conjugating),
but $[k+1,k]\cdot \chi_0$ is not multiplicity-free.

For $\al=(k+1,k-1)$ for $k>4$, we consider the possible constituents in $\chi$.
Then $\chi$ could only have $(n)$, $(n-1,1)$ or $(k,k)$ (and their conjugates)
as constituents.
As before, $\chi$ must then contain $\chi_0=[n]+[n-1,1]$ (up to conjugating),
but $[k+1,k-1]\cdot \chi_0$ is not multiplicity-free.

Finally we turn to rectangles.
First, let $\al=(a^b)$ with $a\ge b$, and assume $b>2$.
Then $\chi$ could only have $(n)$, $(n-1,1)$, $(n-2,2)$ or $(n-2,1^2)$
 and their conjugates  as constituents, except
$(i)$ for $\al=(3^3)$ when $n=9$
where $(5,4)$ and $(6,3)$ or their conjugates possibly appear, or
$(ii)$ $\al=(4^3)$ when
$n=12$,
 where $(6^2)$ or their  conjugates possibly appear.

 We first exclude the cases $(i)$ and $(ii)$.

If the maximal constituent in $\chi$ is $[n]$, by Lemma~\ref{lm:special-skew-mf} and Lemma~\ref{lm:special-rect-prod}
we must have
$\chi=[n-1]\boxtimes[1]=[n]+[n-1,1]$.
In this case, $\chi\cdot[\al]$ is indeed multiplicity-free, and we are in
situation (2) of the proposition.

If the maximal constituent in $\chi$ is $[n-1,1]$, by  Lemma~\ref{lm:special-skew-mf}
we only have to discuss the cases when $\chi$ is
one of the skew characters
$[n-2,1]\boxtimes[1]=[n-1,1]+[n-2,2]+[n-2,1^2]$,
$[1^2]\boxtimes[n-2]=[n-1,1]+[n-2,1^2]$,
or $[(n-1,n-2)/(n-3)]=[n-1,1]+[n-2,2]$.
By Lemma~\ref{lm:special-rect-prod}(1)  we already know that in the first case
the product $\chi\cdot[\al]$ is not multiplicity-free.

We now consider the second case, where
$\chi=[n-1,1]+[n-2,1^2]$.
In the  computation of the following scalar product
we use the information on special constituents in squares given by Proposition~\ref{prop:squares} several times (here,
we just write ${\uparrow}$ for ${\uparrow}^{\S_n}$).
\begin{align*}
\chi \cdot [\al]\geq 		&\langle [\al]\cdot [n-1,1], [\al]\cdot [n-2,1^2] \rangle \\
 =
&\langle [\al]\cdot [n-1]{\uparrow} - [\al], [\al]\cdot [n-2,1^2] \rangle
\\
= &
\langle [\al]^2, [n-1]{\uparrow} \cdot [n-2,1^2] \rangle
- \langle [\al]^2,  [n-2,1^2] \rangle
\\
= &
\langle [\al]^2, [n-1]{\uparrow} \cdot [n-2,1^2] \rangle \\
= &
\langle [\al]^2, ([n-3,1^2]+[n-2,1]){\uparrow} \rangle \\
= &
\langle [\al]^2, 2[n-2,1^2]+[n-3,2,1]+[n-3,1^3]+[n-1,1]+[n-2,2]  \rangle \\
= &
\langle [\al]^2, [n-3,1^3]+[n-2,2] \rangle = 2,
\end{align*}
and hence $\chi \cdot [\al]$ is not multiplicity-free in this case.
  In the third case,
where $\chi=[n-1,1]+[n-2,2]$,
we follow the same strategy as above and compute
\begin{align*}
\chi \cdot [\al]\geq
 &\langle [\al]\cdot [n-1,1], [\al]\cdot [n-2,2] \rangle\\
 =&
\langle [\al]\cdot [n-1]{\uparrow} - [\al], [\al]\cdot [n-2,2] \rangle
\\
= &
\langle [\al]^2, [n-1]{\uparrow} \cdot [n-2,2] \rangle
- \langle [\al]^2,  [n-2,2] \rangle
\\
= &
\langle [\al]^2, ([n-3,2]+[n-2,1]){\uparrow} \rangle - 1 \\
= &
\langle [\al]^2, 2[n-2,2]+[n-3,3]+[n-3,2,1]+[n-1,1]+[n-2,1^2] \rangle -1 \\
= &
\langle [\al]^2, 2[n-2,2]+[n-3,3] \rangle - 1 = 2
\:.
\end{align*}
Again, it follows that $\chi\cdot [\al]$ is not multiplicity-free.

Now we may assume that $\chi$ contains none of $[n], [n-1,1]$
(or their conjugates).
Note that $\chi$ must contain a neighbour pair sum, so (up to conjugating)
we may now assume that $\chi$ contains
 $[n-2,2]+[n-2,1^2]$.
Given our assumption that $a \geq b>2$, it follows from
 Lemma~\ref{lm:special-rect-prod} that $\chi\cdot[\al]$
 is not multiplicity-free.

We now consider the special cases $(i)$  $\al=(3^3)$ and
 $(5,4)$ and $(6,3)$ or their conjugates   appear in  $\chi$, or $(ii)$
 $\al=(4^3)$  and   $(6^2)$
 or its conjugate    appears in  $\chi$. We remark that these cases can also be checked  by computer.

We first consider case $(i)$.
We assume that $\chi$ has none of the pair sums discussed above;
in which case $\chi$ must have one of the pair sums
$[5,4]+[6,3]$ or $[6,3]+[7,2]$ (up to conjugation).
From the formula in Theorem~\ref{thm:classification}
we immediately see that the first
pair sum does not give a multiplicity-free product.
By Proposition~\ref{prop:squares} we see that
$[3^3]\cdot [3^3]$ contains both $[6,3]$ and $[7,2]$ (with multiplicity~1),
so $[3^3]\cdot([6,3]+[7,2])$ contains $[3^3]$ with multiplicity~2.

We now consider case $(ii)$.
We assume that $\chi$ has none of the pair sums discussed above;
in which case $\chi$ must have
$[6^2]$ as a constituent  (up to conjugation).
 But $\chi$ cannot contain a
neighbour of this constituent, so $\chi$ must contain one of the
pair sums considered above, and so we are done.
This finishes the case of rectangles $(a^b)$ such that $a\geq b >2$.

Now let $\al=(k,k)$ for $k>3$.
Then the constituents in $\chi$ could only
be labelled by
$(n)$, $(n-1,1)$, $(k,k)$, $(k+1,k-1)$, $(n-2,2)$, $(n-3,3)$
$(n-2,1^2)$ (and their conjugates) and hooks,
except for $n=12$ and $k=6$, when also $(4^3)$ or its conjugate can appear.

We follow a similar strategy as before.
We assume first that $n> 12$.

If the maximal constituent in $\chi$ is $[n]$,
then   Lemma~\ref{lm:special-skew-mf}
and Lemma~\ref{lm:special-rect-prod}(2) imply that
$\chi=[n-1]\boxtimes[1]=[n]+[n-1,1]$.
             In this case, $\chi\cdot[\al]$ is indeed multiplicity-free, and we are in
situation (2) of the proposition.

If the maximal constituent in $\chi$ is $[n-1,1]$,
Lemma~\ref{lm:special-skew-mf} now implies that
$\chi$ is one of the following three skew characters:
 $[1^2]\boxtimes[n-2]=[n-1,1]+[n-2,1^2]$;
 $[(n-1,n-2)/(n-3)]=[n-1,1]+[n-2,2]$;
or
$[(n-1,n-3)/(n-4)]=[n-1,1]+[n-2,2]+[n-3,3]$.
For each of the first   two   skew characters,
 the
simple character $[k+1,k-1]$ appears with multiplicity 2 in $\chi \cdot [\alpha]$ using \cref{Products with a partition of small depth}$(i)$ and \cref{lem:LNat}.  The third character contains the second and so the product $\chi\cdot [\alpha]$ also contains multiplicities.

Hence we may now assume that $\chi$ contains none of $[n], [n-1,1]$
(or their conjugates).
As $\chi$ must contain a neighbour pair sum,
we may now assume that $\chi$ contains (up to conjugating)
 one of the skew characters
$(i)$ $[n-2,2]+[n-2,1^2]$,
$(ii)$  $[n-2,2]+[n-3,3]$,
  $(iii)$ a sum of two neighbour hooks (not involving
$[n], [n-1,1]$ and their conjugates),
or $(iv)$  $[k,k]+[k+1,k-1]$.

In case $(i)$,  we know by Lemma~\ref{lm:special-rect-prod} that $\chi\cdot[\al]$
is not multiplicity-free.
  In case $(ii)$
 the
simple character $[k+1,k-1]$ appears with multiplicity 2 in $\chi \cdot [\alpha]$ using \cref{Products with a partition of small depth}$(i)$ and $(iii)$.

In case $(iii)$, the character
$\chi$ contains
 a sum of two neighbouring hooks,
say $[n-a,1^a]+[n-a-1,1^{a+1}]$, with $1<a \leq k-1$.
  By \cref{anfmdsanfsd}, the character $[k+1,k-a,1^{a-1}]$ appears with multiplicity 1 in both
$[k,k]\cdot [n-a,1^a]$ and $[k,k]\cdot [n-a-1,1^{a+1}]$.
Hence $\chi\cdot [k,k]$ is not multiplicity-free.

Finally, we consider the last possible neighbouring pair (from case $(iv)$) which can appear in $\chi$.
 If   $\chi  =[(k+1,k)/(1)]=[k,k]+[k+1,k-1]$, then
by  Proposition~\ref{prop:special-2part-prod}
the product
$$([k,k]+[k+1,k-1]) \cdot [\al]= ([k,k]+[k+1,k-1]) \cdot [k,k] = \sum_{\ell(\la)\leq 4} [\la]$$
is indeed multiplicity-free.
 Now assume the containment
$ [k,k]+[k+1,k-1] \subseteq \chi$ is strict; in which case
 $\chi$ must contain (in addition
to $[k,k]+[k+1,k-1]$) one  of the other
possible constituents $[n-2,2]$, $[n-2,1^2]$,
$[n-3,3]$ or their conjugates, or a hook.

The product of $[k,k]$ with
any of $[n-2,2]$, $[n-2,1^2]$, $[n-3,3]$
has a constituent of length~4, and therefore
$\chi$ cannot contain any of these.
 Next we want to show that $\chi$ cannot
contain any of the conjugates
of $[n-2,2]$, $[n-2,1^2]$, $[n-3,3]$;
note that the first is a neighbour of the other two,
so it cannot occur together with one of those.

First assume that
$\chi=[k+1,k-1]+[k,k]+[2^2,1^{n-4}]$.
Note that this implies that $\la/\mu$ has $k-1$ columns of length~2
and two of length~1, and it has two rows of length~2 and $n-4$ of length~1.
But since $\la/\mu$ is the diagram of a multiplicity-free skew character,
it is connected or has two components of shape as described in Theorem~\ref{thm:mf-outer}, up to rotation of the pieces, and this is
clearly impossible (recall that $n>12$).

Next assume that
$\chi=[k+1,k-1]+[k,k]+[2^3,1^{n-6}]$
 or $\chi=[k+1,k-1]+[k,k]+[3,1^{n-3}]$.
Then, similarly as above,
we obtain a contradiction.

It remains to exclude the case of an additional hook appearing
in $\chi$.
As before, we may assume that $\chi$ does not contain $[n]$
or $[n-1,1]$ (or their conjugates), or pair sums already dealt with.
So assume  $\chi$ has a hook constituent
$[n-m,1^m]$, with $2<m<n-3$; if there is more than
one hook constituent we consider the one with minimal~$m$.
If $n-m>k$, then the hook constituent would be maximal,
and hence then $\la/\mu$ has one column of length $m+1$,
and all others are of length~1.
But then it is clearly impossible that $\chi$ contains $[k,k]$.
On the other hand, if $n-m\le k$, and
$\chi$ also contains any of the conjugates of $[n-2,2]$, $[n-2,1^2]$, $[n-3,3]$,
the previous arguments give again a contradiction,
and finally, the case $\chi=[k+1,k-1]+[k,k]+[n-m,1^m]$ can be handled similarly
as above.

Now as the last case for $\al=(k,k)$,
it only remains to consider the small cases $k\in\{4,5,6\}$.
Here, the arguments used above supplemented by
computer calculations give the claim.
Note that for $n=10$, $k=5$ we may also have the possible
neighbour pair sum $\chi_0=[6,4]+[7,3]=[(7,4)/(1)]$ in $\chi$, but
$\chi_0\cdot [5,5]$ is not multiplicity-free.  This concludes the case in which $\alpha$ is a rectangle.

Finally, it remains to consider the case where $\al$ is a fat hook that
is not of one of the special types discussed so far.  Excluding the cases considered so far, we may conclude that   $n >4$  and $|\Remm(\al)|\geq 2$.
 Therefore $\chi$ must contain
$\chi_0=[n]+[n-1,1]$  (up to conjugation),
but $[\al]\cdot \chi_0$ is not multiplicity-free, as required.  \end{proof}

\begin{prop}\label{prop:step2-to-skew}
Assume that Theorem~\ref{thm:classification} holds for a fixed $n\in \NN$.
Then no product of two proper skew characters of $\S_n$ is multiplicity-free.
\end{prop}

\proof
Under the assumption of our proposition, we have already classified in Proposition~\ref{prop:step1-to-skew}
the multiplicity-free products of a proper skew character and
an irreducible character.

Let $\chi$ be a multiplicity-free proper skew character of $\S_n$ (and therefore $n>2$).
Now by Proposition~\ref{prop:step1-to-skew},
 if $\alpha \vdash n$ is such that   $\chi\cdot [\al]$ is  multiplicity-free, then $\alpha$ is a   rectangle.
 If  $\be$ is a neighbour of $\al$, then $\beta$ is not a rectangle (as $n>2$) and so  $\chi\cdot [\be]$ is not multiplicity-free.
But every proper skew character $\psi$ has two neighbouring constituents, by Lemma~\ref{skew-neighbours}, hence $\chi \cdot \psi$ cannot be
multiplicity-free.
\qed

 \begin{cor}\label{cor:step-to-skew}
 If Theorem~\ref{thm:classification} holds for a fixed $n\in \NN$, then Theorem~\ref{thm:classification-skew} also holds for $n$.\end{cor}

\begin{rmk}\label{induction}
For the remainder of the paper, we shall assume that \cref{thm:classification} (and hence also    \cref{thm:classification-skew}) has been proven by induction for all pairs of partitions of degree strictly less than $n\in \mathbb{N}$.
 We refer to any pair $(\rho,\sigma)$ of partitions of degree strictly less than $n$ and satisfying $g(\rho,\sigma)>1$ as a {\sf seed} (for multiplicity).

    \cref{theo:ThDvir1} implies that  a necessary    condition for $g(\la,\mu)=1$ is that the pair
 $[\la / \la\cap \mu]$,
  $[\mu / \la\cap \mu]$ belongs to the lists in \cref{thm:classification,thm:classification-skew}.

\end{rmk}

\section{Products with a rectangle}\label{sec:rectangle}
 In this section, we shall assume that
$\mu =(a^b)$ is a partition of $n=ab$ with $a,b\geq 3$.

 \begin{prop}
Let $\lambda \vdash n$.    The product
 $[\mu]\cdot [\lambda]$  is multiplicity-free if and only if
 $\lambda$ is one of $(n-2,2)$, $(n-2,1^2)$, $(n-1,1)$, $(n)$, or
 one of the special  partitions $(6,3), (5,4), (6^2)$
 (or conjugate to one of the listed partitions).
 \end{prop}

One half of the proposition follows from \cref{sec:mf-products}.
In this section, we prove the other half of this
 proposition via a series of lemmas.

 {\bf Important standing assumption:} For the remainder of this section, we assume   that $\lambda$ is not one
 of the listed partitions giving a multiplicity-free product,
 and we want to deduce that $[\lambda] \cdot [\mu]$ contains multiplicities.
 We may assume that $\lambda$ is neither a hook, or 2-line partition, and that $\lambda\neq \mu$, as we have already dealt with these cases in \cref{sec:warm-up}.

There are two possible intersection diagrams
for $\lambda$ and $\mu$, up to conjugation;
these are given in Figure \ref{intersectiondiagramrect}.
As indicated in the intersection diagram, we may  assume (by conjugating if necessary) that $w(\lambda) \geq w(\mu)$ for the remainder of this section.
 We will also use the notation indicated there, in other words we let  $\be=\mu\cap \la$,
 $\delta=\la/\be$ (in the second case  $\la/\be = \delta= \delta'\cup\delta''$) and $\gamma = \mu/\be$.

\begin{figure}[ht!]

$$
  \begin{minipage}{45mm}\scalefont{0.8} \begin{tikzpicture}[scale=0.55]
  \draw[very thick]
 (0,1)--(4.5,1)--(4.5,-3.5)--(0,-3.5)--(0,1);
     \draw
(4.5,1)--(6,1)--(6,-0)--(5.5,-0)--(5.5,-1)--(5.5,-1)--(5,-1)--(5,-1.5)--(4.5,-1.5);
     \draw
(4.5,-2)--(4,-2)--(4,-2.5)--(2.5,-2.5)--(2.5,-3)--(1,-3)--(1,-3.5);
          \draw(3.5,-3) node {$\gamma$};                                        \draw(2,-1) node {$\beta$};
          \draw(5,-0) node {$\delta$};
           \end{tikzpicture}\end{minipage}
  \begin{minipage}{38mm}\scalefont{0.8}\begin{tikzpicture}[scale=0.55]
  \draw[very thick]
(0,0)--(4.5,0)--(4.5,-3.5)--(0,-3.5)--(0,0);

      \draw
(4.5,0)--(7,0)--(7,-0.5)--(5.5,-0.5)--(5.5,-1.5)--(4.5,-1.5);
\draw(4.5,-2)--(2,-2)--(2,-3.5);
\draw (1.5,-3.5)--(1.5,-4.5)--(0,-4.5)--(0,-3.5);
          \draw(3.25,-2.75) node {$\gamma$};
          \draw(5,-0.75) node {$\delta'$};
                    \draw(0.75,-4) node {$\delta''$};
                                        \draw(1.25,-1.25) node {$\beta$};
           \end{tikzpicture}\end{minipage}
$$\caption{The two distinct possible intersection diagrams
for a pair $(\la,\mu)$ such that   $\mu$ is a rectangular partition
 (up to conjugation).  }
\label{intersectiondiagramrect}
\end{figure}
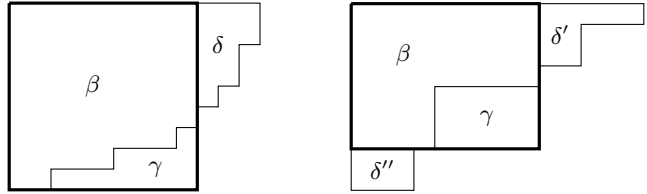

\begin{lem}  \label{rectan}
If   $\lambda=(c^d)$ is a rectangular partition of $n=cd$  for $c,d\ge 2$, then
 $g(\lambda,\mu)>1$.
 \end{lem}

\begin{proof}
We may assume that $a\geq b$.
   Without loss of generality we may assume that $a>c$, and thus $b<d$ (as we have assumed $\lambda\neq \mu)$.
As we have already dealt with 2-line partitions, we may also assume that
$c,d\geq 3$.

Under these assumptions, $\beta= (c^b) \supseteq (3^3)$, and
$\gamma=((a-c)^b)$ and $\delta=(c^{d-b})$ are $(SG)$-removable.
It
then follows that  $1< g( \beta,\beta)\le g(\la,\mu)$ by Lemma \ref{selftensor} and Proposition~\ref{prop:monotonicity}.
\end{proof}

 \begin{lem}\label{lingamrec}
    If the partition $\gamma^{\rm rot}$ is $(1^k)$ for $k\geq 1$, or $(2,1^{k-2})$
    for $k\geq 3$,   and $\delta$ has one connected component, then $g( \la,\mu)>1$.
 \end{lem}

 \begin{proof}
The structure of this proof (and the idea behind many future proofs) is as follows.  Our assumption on $\gamma$ implies that $w(\gamma)\leq 2$.
Generically, we can proceed by removing the first $(a-2)$ columns common to both partitions $\la$ and $\mu$ to obtain
partitions $\tla$  and
 $\tmu$, such that $\tmu$  is  a 2-column partition and $g(\la,\mu) \geq g(\tla,\tmu)$ by the semigroup property.
As $\tmu$ is a  2-line partition, we can then (in most cases)
 apply \cref{prop:res2part} to deduce that $g(\tla,\tmu)>1$.
 However,  if $\ell(\gamma)= b-1$, we shall see that this argument can fail because it is possible that we have reduced to a pair $(\tla,\tmu)$ for which
  $g(\tla,\tmu)=1$.
  We therefore refer to the case in which  $\ell(\gamma)= b-1$ as an `exceptional case' and  provide a separate argument.

We begin with the generic case.  Given $\gamma $ such that  $1\leq \ell(\gamma) \leq b-2$,  we may remove the first $a-2$ columns  from $\lambda$ and $\mu$
 and hence obtain partitions $\tilde{\mu}=(2^{b})$
 and   $\tla=(2^{b-k},1^{k})+\delta $
  (respectively  $\tla=(2^{b-k+1},1^{k-2})+\delta $)
 for $\gamma^{{\rm rot}}=(1^k)$
  (respectively $\gamma^{{\rm rot}}=(2,1^{k-2})$). The result then follows from the case for 2-line partitions.

Now assume $\ell(\gamma)={b-1}$; we have that
$\lambda$ is equal to either
 $(a+b-1, (a-1)^{b-1})$ or $(a+b,(a-1)^{b-2},a-2)$
   for    $\gamma^{{\rm rot}}$ being $(1^{b-1})$
   or  $(2,1^{b-2})$, respectively.
 We first  deal  with the case $\gamma^{{\rm rot}}=(2,1^{b-2})$.
 We set $\tilde\mu=(2^b)$ and $\tilde\la=(3,2^{b-2},1)$ and  rewrite our partitions as follows
 $$\mu=((a-2)^b)+\tilde\mu \; , \;  \lambda=(a-3+b,(a-3)^{b-1})+\tilde\la\:,$$
 and by \cref{prop:monotonicity}, we  have that
 $g(\mu,\la)\geq g(\tilde\mu,\tilde\la)$.
 Now, by \cref{prop:res2part} we have that  $ g(\tilde\mu,\tilde\la)>1$, and        so the result follows.

We now  deal  with the case  $\gamma^{\rm rot}=(1^{b-1})$. We set
 $\tilde{\mu}=(3^b)$ and $\tilde{\lambda}=(b+2,2^{b-1})$  and rewrite our partitions as follows
 $$\mu=((a-3)^b)+\tilde\mu  \; , \; \lambda=((a-3)^b)+\tilde\la\:.$$
For $b=3$ or $b=4$, a direct computation shows $g(\tilde\mu,\tla)>1$.
When $b\ge 5$, we have that
  $$[\tilde{\lambda} / (\tilde{\lambda}\cap \tilde{\mu}^t)]=
  [2^{b-3}]\boxtimes [2]  \; , \;
    [\tilde{\mu}^t / (\tilde{\lambda}\cap \tilde{\mu}^t)]=
  [(b-2)^2]\:.$$
The product of these characters is not multiplicity-free by Proposition~\ref{prop:res2part},  and our inductive proof.
  Therefore, by Proposition~\ref{prop:monotonicity} and Corollary \ref{dvirmaximal}, we have  $g(\la,\mu) \geq g(\tla,\tmu)>1$.
\end{proof}

 \begin{lem}\label{lingamrec2}
    If $\gamma^{\rm rot}$ is $(k)$ or $(k-1,1)$,
      and $\delta$ has one connected component, then $g(\la,\mu)>1$.
 \end{lem}
 \begin{proof}
By \cref{lingamrec}, we may assume that $\gamma^{\rm rot}\neq (1)$ or $ (2,1)$.

 We first consider the case  $\gamma^{\rm rot}=(k)$  for some $2\leq k \leq a$.
We first deal with the exceptional cases which occur for small values of $k$; namely $k=2,3$, and  $(\gamma,\delta)=((4), (2^2))$   for $k=4$.

  We consider the exceptional cases for $k=2$ in detail.
 Here  $\delta$ is equal to $(2)$ or $(1^2)$.
    If we remove all rows and columns common to $\la$ and $\mu$,
   we obtain $(\tla,\tmu) \in \{ ((4), (2^2)), ((3^2),(2^3))\}$.
     Unfortunately,
 $g(\tla,\tmu)=1$ in these cases, and so we have gone too far.
 In other words, we have removed   too many rows or columns.
 If $\delta=(2)$, then there are three ways in which we may have removed too many rows or columns,
 $$   ((5,3,1), (3^3))
\; ,\;   ((4,2,2), (2^4)) \; , \;
  ((8,4), (6^2)). $$
  However, since our original partition $\mu$  contained $(3^3)$ (by assumption), we can choose to reduce only to $((5,3,1), (3^3))$.
   One can deal with $\delta=(1^2)$ in a similar fashion,  and here reduce to the
   exceptional  case $ ((4^2,1), (3^3))$.
     For all these pairs we have $g(\tilde\lambda,\tilde\mu)>1$ by direct computation and the result follows by   Proposition~\ref{prop:monotonicity}.

 For $k=3$, we remove almost all rows and columns
 common to $\la$ and $\mu$    until we reach one of the following pairs  $(\tilde\lambda,\tilde\mu)$:
  $$
   ((6,3^2), (3^4)) \; , \;
  ((7,4,1), (4^3)) \; , \;
  ((6,5,1), (4^3)) \; , \;
  ((5,4,3), (3^4)) \; , \;
  ((4^3), (3^4)).
  $$
  For all these pairs we have $g(\tilde\lambda,\tilde\mu)>1$ by direct computation.
   If
   $(\gamma,\delta)=((4), (2^2))$,  we  remove most rows and columns common to $ \lambda$ and $ \mu$
   and reduce to  $(\tilde\lambda,\tilde\mu)=((6^2,4), (4^4)) $
   or   $((7^2,1),(5^3))$, which satisfy $g(\tilde\lambda,\tilde\mu)>1$ by direct computation.

 We now assume that   we are not in one of the exceptional cases outlined above and so $k\geq 4$ and
 $(\gamma,\delta)\neq ((4), (2^2))$.
 Remove all
  columns common to $ \lambda$ and $ \mu$
   to obtain partitions $\tilde\lambda $ and $\tilde\mu$.
   In the  case $b=3$,  we have that
   $\tilde\lambda$ is a 2-line partition and
   $\tmu=(k^3)$ such that
   $(\tla,\tmu)\neq ((6^2), (4^3))$.
   Therefore $g(\tilde\lambda,\tilde\mu)>1$ by Proposition \ref{prop:res2part}.
   In the  case $b\geq 4$,
   $\tilde\mu\cap\tilde\la = \tilde{\beta}=(k^{b-1})$ with $k\geq 4$ and $b-1\geq3$
   and   $\gamma$ and  $\delta$
   are $(SG)$-removable; the result follows as $g(\lambda,\mu) \geq
  g(\tla,\tmu) \geq  g(\tilde\beta,\tilde\be)>1$.

 We now assume that $\gamma^{\rm rot}=(k-1,1)$ with $k\ge 4$; by \cref{induction} we can assume that $\delta$ is a fat hook.
We first deal with the exceptional cases in which $k=4$ or $5$.
If $(\gamma^{\rm rot},\delta)=((3,1),(4))$ then we remove all but one row or column common to both $\la$ and $\mu$
to obtain pairs of partitions $(\tla,\tmu)$.  For all other pairs of partitions of 4 or 5, we remove all rows and columns common to both $\la$ and $\mu$ to obtain $(\tla,\tmu)$.  The partitions $(\tla,\tmu)$
obtained in this fashion  are all of degree less than or equal to 28, and so can be checked directly (one can reduce this degree even further using the semigroup property, but we do not wish to go into these arguments here).

 We now assume that $\gamma^{\rm rot}=(k-1,1)$ and
     $k\geq 6$.
 We remove all rows and columns common to both $\lambda$ and $\mu$ to obtain $\tilde\lambda$ and $\tilde\mu$.  If   $ \ell(\delta) \leq k-4$
    then
    $ \tilde\lambda\cap \tilde\mu^t= ((k-2)^{k-3})
 $
 and so both
 $\tilde\lambda / (\tilde\lambda\cap \tilde\mu^t)
 $
 and
 $\tilde\mu^t / (\tilde\lambda\cap \tilde\mu^t)
 $ are   $(SG)$-removable; therefore
  $g(\la,\mu)\ge g(\tilde\la,\tilde\mu^t)\geq g( \tilde\lambda\cap \tilde\mu^t, \tilde\lambda\cap \tilde\mu^t ) >1$.

 If $\ell(\delta)\in \{ k-3,k-2,k-1,k\}$
   it remains to check   each of the possible seven such cases.
 If $\delta=(2,1^{k-2}),(3,1^{k-3}), (4,1^{k-4})$ then we may remove
 an appropriate hook of length $k-1$ from~$\tilde\lambda$ (namely $(1^{k-1}), (2,1^{k-3}),(3,1^{k-4})$, respectively)
 and the final row of length $k-1$ from $\tilde\mu$ to obtain a pair of partitions $\hat\lambda,\hat\mu$ which
 differ only by adding and removing a single node; so the result follows from \cref{lingamrec}.
 If $\delta=(1^k)$ then
$\tilde\lambda\cap \tilde\mu^t =(k^{k-1})$
and
 $\tilde\lambda / (\tilde\lambda\cap \tilde\mu^t)
 $
 and
 $\tilde\mu^t / (\tilde\lambda\cap \tilde\mu^t)
 $ are both $(SG)$-removable;
 the result follows as $g(\la,\mu) \geq g(\tilde\lambda\cap \tilde\mu^t,\tilde\lambda\cap \tilde\mu^t)>1$ by \cref{squares}.
 If $\delta=(2^2,1^{k-4})$ or $(2^3,1^{k-6})$ then
 $\tilde\lambda / (\tilde\lambda\cap \tilde\mu^t)
 $
 and
 $\tilde\mu^t / (\tilde\lambda\cap \tilde\mu^t)
 $
 are both linear and the result follows from   \cref{lingamrec}.
If $\delta=(3,2,1^{k-5})$ then $g(\delta,\gamma^{\rm rot})>1$
and so we are done by \cref{dvirmaximal}.
   \end{proof}

\begin{rmk}\label{animportantremark}
In the proof of \cref{lingamrec2}, we
  used our assumptions on $\la$ and $\mu$ to reduce
 our list of exceptional cases  to the pairs
 $((5,3,1), (3^3)) $  and $((4^2,1), (3^3))$,
 whereas one naively could have thought we had to check
 $$((5,3,1), (3^3))
\; , \;   ((4,2^2), (2^4)) \; , \;
  ((8,4), (6^2))    \; , \;   ((4^2,1), (3^3))
 \; , \;   ((4,2^2), (2^4)).$$
   In   future proofs, we shall  use this technique (as detailed in the proof above) without going into further detail.
\end{rmk}

 \begin{lem}\label{givemealabel}
 If $\delta = (1^k) $ or $(2,1^{k-2})$ for $k\geq 4$, then
 $g(\la,\mu)>1$.
 \end{lem}

\begin{proof}
Note that  by Lemmas \ref{lingamrec} and \ref{lingamrec2},
we may assume that $\gamma$ is non-linear.
   If   $\gamma=(2^2)$ and  $\delta=(1^4)$  then we remove all
   but possibly one common  row or column from $\lambda$ and $\mu$ to obtain  a pair of partitions $(\tilde\lambda,\tilde\mu)$ equal to one of the seeds
   $((4^4,1^2),(3^6))$ or    $((3^4,2),(2^7))$.

We may  now assume that   $ \gamma$ is non-linear and
 $(\gamma, \delta)\neq ((2^2),(1^4))$.
 Remove all rows and columns common to both $\lambda$ and $\mu$ to obtain pairs of partitions $(\tmu,\tla)$ equal to
   $$ (   (w(\gamma)^{k+\ell(\gamma)}),
   ((w(\gamma)+1)^k, \gamma^c) )
  \; , \;
   ((w(\gamma)^{k+\ell(\gamma)-1}),
    (w(\gamma)+2, (w(\gamma)+1)^{k-2}, \gamma^c))$$
   for $\delta=(1^k)$ and   $\delta=(2,1^{k-2})$, respectively;
  here $\gamma^c=(w(\gamma)^{\ell(\gamma)}/\gamma )$ is the rectangular complement of $\gamma $.

    Let  $w(\gamma)=2$  (and $\delta=(1^k),(2,1^{k-2})$)
    for $k\geq 4$.
  Then     $(3^3)\subseteq \tilde\lambda$
      and
     $\tilde\mu=(2^{k+\ell(\gamma)})$.
The result follows as $g(\tla,\tmu)>1$ by   \cref{section2parter}.

If $w(\gamma)\geq 3$ and $\delta=(1^k)$  then
 $\tilde{\mu}^t /(\tilde{\mu}^t\cap \tilde\lambda)$ and $\tilde{\lambda} /(\tilde{\mu}^t\cap \tilde\lambda)$ are $(SG)$-removable
 and
 $\tilde{\mu}^t\cap \tilde\lambda=((w(\gamma)+1)^{w(\gamma)})$.  The result follows as $g(\tilde{\mu}^t\cap \tilde\lambda,\tilde{\mu}^t\cap \tilde\lambda)>1$  by \cref{squares}.

   If $\delta=(2,1^{k-2})$,  and  $\gamma$ is a rectangle  such that  $w(\gamma)\geq 3$,
   then the partitions $ \gamma$ and $ \delta$ are $(SG)$-removable and $g(\tilde\beta,\tilde\beta)>1$ by \cref{rectan}.
 By   \cref{lingamrec2} and the above, we can now assume that $\gamma^{\rm rot}$ is a   hook  not equal to $(k)$ or $(k-1,1)$.
 If $\gamma^{\rm rot}\neq (k-2,2)$,  then
$\tla/ (\tla \cap \tmu^t)$  is a  proper fat hook
and $[\tmu^t/ (\tla \cap \tmu^t)]$ is  not the natural character;
 therefore $g(\tla,\tmu)>1$ by  \cref{induction}.
Finally, if  $\gamma^{\rm rot}=(k-2,2)$ then
\begin{align*}
(  \tmu, \tla)
&=
(((k-2)^{k+1}),
(k,(k-1)^{k-2},k-4)		)
\\
&=
(((k-3)^{k+1}),
((k-2)^{k-2},k-5)		)
  +
  ( (1^{k+1})	, (2,1^{k-1})	).
  \end{align*}
We have that     $ g((k-3)^{k+1}),
((k-2)^{k-2},k-5)		)>1$ by \cref{lingamrec2}.   The result follows by \cref{prop:monotonicity}.
  \end{proof}

 \begin{lem}\label{rectnagledeltalinear}
 If $\delta = (k) $   then $g(\la,\mu)>1$.  \end{lem}

\begin{proof}

We first consider the case where $\gamma$ is a rectangle.   If $\gamma=(2^2)$, then we remove almost all rows and columns common to $\lambda$ and $\mu$ to obtain $(\tilde\lambda,\tilde\mu)$ equal to one of
the seeds  $((6,2^2),(2^5))$
 $((7,3,1^2),(3^4))$
$((8,2^2), (4^3))$.
If  $\gamma=(2^k)$ for $k\geq 3$ then we remove almost all
common   rows and  columns   of $\lambda$ and $\mu$ to obtain  $
(\tilde\lambda,\tmu)$ equal to either of
$$
(
(2k+3,1^k),
(3^{k+1})
)
\; , \;
((2k+2,2^2),
(2^{k+3}))\:.
$$
   In the former case, the result follows by \cref{sectionhook!} as
   $\tla$ is a hook. In the latter case the result follows from \cref{section2parter} as
    $\tilde\mu$ is a 2-line partition.

   Now assume $\gamma=(k,k)$ for $k\geq3$.
We remove almost all common   rows and columns to  obtain pairs of partitions $(\tla, \tilde\mu)$ equal to either of
$$
((3k,k),(k^4))
\; . \;
((3k+2,2^2), ((k+2)^3))\:.
$$
   In the former case, $\tilde\lambda$ is a 2-line partition and the result follows.
 In the latter case, remove the final row   of
 $\tilde\mu$ and
 the partition $(k+2)$ from the first row of  $\tilde\lambda$ to obtain partitions
   $\hat\mu=(k+2,k+2)$ and $\hat\lambda=(2k,2,2)$.
 The result again follows from \cref{section2parter}.

We now consider the case that $\gamma=(t^u)$ is a fat rectangle for $t,u \geq 3$.
We may proceed as above by removing all but one  common   row or  column to  obtain pairs of partitions $(\tla, \tilde\mu)$ equal to either of
 $$((tu+t,t), (t^{u+2}))
   \; , \; ((tu+t+1,1^u) ,((t+1)^{u+1}))\:,$$
 respectively.
  In the former  (respectively latter) case the result follows from \cref{section2parter} (respectively \cref{sectionhook!}).

We now assume that $\gamma^{\rm rot}=(t^u,v^w)$ is a non-rectangular fat hook, in other words $t\neq v$ and $u,w\neq0$.
  We first  consider the case where $\ell(\gamma)<b-1$ or $w(\gamma)<a$.
 By assumption, $\beta=\mu\cap \lambda$ has at least
 two removable nodes $A_1$ and $A_2$
  such that  $A_i$ and $\delta$ are disconnected    for $i=1,2$.
 We may assume that $\gamma \cup\{A_1\}$ is not a rectangular partition.

We want to apply \cref{lem:prepalmext}
and recall the definition of the virtual character $\chi$ given there
in \cref{chi};
note that here $\al=\gamma^{\rot}$.
\begin{equation}
\chi=
\sum_{A\in \Remm (\beta)} [   \la/ \beta_A] \cdot [ \mu / \beta_A]
-  \sum_{B \in \Add(\alpha)} \alpha^B
\:.
\end{equation}
For the two terms on the right-hand side, we note that the subtracted term is multiplicity-free.
 By assumption $[\lambda / \beta_{A_i}] = [k+1]+[k,1]$  for   $i=1,2$.
Also note that $[\gamma\cup\{A_1\}]=[\al^{A_1}]$.  Therefore, we have that
 \begin{align*}
\langle \chi , [ \al^{A_1}]  \rangle
& \geq
\sum_{i=1,2 } \langle
 [ \mu / \beta_{A_i}] \cdot [\lambda/ \beta_{A_i}], [ \al^{A_1} ] \rangle
   - 1
\\
&\geq
\langle
 [ \mu / \beta_{A_1}] \cdot([k+1] + [k,1]), [ \al^{A_1} ] \rangle
 +
\langle  [ \mu / \beta_{A_2}] \cdot ([k+1] + [k,1]), [ \al^{A_1} ] \rangle
  -1 \\
&  \geq 2+1-1
    \end{align*}
and the result follows by \cref{lem:prepalmext}.

We now consider the case in which $\ell(\gamma)=b-1$ and $w(\gamma)=a$ and so $t\geq 3, u+w\geq 2$.
If  $w=1$, then the result follows   as $\la$ is a 2-line partition.
If $w>1$, we remove the final $u$ rows from $\mu$
 and $(tu)$ from the first row of $\lambda$ to obtain
  $\tilde{\mu}=(a^{1+w})$ (and so has at least three lines) and $\tilde\lambda$ a partition which is neither
  a hook nor a 2-line partition.
In this case,  $\tilde\gamma=\tmu/(\tmu\cap\tla)$ is a rectangle,
therefore the result follows from the above and \cref{prop:monotonicity}.

 Finally, we consider the case in which $\gamma^{\rot}$ is not a fat hook, i.e.,
 $|\Remm(\gamma^{\rot})|>2 $.
 Then we apply the following iterative procedure to reduce to the situation
dealt with before.
 \begin{itemize}[leftmargin=0pt,itemindent=1.5em]
\item[$(1)$]  If $w(\gamma)\neq w(\mu)$, and $|\Remm(\gamma^{\rot})|>2 $,
  then we remove all columns  common to both
 $\la$ and $\mu$ to obtain a pair $(\tla,\tmu)$ such that
 $ \tmu/(\tla\cap\tmu) =\gamma $ and therefore $w(\tilde\mu)\geq 3$, $\ell(\tilde\mu)\geq 4$.
\item[$(2)$]  If $w(\gamma)= w(\mu)$, and $|\Remm( \gamma^{\rot})|>2 $,
  then we remove the final $\ell(\mu)-\ell(\lambda)$ rows from
   $\mu$ and the corresponding number of nodes from $\la_1$
    to obtain a pair $(\tla,\tmu)$ such that
 $| \Remm((\tmu/\tla\cap\tmu)^{\rm rot})|=  |\Remm(\gamma^{\rot})|-1  $
  and $w(\tmu/(\tla\cap\tmu))<w(\tmu)$ and $w(\tmu)\geq 3$, $\ell(\tmu)\geq 3$.
\item[$(3)$] Having completed $(1)$ or $(2)$ above, relabel the partitions $(\la,\mu):=(\tla, \tmu)$
and apply $(1)$ or $(2)$ again, if possible.
  \end{itemize}
The above procedure eventually terminates by producing a pair of partitions
$( {\la},  {\mu})$ such that $w(\mu)\geq 3$, $\ell(\mu)\geq 3$,
$|\Remm(\gamma)|=2$;  therefore  the result  follows by the semigroup property and the case for fat hooks, covered above.
  \end{proof}

\begin{lem}\label{lem:delta-nat}
 If $\delta = (k-1,1)$, then  $g(\la,\mu)>1$.
 \end{lem}

\begin{proof}
By \cref{induction}, we may assume that $\gamma^{\rot}$ is a  fat hook.
If $\gamma=(2^k)$ and $\delta=(2k-1,1)$, then we remove all but one row or column of $\lambda$ and $\mu$ to obtain partitions $\tilde\lambda$ and $\tilde\mu$ such that $\tilde\la \cap \tilde\mu=(2^3)$ or $(3^2,1^{k})$ respectively.
  If    $\gamma\neq (2^k)$, then remove all rows and columns common to both $\lambda$ and $\mu$ to obtain partitions
  $\tilde\lambda$ and $\tilde\mu=((w(\gamma))^{2+\ell(\gamma)})$.

In either case, we now
 remove the final row of $\tmu$ to obtain $\hat\mu$ and we let $\hat\lambda$ denote the partition such that
  $\hat\lambda + (w(\hat\mu)-1,1) = \tla$.
 The partition
 $\hat\mu$ is  a rectangle  and $\hat\lambda$   is either  a proper fat hook or   $|\Remm(\hat\lambda)|=3$
 and  such that $\hat\lambda / (\hat\mu \cap \hat\lambda) = (k-w(\tmu))$.  The result follows from \cref{rectnagledeltalinear}.
  \end{proof}

\begin{rmk}\label{nonatural}
For the remainder of this section, we shall assume that
  $[\delta]$ is not equal to a linear character or the natural character or its conjugate.
  Similarly  if $\delta$ has one connected component,
  then we shall assume that
    $[\gamma]$ is not equal to a linear character or the natural character or its conjugate.
\end{rmk}

\begin{lem}\label{2erssss}
If $\gamma^{\rm rot}$ and $\delta$ are both 2-line partitions,
  then  $g(\la,\mu)>1$.
\end{lem}

\begin{proof}
We first suppose that $w(\gamma)=\ell(\delta)=2$.
 There are three cases to consider:  $(i)$   $\gamma=(2^k)$;
  $(ii)$ $\gamma \neq (2^k)$ and $\delta=(k^2)$;
  $(iii)$
 $\gamma \neq (2^k)$ and $\delta\neq (k^2)$.

Case $(i)$.  Remove all common rows and  all but  one  common  column   of $\la$ and $\mu$ to obtain $(\tla,\tmu)=
 ( (3^2,1^k)+\delta, (3^{k+2})) $.
For $k=2$ and  $k=3$ it is easily checked that the corresponding pairs are seeds.
When $k>3$, we
   note that at least one of $(3^2), (4,2), (5,1)$ is $(SG)$-removable from $\delta$ (and hence is  also $(SG)$-removable from the first two rows of $\tla$).
 In this case, we remove the final two rows of $\tmu$ and the relevant partition from $\tla$
  to obtain  $(\hat\mu,\hat\la)$ such
 that $\hat\mu/ (\hat\mu\cap \hat\la)$ is a non-linear rectangle and  $\hat\la/ (\hat\mu\cap \hat\la)$ is a proper skew partition
 not of one of the forms described in cases $(2)$ and $(3)$ in \cref{thm:classification-skew}.
  The result then follows by \cref{induction}.

In case  $(ii)$ (respectively $(iii)$) we remove all rows and columns common to $\la$ and $\mu$ to obtain $\tmu$ a 2-column rectangle and $\tla$  a proper fat hook (respectively $\tla$ such that $|\Remm(\tla)|=3$).  The result then follows from \cref{section2parter}.

For the remainder of the proof we assume that at least one of
$w(\gamma)$ and $\ell(\delta)$ is greater than 2.  In the generic case, we remove all common rows and columns from $\la$ and $\mu$ to obtain partitions $\tla$ and $\tmu$ and proceed case-by-case. We will deal with the exceptional cases when they appear in that discussion.

 If $\gamma=(k^2)$
and $\ell(\delta)=2$ (respectively $\gamma=(2^k)$
 and $\ell(\delta)>2$)
then  $\tla$ (respectively   $\tmu$) is a 2-line partition and the result follows from \cref{section2parter} as long as we are not in the case
$\gamma = (3^2)= \delta$.
In the exceptional cases we remove all but one common row or column from $\la$ and $\mu$ to obtain $(\tla,\tmu)$.
For $\gamma=(k^2)$, we have in the exceptional case $(\tla,\tmu)= ((6^2,3),(3^5))=((4^2,2)+(2^2,1)),((2^5)+(1^5)))$ or $((7^2,1^2),(4^4))=((3^2,1^2)+(4^2),(2^4)+(2^4))$, respectively.
Hence we can reduce to  $(\hat\la,\hat\tmu)=
((4^2,2),((2^5))$ or $((3^2,1^2),(2^4))$, respectively,
and $g(\hat\la,\hat\tmu)>1$ by \cref{section2parter}.
In the exceptional case for $\gamma=(2^k)$ we quickly reduce to
a pair involving a 2-column partition where we
can again appeal to  \cref{section2parter}.

 If $\gamma=(k^2)$
 and $\ell(\delta)>2$, then $\delta$ and $\gamma$ are $(SG)$-removable and
  $(3^3)\subseteq \tla\cap \tmu$ and the result follows from \cref{squares}.

We may now assume $\gamma\neq (k^2)$ up to conjugation.
If $w(\gamma)=2$, then $\tmu$ is a 2-line partition and
$\tla$  is a proper fat hook or $|\Remm(\tla)|=3$.
Now assume  $\ell(\gamma)=2$ and $\ell(\delta)=2$.
 If $\gamma^{\rot}=(k,k-1)$, remove the two lower rows from $\tmu$ to obtain $\hat\mu=(k^2)$, and note that $\tla= (k^2)+\hat\la$ where $\hat\la$ is a partition with  $|\Rem(\hat\la)|=3$; hence the result follows by \cref{section2parter}.
If $\gamma^{\rot}=(k,k-j)$ for $j>1$, we have $\tla=(3^2,2)+\hat\la$ for some partition $\hat\la$, and $\tmu=(2^4)+((k-2)^4)$, so with
$g(\la,\mu)\ge g(\tla,\tmu) \ge g((3^2,2),(2^4))$ the claim follows.

 It remains to check
  the cases in which $\ell(\gamma)=2$ and $\ell(\delta)>2$; namely
 $(i)$  $\gamma^{{\rm rot}}=(2k-2,2)$ and $\delta=(2^k)$;
 $(ii)$  $\gamma^{{\rm rot}}=(2k-3,3)$ and $\delta=(2^k)$;
 $(iii)$ $\gamma^{{\rm rot}}=(k+1,k-1)$ and $\delta=(2^{k-1},1^2)$;
 $(iv)$ $\gamma^{{\rm rot}}=(k+1,k-1)$ and $\delta=(2^k)$;
 $(v)$ $\gamma^{{\rm rot}}=(k+1,k)$ and $\delta=(2^k,1)$.

 In case $(i)$, for $k\geq 6$
 (one can check the seeds for $k=3,4,5$ directly) we have that $\tla^t  /  (\tla^t\cap \tmu)$ and
 $\tmu /  (\tla^t\cap \tmu)$ are $(SG)$-removable and
the rectangle $\tla^t\cap \tmu$ contains $(3^3)$,
so the result follows from \cref{squares}.
 Case $(ii)$ is similar.
 In  cases $(iii)$ to $(v)$, we have that $\tla^t  /  (\tla^t\cap \tmu)$ and
 $\tmu /  (\tla^t\cap \tmu)$ are linear partitions and the result follows from \cref{lingamrec}.
   \end{proof}

 \begin{lem}
If $\gamma$ or $\delta$ is a fat rectangle and $\delta$ has one connected component,    then  $g(\la,\mu)>1$.
 \end{lem}

\begin{proof}
  By \cref{induction}  and  \cref{nonatural} we may assume that one of  $\delta$ and $\gamma$ is a fat rectangle and the other is $ (k-2,2)$, or $(k-2,1^2)$   or $(5,4)$ or $(6,3)$ up to conjugation.

  We first suppose  that $\gamma$  is a fat rectangle. Remove all rows and columns common to both partitions $\la,\mu$ to obtain $\tilde\lambda$ and $\tilde\mu$.
   If $\ell(\delta) \geq 3$, then the result holds as
  $\gamma$ and $\delta$ are $(SG)$-removable and
  $(3^3) \subseteq \tla \cap \tmu$.
In the remaining cases, $\delta=(k-2,2)$,   $  (5,4)$ or $(6,3)$, the partition
 $\tla$ is a 2-line partition and
$\tmu$ is a   fat rectangle; the result follows  by \cref{section2parter}.

We now suppose   $\delta$ is a fat  rectangle.
Remove all rows and columns common to both partitions $\la,\mu$ to  obtain $\tla,\tmu$.
For $\gamma^\rot=  (5,4)$ or $(6,3)$, this follows via $(SG)$-removability from $g((5^3),(3^5))>1$, the conjugate case is immediate from \cref{section2parter}.

If $\gamma^{\rm rot}=(k-2,2), (k-2,1^2)$ or $ (3,1^{k-3})$ then
$\tla^t\cap \tmu$ is a fat rectangle and $\tla^t/\tla^t\cap \tmu$ and
$\tmu^t/\tla^t\cap \tmu$ are $(SG)$-removable; the result follows by \cref{squares}.
  If $\gamma^{\rm rot}=(2^2,1^{k-4})$, then $\tmu$ is a 2-line partition and $\tla$ is a proper fat hook; the result follows by \cref{section2parter}.
   \end{proof}

   \begin{lem}\label{hookhook}
If one of $\gamma^{\rm rot}$ or  $\delta$ is a hook
and the other is equal to $(k^2)$ or $(2^k)$, then  $g(\la,\mu)>1$.
\end{lem}

\begin{proof}
First note that our assumptions in \cref{nonatural} imply that $k>2$.
  If $\delta$ is a hook,  remove all columns common to $\la$ and $\mu$ to obtain $\tla$ and $\tmu$.
 If  $\gamma=(2^k)$ then the result follows from the result for 2-line partitions.
If   $\gamma=(k^2)$, then
$\tla / \tla\cap \tmu$ and $\tmu / \tla\cap \tmu$ are $(SG)$-removable and
$\tla\cap \tmu$ is a fat rectangle;  the result follows by \cref{rectan}.

Now assume that $\gamma^\rot$ is a hook and $\delta= (2^k)$ or $(k^2)$. Remove all rows and columns common to both $\la,\mu$ to obtain partitions $\tilde\mu=(t^u)$ and
$\tilde\lambda=  ((t+2)^k,(t-1)^{u-k-1})$
or $((t+k)^2, (t-1)^{u-3})$
  respectively; in these cases, $t+u=3k+1$ and $t+u=2k+3$, respectively.

In the case $t=u$, i.e., $\tilde\mu=(t^t)$ is a square, we must have
$\delta=(2^k)$.
Let $\gamma^\rot=(2k-m,1^m)$, where we have $2\le m\le 2k-3$.
Since $t=k+m+1$ and $t+m=2k$, we  obtain  $t=3m+2$.
Hence the final $m+1$ rows of $\tilde\mu=(t^t)$ form a
partition of size $3m^2+5m+2$; removing this gives
$\hat\mu=((3m+2)^{2m+1})$.
On the other hand, we can remove a partition of the
corresponding size from $\tla$, as $\tla = ((m+2)^{2m+1},m^m) +\hat\lambda$,
with $\hat\lambda=((2m+2)^{2m+1},{(2m+1)}^m)$.
Thus $\hat\lambda$ and $\hat\mu$ are $(SG)$-removable; since $\hat\lambda \cap \hat\mu=((2m+2)^{2m+1})$ and
$g(\hat\lambda \cap \hat\mu,\hat\lambda \cap \hat\mu)>1$,
we are done in this case.

Hence we may now assume that $t\ne u$.
We conjugate the partition $\tilde\mu$ and consider the possible intersection diagrams
 $D_1=\tilde\lambda/(\tilde\lambda\cap \tilde\mu^t)$ and
 $D_2=\tilde\mu^t/(\tilde\lambda\cap \tilde\mu^t)$.

By our assumptions we have $k >2$, $t\ge 3$, $u\ge 5$.
Thus if $D_1$ is disconnected,
both components are of size strictly  greater than 1.
When $\delta=(2^k)$, $D_1$ could be disconnected only when both
$u<t+2$ and $t<u-1$, which is impossible.
When $\delta=(k^2)$, $D_1$ is disconnected if and only if $t+1<u<t+k$.
Then $D_2$ is a rectangle of width $u-t+1\ne 1$ and height $t-2$.
If $t=3$, then $u=u+t-3=2k$, hence $k<t=3$, contradiction.
Hence $D_2$ is a non-linear rectangle.
But then the character pair $([D_1],[D_2])$
is not on the list in \cref{thm:classification-skew}; hence
by induction $g(\tla,\tmu^t)>1$, and thus $g(\la,\mu)>1$.

So we now assume that $D_1,D_2$ are both connected;
in fact, then $D_1$ must be a rectangle and $D_2$ is a fat hook.
If $D_2$ is a proper fat hook or $D_1$ a fat rectangle,
or if one is a rectangle and the other is not a hook,
we are done by the previous results of this section.
It remains to consider the case where $D_1$ is a 2-line
rectangle of size $2r>4$ and $D_2$ is a hook
of the form $(2r-m,1^m)$ for $2\leq m\leq 2r-3$.
When $\delta=(k^2)$, this implies that
$t=3$ and $2k=u=t+k+1$, hence $k=4$, $u=8$.
When $\delta=(2^k)$, this implies that
$t=k+1$ and $t+3=u=2k$, so again $k=4$, $u=8$.
In both cases, we can remove a column of length 8 from
$\tmu$ and remove $\delta=(4^2)$ from $\tla$, and the result then follows from \cref{lingamrec2}.
 \end{proof}

\begin{lem}
If $\delta$ has two connected components, then  $g(\la,\mu)>1$.
\end{lem}

\begin{proof}
By  \cref{induction},
it suffices to consider  (up to conjugation of $\la$ and $\mu$) the cases $(i)$   $\delta''=(l)$  or $(1^{l})$ and $ \delta' =(1)$
 and $\gamma\vdash l+1$ is a rectangle;
$(ii)$ $\gamma = (k+l) $  and  $[\delta]=[\delta']\boxtimes [\delta'']$ is one of the products from the list in \cref{thm:mf-outer} with $\delta',\delta''$ of size $k,l$, respectively, and $(\delta',\delta'')$ not a pair as in $(i)$.  We cover both cases uniformly.

The unique exceptional subcase is
 $\gamma=(k+l)$ and $ \delta' =(k)$, $\delta''=(1^l)$
 (up to conjugation of $\la$ and $\mu$) in which case we remove all rows and columns common to both partitions with the exception of one row (which exists by our assumption that $\mu$ is not a 2-line partition) to obtain $(\tla,\tmu)$ of the form
 $$
    ((2k+l+1, k+l+1, 1^{l+1}),
    (  (k+l+1)^3)	).
    $$
    Now suppose that
 $\gamma=(k+l)$  and $\delta'\vdash k$ and $\delta''\vdash l$
 are not of the above form.  Remove all rows and columns common to $\la$ and $\mu$ to obtain $\tla$ and $\tmu$.

  Now,   if $\delta'=(1)$   and $\delta''=(1^{l})$ then
   $\tla / (\tmu^t \cap \tla)$ is disconnected,
   with two components  $(l,l-1)$ and $(1)$, and
   $\tmu^t / (\tmu^t \cap \tla) = (2^l)$;
   for $l>1$ the result follows as this product is not on the list in \cref{thm:classification-skew}, and for $l=1$, the pair $((4,3,1^2),(3^3))$ is a seed.
If $k>1$ (in either the exceptional or generic cases) then
  $\tmu$ is a rectangle and  $$w(\tmu) = k+l+1 = |\delta'|+|\delta''|+1 \geq \ell(\delta')+\ell(\delta'')+1 = \ell(\tla)$$
  and therefore $\tmu^t / (\tla \cap \tmu^t)  $
 and
  $\tla / (\tla \cap \tmu^t)  $
  each have precisely one component.     The result follows by the earlier results in this section.

 We now suppose that  $\delta'=(1)$   and  $ \delta'' =(l)$ or $(1^l)$ for some $l\geq 3$   and that $\gamma=(t^u)$ for $t,u >1$.
 Remove all rows and columns common to $\la$ and $\mu$ to obtain $\tla$ and $\tmu$.
  If  $ \delta'' =(1^l)$, then $\tla\supset (4,1^3)$ is a hook partition and $\tmu$ is a fat rectangle and the result holds by \cref{sectionhook!}.
  If   $\delta''=(l)$,  then we remove the first row $\tla_1=(t+tu)$ of $\tla$ and
  the final $u$ columns of $\tmu$ to obtain the pair $\hat\la=\hat\mu=((u+2)^l)\supseteq (3^3)$  and the result follows by \cref{squares}.
     \end{proof}

\section{Products with a proper fat hook}\label{sec:fatty}

In this section, we shall consider tensor products in which one of the
labelling partitions   is a \emph{proper} fat hook
(in other words, a fat hook  which is not a 2-line,  hook, or rectangular partition).
We assume throughout this section that $\mu=(a^b,c^d)$ is a proper fat hook partition.
 \begin{prop}
Let $\lambda \vdash n$.  The product $[\mu] \cdot [\la]$ is multiplicity-free if and only if $\la=(n)$ or $(n-1,1)$ up to conjugation.
\end{prop}

One half of the proposition follows from \cref{sec:mf-products}.
In this section, we prove the other half of this
 proposition via a series of lemmas.
 For the remainder of this section, we assume   that $\lambda\neq (n)$ or $(n-1,1)$ up to conjugation
 and we want to deduce that $[\lambda] \cdot [\mu]$ contains multiplicities.
 We may assume that $\lambda$ is neither a rectangle, a hook, or 2-line partition, and that $\lambda\neq \mu$, as we have already dealt with these cases in  \cref{sec:warm-up,sec:rectangle}.

The possible  intersection diagrams
for $\lambda$ and $\mu$, up to conjugation,
are given in \cref{singlefatty,singlefatty2}.
  We will also use the notation indicated there, in other words we let  $\be=\mu\cap \la$,
 $\delta=\la/\be$   and $\gamma=\mu/\be$.
Informally, we refer to the overlapping rectangles
 of shape $(a^b)$ and of shape $(c^{b+d})$  as  the arm    and the leg of $\mu$, respectively.

 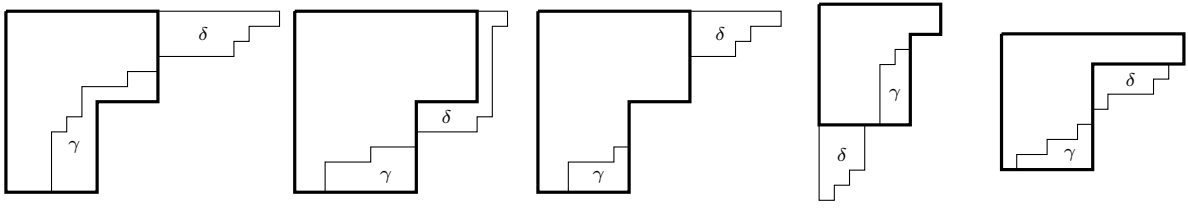
\begin{figure}[ht!]
$$
\scalefont{0.6}
     \begin{minipage}{38mm}\begin{tikzpicture}[scale=0.4]
  \draw[very thick]
      (-1,1)--(4,1)--(4,-2)--(2,-2)--(2,-5)--(-1,-5)--(-1,1);
  \draw
      (4,1)--(8,1)--(8,0.5)--(7,0.5)--(7,0)--(6.5,0)--(6.5,-0.5)--(4,-0.5)--(4,1);
  \draw
       (4,-1)--(3,-1)--(3,-1.5)--(1.5,-1.5)--(1.5,-2.5)--(1,-2.5)--(1,-3)--(0.5,-3)--(0.5,-5);
        \draw(1.25,-3.5) node {$\gamma$};
          \draw(5.5,0.25) node {$\delta$};
       \end{tikzpicture}\end{minipage}
      \begin{minipage}{32mm}\begin{tikzpicture}[scale=0.4]
  \draw[very thick]
      (-2,1)--(4,1)--(4,-2)--(2,-2)--(2,-5)--(-2,-5)--(-2,1);
  \draw
      (4,1)--(5,1)--(5,0.5)--(4.5,0.5)--(4.5,-2.5)--(4,-2.5)--(4,-3)--(2,-3);
  \draw
       (2,-3.5)--(.5,-3.5)--(.5,-4)--(-1,-4)--(-1,-5);
         \draw(1,-4.5) node {$\gamma$};
          \draw(3,-2.5) node {$\delta$};
       \end{tikzpicture}\end{minipage}
      \begin{minipage}{37mm}\begin{tikzpicture}[scale=0.4]
  \draw[very thick]
      (-1,1)--(4,1)--(4,-2)--(2,-2)--(2,-5)--(-1,-5)--(-1,1);
 \draw
      (4,1)--(7,1)--(7,0.5)--(6,0.5)--(6,0)--(5.5,0)--(5.5,-0.5)--(4,-0.5)--(4,1);
  \draw
        (2,-3.5)--(1.5,-3.5)--(1.5,-4)--(0,-4)--(0,-5);       \draw(1,-4.5) node {$\gamma$};
          \draw(5,0.25) node {$\delta$};
       \end{tikzpicture}\end{minipage}
            \begin{minipage}{24mm}\begin{tikzpicture}
              [scale=0.4,x={(0,-1cm)},y={(-1cm,0)}]
   \draw[very thick]
      (0,1)--(4,1)--(4,-2)--(1,-2)--(1,-3)--(-0,-3)--(0,1);
  \draw
      (4,1)--(6.5,1)--(6.5,0.5)--(6,0.5)--(6,0)--(5.5,0)--(5.5,-0.5)--(4,-0.5)--(4,1);
 \draw
       (4,-1)--(2,-1)--(2,-1.5)--(1.5,-1.5)--(1.5,-2);
       \draw(3,-1.5) node {$\gamma$};
          \draw(5,0.25) node {$\delta$};
       \end{tikzpicture}\end{minipage}
                   \begin{minipage}{35mm}\begin{tikzpicture}[scale=0.4,x={(0,-1cm)},y={(-1cm,0)}]
   \draw[very thick]
      (-0.5,1)--(4,1)--(4,-2)--(0.5,-2)--(0.5,-5)--(-0.5,-5)--(-0.5,1);
  \draw
       (4,0.5)--(3.5,0.5)--(3.5,-0.5)--(3,-0.5)--(3,-0.5)--(3,-1)--(3,-1.5)--(2.5,-1.5)--(2.5,-2);
       \draw(3.5,-1.25) node {$\gamma$};
        \draw
       (2,-2)--(2,-2.5)--(1.5,-2.5)--(1.5,-4)--(1,-4)--(1,-4.5)--(0.5,-4.5) ;
          \draw(1,-3.25) node {$\delta$};
       \end{tikzpicture}\end{minipage}           $$
\caption{The possible intersection diagrams  for which $\mu$  is a proper fat hook and
$\gamma$ and $\delta$ each have one connected component.
We label these diagrams by
$(1a)$, $(1b)$, $(1c)$, $(1d)$    and $(1e)$ respectively.
}
\label{singlefatty}
\end{figure}

\begin{figure}[ht!]
$$
\scalefont{0.6}
      \begin{minipage}{32mm}\begin{tikzpicture}[scale=0.4,x={(0,-1cm)},y={(-1cm,0)}]
   \draw[very thick]
      (-1,1)--(4,1)--(4,-2)--(1,-2)--(1,-3.5)--(-1,-3.5)--(-1,1);
  \draw
      (4,1)--(6,1)--(6,0.5)--(5,0.5)--(5,0)--(4.5,0)--(4.5,-0.5)--(4,-0.5)--(4,1);
 \draw
       (4,-1)--(2,-1)--(2,-1.5)--(1.5,-1.5)--(1.5,-2);
 \draw
       (0.5,-3.5)--(0.5,-4)--(0,-4)--(0,-4.5)--(-0.5,-4.5)--(-0.5,-5)--(-1,-5)--(-1,-3);
       \draw(3,-1.5) node {$\gamma$};
          \draw(4.6,0.5) node {$\delta''$};
           \draw(-0.5,-4) node {$\delta'$};
       \end{tikzpicture}\end{minipage}
      \begin{minipage}{32mm}\begin{tikzpicture} [scale=0.4,x={(0,-1cm)},y={(-1cm,0)}]
   \draw[very thick]
      (-1,1)--(4,1)--(4,-2)--(0,-2)--(0,-4.5)--(-1,-4.5)--(-1,1);
  \draw
      (4,1)--(6,1)--(6,0.5)--(5,0.5)--(5,0)--(4.5,0)--(4.5,-0.5)--(4,-0.5)--(4,1);
 \draw
       (4,-1)--(2,-1)--(2,-1.5)--(1.5,-1.5)--(1.5,-2);
 \draw
        (1,-2)--(1,-3.5)--(0.5,-3.5)--(0.5,-4)--(0,-4);
       \draw(3,-1.5) node {$\gamma$};
          \draw(4.6,0.5) node {$\delta''$};
           \draw(0.5,-3) node {$\delta'$};
       \end{tikzpicture}\end{minipage}
                      \begin{minipage}{32mm}\begin{tikzpicture}[scale=0.4]
 \draw[very thick]
      (-1,1)--(4,1)--(4,-2)--(2,-2)--(2,-5)--(-1,-5)--(-1,1);
  \draw
      (4,1)--(5.5,1)--(5.5,-0)--(4.5,-0)--(4.5,-0.5)--(4.5,-0.5)--(4.5,-1.5)--(4,-1.5)--(4,-2)--(3.5,-2)--(3.5,-2.5)--(3,-2.5)--(3,-3)--(3,-3)--(2,-3);
 \draw
       (2,-3.5)--(.5,-3.5)--(.5,-4)--(-0.5,-4)--(-0.5,-4)--(-0.5,-5);
        \draw(1,-4.5) node {$\gamma$};
          \draw(2.65,-2.5) node {$\delta''$};
 \draw(5,0.5) node {$\delta'$};
           \draw[white] (5,-5.25) node {$\delta$};
        \end{tikzpicture}\end{minipage}
     \begin{minipage}{32mm}\begin{tikzpicture}[scale=0.4]
 \draw[very thick]
      (-2,0)--(4,0)--(4,-2)--(2,-2)--(2,-5)--(-2,-5)--(-2,0);
  \draw
      (4,0)--(5,0)--(5,-0.5)--(4.5,-0.5)--(4.5,-2.5)--(4,-2.5)--(4,-3)--(2,-3);
  \draw
       (2,-3.5)--(.5,-3.5)--(.5,-4)--(0,-4)--(0,-4)--(0,-5);
   \draw
  (-0.5,-5)--(-0.5,-5.5)--(-1,-5.5)--(-1,-6)--(-2,-6)--(-2,-5);
        \draw(1,-4.5) node {$\gamma$};
          \draw(3,-2.5) node {$\delta'$};
                    \draw(-1.5,-5.5) node {$\delta''$};
       \end{tikzpicture}\end{minipage}
     $$
    $$\scalefont{0.6}
      \begin{minipage}{34mm}\begin{tikzpicture}[scale=0.4]
   \draw[very thick]
      (-1,1)--(4,1)--(4,-2)--(2,-2)--(2,-5)--(-1,-5)--(-1,1);
  \draw
      (4,1)--(7,1)--(7,0.5)--(6,0.5)--(6,0)--(5.5,0)--(5.5,-0.5)--(4,-0.5)--(4,1);
  \draw
        (2,-3.5)--(1.5,-3.5)--(1.5,-4)--(0,-4)--(0,-5);
         \draw
        (4,-1)--(3,-1)--(3,-1.5)--(2.5,-1.5)--(2.5,-2);
             \draw(1,-4.5) node {$\gamma'$};
                          \draw(3.5,-1.5) node {$\gamma''$};
          \draw(5,0.25) node {$\delta$};
       \end{tikzpicture}\end{minipage}
            \begin{minipage}{34mm}\begin{tikzpicture}[scale=0.4]
   \draw[very thick]
      (-1,1)--(7,1)--(7,-1.5)--(2,-1.5)--(2,-5)--(-1,-5)--(-1,1);
  \draw
       (3.5,-3+1.5)--(3.5,-3+0.5)--(3,-3+0.5)--(3,-3+0)--(3,-3+0)--(3,-2.5+-0.5)--(-2+4,-2.5+-0.5)--(-2+4,-3+1);
  \draw
        (2,-3.5)--(1.5,-3.5)--(1.5,-4)--(0,-4)--(0,-5);
          \draw
         (7,0)--(6,0)--(6,-0.5)--(4,-0.5)--(4,-1.5);
             \draw(1,-4.5) node {$\gamma'$};
                          \draw(2.7,-2.17) node {$\delta$};
          \draw(6.25,-0.75) node {$\gamma''$};
       \end{tikzpicture}\end{minipage}
     $$
\caption{The possible intersection diagrams  for which  $\mu$ is a proper fat hook and
one of $\gamma$ and $\delta$    has two connected components.
 We label these diagrams by
$(2a)$, $(2b)$, $(2c)$, $(2d)$  $(2e)$  and $(2f)$ respectively.  }
\label{singlefatty2}
\end{figure}
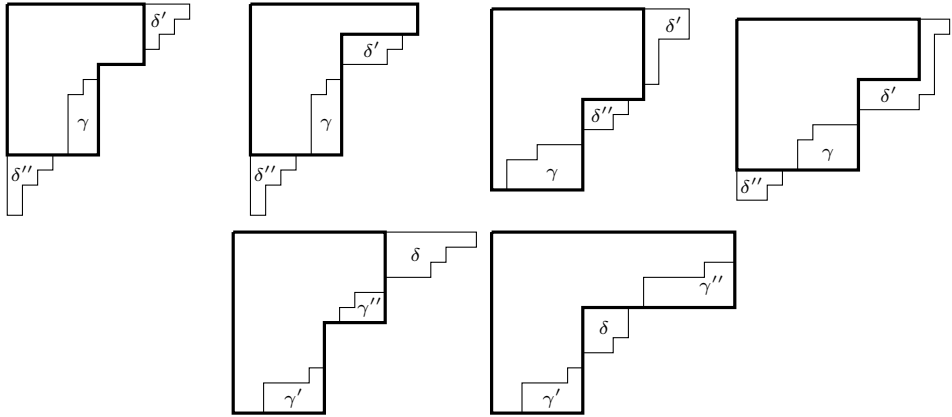

\begin{lem}\label{fathooklemma2rec}
If $\gamma$ and $\delta$ are both rectangles, then $g(\la,\mu)>1$.
\end{lem}

\begin{proof}
We invite the reader to check the cases
where the size of the partitions $\gamma$ and $\delta$
is at most 2 by hand.
These can   easily be reduced to    small cases
 (however,  listing them is a somewhat tedious exercise). One can easily show that these products contain multiplicities using simplifications of the arguments used here
 (or these can be checked by computer as the degrees of the partitions are small).
We shall assume throughout that $\gamma$, $\delta$ are of size strictly
greater than~2.
Assuming $\gamma$ and $\delta$ are both rectangles,   cases $(1a)$ and $(1b)$ are empty.

 We first consider case $(1c)$. If
    $\delta=(k)$ and
 and $\gamma=(k)$ (and by assumption
$a-c\geq 1$ and  either $b$ or $d$ is strictly greater that 1)  we
remove all but one column in the arm and
 all but one row common to both partitions and hence
arrive at
  $(\tmu, \tla) $ equal to one of the following subcases
$$
((2k+1,k),(k+1,k^2))
\; , \;
((2k+1,k+1), ((k+1)^2,k)).
$$
  If
    $\delta=(k)$ and
 $\gamma=(1^k)$ (and by assumption
  $a-c\geq 1$ and
   either $b$ or $d$ is strictly greater than~1)  we remove  most rows and columns in order to
arrive at
$$
((3+k,1^k),
(3,2^k))
\; , \;
((k+2,2),(2^2,1^k)).
$$
  In all four of the above subcases, we have that $g(\la,\mu)\geq g(\tla,\tmu)>1$ by \cref{section2parter,sectionhook!}.

 Now assume that $ \delta =(k^2)$ and $ \gamma=(1^{2k})$, with $k>1$.
By \cref{animportantremark} and our assumption that $a>2$, we can remove all rows and all but  two of the columns common to both partitions  until we obtain $(\tla,\tmu)$ equal to one of the following subcases
 $$
( (3^2,1^{2k}),((k+3)^2))
 \; , \;
( (3^2,2^{2k}),((k+3)^2,1^{2k}))
 $$
 the latter case follows by \cref{section2parter}.
In the former case,
    remove the final two rows of~$\tla$ and
   the final column of~$\tmu$
 to obtain  $(((k+3)^2,1^{2k-2}),(2^{2k+2}))$; the result follows by \cref{section2parter}.
By \cref{animportantremark}, if $\gamma=\delta=(1^3)$ we can reduce using the semigroup property to the seed $((3^3), (3^3))$.

We now  consider the generic case (not of the above form) for $(1c)$.
   Remove all rows and all columns common to both $\la$ and $\mu$ with the exception of  one column from the arm.  We hence
  obtain $\tmu \cap \tla=((w(\gamma)+1 )^{\ell(\delta)})$.
  If $w(\gamma)=1$ and  $\ell(\delta)\geq 3$ then the result follows from \cref{section2parter} (the $\ell(\delta)\leq 2$ case was covered above).
  If $w(\gamma)>1$ and $\ell(\delta)=2$ then  the result follows from \cref{section2parter}.
   If $w(\gamma)>1$ and $\ell(\delta)>2$  then $\gamma$ and $\delta$ are $(SG)$-removable and $(3^3)\subseteq \tmu\cap\tla$ and the result follows by \cref{squares}.

We now consider case $(1d)$; there are two subcases.  If $\delta\neq (1^k)$, then we  remove all common columns from the arms of $\mu$ and $\lambda$ until we obtain the partitions
$\tmu=(c^{b+d})$ and $\tilde\lambda $ a proper fat hook. If $\delta=(1^k)$, then we   remove common columns and rows until we obtain $\tmu=(w(\gamma)+2,(w(\gamma)+1)^{\ell(\gamma)})$ and $\tla=(w(\gamma)+2, 1^{k+\ell(\gamma)}  )$.
 The result follows by \cref{section2parter,sectionhook!}.

We now consider case $(1e)$.  By \cref{animportantremark}, if
$w(\gamma)=\ell(\delta)=1$  we can remove successive rows and columns
from $\mu$ and $\lambda$ until we obtain
$(\tmu,\tla)$
equal to one of the following pairs:
 $$  (((k+1)^2,1^{k+1}), ((k+1)^3) ) \; . \;
 ((k+2,2^{k+1}), ((k+2)^2,1^k)).
$$
The first  (respectively second) case follows by \cref{sec:rectangle} (respectively \cref{squares}).
By \cref{animportantremark}, if  $ \gamma =(1^k)$ and $\delta=(1^k)$  we can remove all but one  row in the arm and all but one common column
  to obtain
$(\tmu,\tla)
$ equal to one of the  pairs:
 $$
 ((3,1^{2k}), (3,2^k))\; , \; ((3,2^{2k}), (3^{k+1},1^k)).
 $$
 The latter case follows by \cref{section2parter}.
 The former can be further reduced to  the seed
  $
 ((3^2,1^{2}), (3^2,2)).
  $
By \cref{animportantremark}, if  $ \gamma =(1^{2k})$ and $\delta=(2^{k})$  we   can
 successively remove common  rows and columns
 until we obtain
$(\tmu,\tla)
$ equal to one of the following pairs:
$$
((3^2,1^{3k}),(3^{k+2})) \; , \;
((2^{3k}) , (4^k,1^{2k}))
$$
 and  the result follows by \cref{sec:rectangle} and \cref{section2parter} respectively.
  We now consider the generic case for $(1e)$. Remove all rows and  columns  common to both
$\mu$ and $\lambda$ with the exception of one row in the arm
 to obtain
$\tmu=((w(\delta)+w(\gamma),(w(\gamma))^{\ell(\gamma)+\ell(\delta)})$ a proper fat hook  and $\tla=((w(\delta)+w(\gamma))^{\ell(\delta)+1})$ a non-linear rectangle.
The result follows by \cref{sec:rectangle}.    \end{proof}

\begin{lem}\label{7.2}
If either  $\gamma$ or $\delta$ is linear and the other is connected, then $g(\la,\mu)>1$.
\end{lem}

\begin{proof}
We may assume that one diagram is linear and the other is not a rectangle,
as the case of two rectangles has already been addressed in \cref{fathooklemma2rec}.

We first consider cases $(1c,d,e)$ with $\gamma$ a linear partition.

Suppose we are in case $(1c)$ with $\gamma=(k)$.
If  $\ell(\delta)=2$,   remove all rows and columns common to $\la$ and $\mu$ with the exception of one column in the arm to obtain $\tmu$
 a proper fat hook and $\tla\supset (4^2)$ a 2-line partition.
 Hence the result follows by \cref{section2parter}.
If  $\ell(\delta)>2$
  remove all rows and columns common to both  $\mu$ and $\lambda$
 to obtain $(\tilde\mu,\tilde\la)$ such that
  $(4^4) \subseteq \tilde\mu$ a rectangle
    and $(4^3) \subset \tilde \lambda$.
     The result follows by \cref{sec:rectangle}.
Now consider   case $(1c)$   with $\gamma=(1^k)$.
      Remove all  rows and columns
       shared by $\mu$ and $\lambda$
       with the exception of
 one column  in the arm  to obtain $\tmu$ and $\tla$
 such that $\tmu\supseteq (2^2,1^3)$ is a 2-line partition  and
 $\tla\supset (3^2)$ is a non-rectangular partition.  The result follows by \cref{section2parter}.

For  case $(1d)$ with $\gamma$ linear,  we remove all rows and columns common to both $\la$ and $\mu$ with the exception of one row from the  arm
to obtain $\tmu$ a non-linear rectangle  and $\tla$ such that  $|\Remm(\tla)|\geq 3$.
The result follows by \cref{sec:rectangle}.

 In case $(1e)$ and $\gamma=(k)$ with $k>3$,   remove all rows and columns
 common to both  $\mu$ and $\lambda$ to obtain $\tmu\supseteq (4^3)$ a  rectangle and $\tla \supset (3^2)$ a non-rectangular partition; for $k=3$ we reduce to the seed $((5^2,4), (5,3^3))$.
  In case $(1e)$ and $\gamma=(1^k)$,   remove all rows and columns  common to both $\la$ and   $\mu$ with the exception of  one row in the arm
 to obtain $( \tmu ,\tla)$. We have that  $\tmu\supseteq (3,1^5)$ is a hook partition and  $\tla\supseteq (3^2,2)$; the result follows by \cref{sectionhook!}.

We now consider cases $(1c,d,e)$ for $\delta$ a linear partition.
   Recall that $\gamma^c=((w(\gamma))^{\ell(\gamma)}) / \gamma$.

Assume we are in case $(1c,e)$ with $|\gamma^c|> 2$.
If $\delta=(k)$ for case $(1c,e)$    and   $w(\gamma)>2$, remove all rows and columns common to $\la$ and $\mu$ to obtain $\tmu$ a fat rectangle and
  $\tla$ such that
 $\tla_1\ge w(\gamma)+4$ and $\tla$ is of depth at least~3.
Therefore the result follows by \cref{induction}
  for cases $(1c,e)$ with $\delta=(k)$,  $w(\gamma)>2$,   and $|\gamma^c|> 2$.

 Continuing with case $(1e)$ with $\delta=(k)$, we  now  assume that  either   $w(\gamma)=2$ or   $|\gamma^c|\leq 2$. In either case,  remove all rows and columns common to both $\la$ and $\mu$ with the exception of one row in the arm to obtain a pair of proper fat hooks of the form
 $$(\tla,\tmu)=
  \left(
 \Big(
  ( 	 w(\gamma)+|\gamma|  )^{2}, \gamma^c
\Big)
 ,
\Big( w(\gamma) +|\gamma|  , w(\gamma)^{\ell(\gamma)+1}
\Big)
 \right).$$
We have that $|\gamma^c| < |\gamma|$ by our assumption that   $w(\gamma)=2$ or   $|\gamma^c|\leq 2$.
By the semigroup property, we can reduce to
  $(\tla,\tmu)=
 (
 (
  ( 	 w(\gamma)+|\gamma|  )^{2}
 )
 ,
 ( w(\gamma) +|\gamma| -|\gamma^c| , w(\gamma)^{\ell(\gamma)+1}
 )
 ) $ and the result follows by \cref{section2parter}.

 Continuing with case $(1c)$ with $\delta=(k)$, we  now  assume that  either    $w(\gamma)=2$ or $|\gamma^c|\leq 2$.  Remove all rows and columns common to $\la$ and $\mu$ with the exception of either one arbitrary row, or one column in the leg.  We hence obtain $\tmu$ a  rectangle and $\tla$ with  at least three removable nodes; the result follows by \cref{sec:rectangle}.

We now consider cases $(1c,e)$ with $\delta=(1^k)$.
  Remove all rows and columns common to both partition to obtain $\tmu$
a non-linear  rectangle and $\tla\supseteq(3^3,1)$;  the result follows by \cref{sec:rectangle}.

For case $(1d)$ with $\delta$ linear,
remove all columns and
all but one row common to both $\la$ and $\mu$
to obtain $\tmu$ a fat rectangle and  $\tla$ a partition with at least 3 removable nodes.
The result follows by \cref{sec:rectangle}.

We now consider  case $(1b)$; here, only $\gamma$ can be linear.
If $\gamma=(k)$    remove all common rows and columns to obtain $\tmu$ and $\tla$.  If $\ell(\delta)=2$ then the result follows by \cref{section2parter}.
Suppose that  $\ell(\delta)\geq3$.
 The shortest row of $\tmu$    is longer than the longest column in $\tla$
and so     $  \tla^t\cap\tmu $ is a   rectangle.  By assumption, $\ell(\la)\geq 3$ and so
$  \tla^t\cap\tmu \supseteq (3^4)$
  and  the    result follows by \cref{squares}.

We now consider the case  $(1b)$ with $\gamma=(1^k)$.
 The exceptional cases  are $(i)$    $(a-c)b\leq 2$ and
$(ii)$
  $\ell(\delta)=2$.
In either  case,  remove all rows and columns with the exception of one column in the leg (which exists by assumption that $\mu$ is neither a hook, nor a 2-line partition) to obtain $(\tla,\tmu)$.
We have that $(a-c)b<k$ by assumption and so we can remove the final $(a-c)$ columns of $\tmu$ and the final
$(a-c)b$ rows of $\tla$ to obtain $\hat\mu = (2^{k+\ell(\delta)})$ and $\hat\la\supset (3,2,1)$.
 The result follows by \cref{section2parter}.

 Now suppose we are in case  $(1b)$ with $\gamma =(1^k)$ and we are not in one of the exceptional cases $(i)$ and $(ii)$ above.  Remove all common rows and columns from $\mu$ and $\la$ to obtain $\tmu$ and $\tla$.
  If $\tmu$ is a 2-column partition, the result follows.
 Otherwise,   remove all nodes
    in $\tla$ to the right of the final column of $\tmu$
     and remove the corresponding number of nodes from the first column of $\tmu_1$ to obtain a pair $(\hat{\la},\hat{\mu})$.
 We have that
     $\hat{\delta}= \hat{\lambda} /  (\hat{\lambda}\cap \hat{\mu}) $
    is a proper partition  and
     $ \hat{\gamma}=\hat{\lambda} /  (\hat{\lambda}\cap \hat{\mu})$ is linear.
    The result follows from the case    $(1e)$ for $\hat{\delta}$ a proper partition, above.

 Finally, suppose we are  in case  $(1a)$; here only $\delta$ can be linear.
  If $\gamma$ is a proper partition,  remove all common rows (or all common columns, respectively) from
  $\la$ and $\mu$ to obtain $\tla$ and $\tmu$.
   The partitions $\tla^t$ and $\tmu^t$ are now as in  case $(1e)$   (respectively $(1c)$) above and therefore $g(\la,\mu)>1$.

  It remains to consider the case when $\gamma$ is a proper skew partition.

Case $(i)$.      If  $\gamma=\rho/(1)$  and $\delta=(k)$,   remove all but one row (in the arm)
      or one column (in the leg) to obtain a pair of partitions $(\tla,\tmu)$.
      In the former case, we remove successive rows from $\tmu$ (and the corresponding number of nodes from the first row of $\tla$) until we obtain $\tmu$ a fat rectangle  and $\tla$ such that $|\Remm(\tla)|=3$.  The result follows by \cref{sec:rectangle}.
      In the latter case, remove the final row of $\tmu$ and the corresponding number of nodes from the  first row of $\tla$ to obtain a pair $(\hat\la,\hat\mu)$.
      If $\hat\mu$ is a rectangle the result follows.
      Otherwise,
      $\hat\mu / (\hat\la\cap \hat\mu)$ is a proper skew partition and
      $[\hat\la / (\hat\la\cap \hat\mu)]=[k']\boxtimes [1]$ with $k'<k$,
      and the result follows from \cref{induction}.

Case $(ii)$.  Now assume
 $\gamma\neq \rho/(1)$ and $\delta=(k)$ and remove all rows and columns common to $\la$ and $\tmu$ to obtain $\tla$ and $\tmu$.
    If $\tla$ is a hook or 2-line partition the result follows.
Otherwise,   if $\gamma=(\rho /\sigma)^{\rm rot}$ for $|\Remm(\rho)|=3$ (respectively 2) remove all rows in
  $\tmu$ which occur below the final row of $\tla$ and remove the corresponding number of nodes from the first row of $\tla$ to obtain $\hat\la$ and~$\hat \mu$.
    We have that $|\Remm(\hat\la)|=2$ (respectively 3)
      and $\hat\mu$ is either a rectangle  or a fat hook
      such that $ \hat\mu/  \hat\mu \cap \hat\la  $ is   a proper partition
       (respectively  $ \hat\mu/  \hat\mu \cap \hat\la =(\hat\rho /\hat\sigma)^{\rm rot}$ for $|\Remm(\hat\rho)|=2$).

 In the former case, the result follows either from \cref{sec:rectangle} or from noting that
   $(\hat\la,\hat\mu)$ are as in case $(1a)$ for $\gamma$ a proper partition.
 In the latter case, repeat the above argument for case $(i)$ or case $(ii)$ as appropriate.

 Finally  assume  $\delta=(1^k)$ in case $(1a)$.
 Remove all rows and columns common to both $\la$ and $\mu$ to obtain a pair $(\tla,\tmu)$.
 If $w(\gamma)=2$, then $\tla$ is a proper fat hook and $\tmu$ is a 2-line partition and so the result follows by \cref{section2parter}.
  Otherwise, by our assumptions  $k\geq 4$  and $3\leq w(\gamma) < k$.
  The shortest column of $\tla$ (which is of length equal to $k$) is longer than
  the widest row of $\tmu$ (equal to $w(\gamma)$)  and so
  $  \tla^t\cap \tmu =(w(\gamma)^k)\supseteq (3^4)$ and so the result follows by \cref{squares}.
   \end{proof}

\begin{lem}\label{42879342578837495278354278354}
If either  $[\gamma]$ or $[\delta]$ is equal to   $[k-1]\boxtimes [1]$ up to conjugation, then $g(\la,\mu)>1$.
\end{lem}

\begin{proof}
If $[\gamma]$ or $[\delta]$  is of the form
 $ [1] \boxtimes [k-1]$  up to conjugacy,
 we may assume  that the other is  a   rectangle by \cref{induction}.
   It is easy to see that case $(2d)$ is never of this form.
   We first consider the pairs of partitions $(\tla,\tmu)$   which form our exceptional cases, in which it is not possible to remove all rows and columns common to both partitions $\la,\mu$.

  In case $(2a)$, suppose that $\gamma$ is linear and
  $[\delta] =[1] \boxtimes [k-1]$.
  By  \cref{animportantremark}, we can remove most  rows and columns common to $\mu$ and $\lambda$
   to obtain
 $(\tmu,\tla)$ equal to one of the following
 $$
( (3,2^k), (2+k,1^{k+1}))
 \;,\;
 ((3,2^k), (4,1^{2k-1}))
 \;,\;
((k+1)^3)), ( (2k,k+1,1^2)	
 \;,\;
((k+1)^3)), ( (k+2,k+1,1^k).
  $$
Otherwise, remove all rows and columns common to both $\la$ and $\mu$ to obtain $(\tla,\tmu)$.
 The result follows by \cref{sectionhook!,sec:rectangle}.

We now consider case $(2b)$ (in case $(2f)$ one can use the semigroup property to  reduce to the  same set of cases, and we therefore do not consider this case explicitly).
 Suppose that $\gamma$ is linear and $[\delta] =[1] \boxtimes [k-1]$ up to conjugation.  The exceptional cases
 are precisely those in which $\ell(\delta')=w(\delta'')=1$  (with notation
 as in case $(2b)$ of \cref{singlefatty2}) and $\gamma$ is linear.
Remove all rows and columns common to both $\mu$ and $\lambda$ with the exception of one row in the arm  to obtain
 $(\tmu,\tla)$ equal to one of the following up to conjugation
 $$
 ((3,2^{k+1}),(3^2,1^{2k-1}))
\;,\;
((k+1,2^{k+1}), ((k+1)^2,1^{k+1}))
$$
$$
((k+2,(k+1)^2),
((k+2)^2,1^k))
\;,\;
(
(2k,(k+1)^2),
((2k)^2,1^2)
).
 $$
 In each    case  we can remove a single node from the first row of $\tmu$ and
 a single node from the first column of $\tla$ to obtain a pair $(\hat\mu,\hat\la)$.
 In the first  case
 $\hat\mu=(2^{k+2})$
 and $\hat\la$ is a proper fat hook
  and the result follows by \cref{section2parter}.
   In the third  case
  $\hat\mu$ is a fat rectangle and $ \hat\la $ is a proper fat hook and the result follows by \cref{section2parter}.
  In the second and fourth cases
with $k>2$ (the $k=2$ cases are covered by the first and third cases)
 $\hat{\mu},\hat{\la}$   are both  proper fat hooks
  and  $\hat{\mu}/  (\hat{\la}\cap \hat{\mu})$ is linear and
$[\hat{\la}/  (\hat{\la}\cap \hat{\mu})]$ is the standard character and so the result follows by \cref{7.2}.

The only exceptional case for $(2c)$ is when     $\ell( \delta')=w(\delta'') =1$ and $\gamma=(1^k)$.
Then remove all rows and columns common to both  $\la$ and $\mu$ with the exception of one column in the arm or leg
  (which must exist as $\mu$ is not a 2-line partition) to obtain $\tla$ and $\tmu$.
 In the former case
 $\tla$ is a proper fat hook and
 $ \tmu \supset (3,1^2)$
 is a  hook partition; the result then follow by \cref{sectionhook!}.
 In the latter case remove a single node from the first column of $\tla$ and the first row of $\tmu$; the result then follows by \cref{section2parter}.

 The only exceptional case for $(2e)$  is that in which  $w(\gamma')=\ell(\gamma'')=1$
and   $\delta=(k)$.
 Remove all rows and columns common to both  $\la$ and $\mu$ with the exception of one row or column in the arm or in the leg
   to obtain $\tla$ and $\tmu$.
For a single row in the leg (respectively column in the arm)
 the result then follows by \cref{sectionhook!} (respectively \cref{section2parter}).
  For a single row in the arm, remove a single node from the first column of $\tla$ and the first column of $\tmu$; the result then follows by \cref{section2parter}.
  For a single column in the leg, remove the final row of $\tmu$ and the final two columns of $\tla$ to obtain
  $(\hat\la,\hat\mu)$.  For $k>2$, both $\hat\la/\hat\la\cap\hat\mu$ and $\hat\mu/\hat\la\cap\hat\mu$ have two connected components and the result follows by \cref{induction}; for $k=2$, we have the seed $((5,2,1),(3^2,2))$.

 Now suppose that we are in one of the cases $(2a,b,c,e,f)$
  and $(\gamma,\delta)$  is not one of the exceptional cases
  (all of which were dealt with above).
 In cases $(2a,b,f)$ we remove all rows and columns  common to both $\mu$ and $\lambda$ to obtain a pair $(\tla,\tmu)$ where $\tmu $ or $\tla$
   is a proper rectangle,   and which  is not on our list.
  In case $(2c)$ we remove all common rows and columns from $\la,\mu$
  and obtain either a 2-line partition $\tla$ with $(\tla,\tmu)$ not on our list, or a pair which can be reduced in one further step to a pair not on our list where at least one is a proper rectangle.
  In case $(2e)$ again remove all common rows and columns
  and obtain either a 2-line partition $\tla$ with $(\tla,\tmu)$ not on our list, or a pair $(\tla,\tmu)$ where we can remove a shape corresponding to $\delta$ from $\tla$ and the final boxes from the $k$ columns of $\tmu$, and
  $g(\tla\cap \tmu,\tla\cap\tmu)>1$.
So the result follows from \cref{section2parter}, \cref{squares} and \cref{sec:rectangle}.
   \end{proof}

\begin{lem}\label{7.4}
If either  $\gamma$ or $\delta$ is  a proper hook partition up to rotation, then $g(\la,\mu)>1$.
\end{lem}

\begin{proof}
 By
 \cref{sec:warm-up} and \cref{dvirmaximal}
 we may assume that up to rotation
 one of $\gamma,\delta$ is a proper hook and the other is a fat hook.
 By \cref{7.2} we may assume that neither partition is linear.

Assume  $(\la,\mu)$ are  as in  $(1a)$ of \cref{singlefatty}.
 We begin with the case $\gamma=(k-1,1)$ (the case $\delta=(2,1^{k-2})$ is identical).
For $k=3$  we remove all rows and columns common to both partitions with the exception of  one column in the leg to obtain the seed  $(\tmu,\tla) =((3^3,2),(5,4,1^2))$.
 For $k>3$, we remove all rows and columns common to both $\la,\mu$ to
 obtain $(\tmu,\tla)$.  If $\ell(\tla)=2$, then the result follows from \cref{section2parter}.
 If $\ell(\tla)>2$, then $\gamma$ and $\delta$ are $(SG)$-removable and $(3^3)\subseteq \tmu\cap\tla$.

 Now assume that  $\gamma=(2,1^{k-2})$, with $k>3$.
 If $k$ is even and $\delta=(k/2,k/2)$, then we remove  all rows and columns common to both partitions with the exception of   one   column in the leg
  to obtain $\tmu$ and $\tla$.   We then remove $(k/2,k/2)$
  from the top of $\tla$ and $(2^{k/2})$ from the bottom of $\tmu$
  to obtain $((3^3,2^{k/2-2}),(3^2,1^{k-1}))$ and then the result follows by \cref{7.2}.
 For $\delta$   not of the above form,
   remove all rows and columns common to both partitions  to  obtain $(\tla,\tmu)$ such that
  $\tmu$ is a 2-line partition and $(\tmu,\tla)$ is not a pair listed in \cref{thm:classification};
    the result follows by \cref{section2parter}.

    If $\delta=(k-1,1)$ and $\gamma$ is not of the above form,  remove all rows and columns common to $\la$ and $\mu$ to obtain $\tla=(w(\gamma)+k-1,w(\gamma)+1)$ and $\tmu$ a proper fat hook.  The result follows from \cref{section2parter}.

 We   now assume that $\gamma,\delta\neq (k-1,1)$  up to conjugation.
 By \cref{sec:warm-up} and \cref{dvirmaximal}, being in case $(1a)$ implies
 that $\gamma$ is a proper hook and $\delta$ is a non-linear rectangle.
 Remove all rows and columns common to $\mu$ and $\lambda$ to obtain $\tmu$ a proper fat hook
  and $\tla$  a rectangular partition; the result follows then by \cref{sec:rectangle}.

Before addressing  case $(1b)$ of \cref{singlefatty}, we first consider cases $(1c,d,e)$.
Assume that $\delta$ is a proper hook.
  In case $(1c)$, there is a single exceptional subcase, where $\gamma=(2^k)$ and $\delta=(2k-1,1)$; here we  remove all rows and columns common to both partitions
 with the exception of one column in the arm   to obtain
   $(\tla,\tmu)=((2k+2,4),(3^2,2^k))$;   the result follows by \cref{section2parter}.
   In case $(1d)$, the unique  exceptional subcase is
   $(\gamma,\delta)=((2^2),(2,1,1))$, which we can reduce to  the seed  $(\tla,\tmu)=((4,2^3,1^2),(4^3))$.
    In case $(1e)$, the single exceptional subcase is given by
     $\gamma=(2^k)$ and $\delta=(2k-1,1)$;
 remove all rows and columns common to both partitions
 with the exception of one row in the arm   to obtain
     $(\tla,\tmu)$.  In which case $(\tla^t,\tmu^t)$ is equal to a pair of partitions as in the second   exceptional case for $(1a)$, above.

 Continuing with      $(1c,d,e)$ with $\delta $ a proper hook, we now argue for the generic case.  Remove all rows and columns common to $\la$ and $\mu$ to reduce to a pair of partitions $(\tla,\tmu)$ such that $\tmu$ is a rectangle and
      $(\tla,\tmu)$ does not belong to our list.  Thus the result follows
      by \cref{sec:rectangle}.

Suppose that  we are in cases $(1c,d,e)$ and that  $\gamma$   is a rotated proper hook.   Remove
all rows and columns common to both partitions
 to obtain  $(\tla, \tmu)$
such that
$\tmu$
 is a rectangle and $\tla$
  is a proper fat hook or has three removable nodes (as $\delta$ is non-linear).
  The result  follows from \cref{sec:rectangle}.

Finally, we consider case $(1b)$.
 If $\delta^{\rm rot}$ is a proper hook and $\gamma$ is a non-linear rectangle,  remove all rows and columns common to $\la,\mu$ to obtain
 $(\tla, \tmu)$
such that
$\tla$
 is a rectangle and $\tmu$
  is a proper fat hook; the result follows then by \cref{sec:rectangle}.
  We may now assume that one of  $\gamma^{\rm rot}$ or $\delta^{\rm rot}$ is equal to $(k-1,1)$ up to conjugation and the   other   is a non-rectangular  fat hook.
   This case is symmetric in swapping $\gamma$ and $\delta$ and therefore we can assume that $\delta^{\rm rot}=(k-1,1)$
  up to conjugation
   and $\gamma^{\rm rot} =(t^u,v^w)\vdash k$ is not a rectangle.   Remove all rows and columns common to $\la,\mu$
 to obtain    $(\tla,\tmu)$ equal to either
    $$ ( (t+k-1)^2, (t-v)^w), (t+k-2,t^{u+w+1}))
   \; , \;
   ((t+2)^{k-1},  (t-v)^w),
   ((t+1)^{k-2} , t^{u+w+1}).
   $$
  The $k=3$ case is the seed $(4^2,1),(3,2^3)$.
  For $k>3$   in the latter case,
  if $\gamma^{\rm rot}$ is of depth at least $4$,
   then
 $\tla \cap \tmu^t =((t+2)^{t+1})$ and so
 $\tla/ (\tla \cap \tmu^t )$
 and
  $\tmu^t/ (\tla \cap \tmu^t )$ are $(SG)$-removable and
  $g(\tla \cap \tmu^t,\tla \cap \tmu^t)>1$.
If  the depth of $\gamma^{\rm rot}$ is smaller than $4$, and $\gamma^{\rm rot} \neq (2^2,1) $    then
 the sum of the first and final columns in $\tla$ is equal to the
  sum of the first and final columns in $\tmu$
  (equal to $2k-1$ in both cases).
Now if  $\gamma^{\rm rot}$ is not one of $(2,1^2),(2,1^3),(2^2,1) $,
we remove these columns and obtain
$\hat{\mu}$ a non-linear rectangle, and
$\hat\la$ a proper fat hook; the result follows by  \cref{semigroup,sec:rectangle}.
If  $\gamma^{\rm rot}= (2^2,1) $, we reduce to the seed
$((3^4,1),(3^3,2^2))$.
If $\gamma^{\rm rot}= (2,1^2)$, we remove the final two rows from $\tla$ and the final row from $\tmu$, giving a rectangle and a proper fat hook;
the result follows by  \cref{semigroup,sec:rectangle}.
If $\gamma^{\rm rot}= (2,1^3)$, we remove the final two columns from $\tla$ and the first column from $\tmu$, giving a pair of 2-line partitions not on our list, so the result follows.

 In the former case with $k>3$,   remove the final column of  $\tla$
  and the final two columns  from $\tmu$ to obtain  $\hat\la$ and $\hat\mu$ such that
  $\hat\la / \hat\la\cap \hat\mu=(k-4,2)^{\rm rot}$ and   $\hat\mu / \hat\la\cap \hat\mu=\gamma$.
 By \cref{induction}, if $\gamma^{\rm rot}$ is not equal to $(k-1,1)$ up to conjugation, we are done.
 If   $\gamma^{\rm rot}=(2,1^{k-2})$ then $g(\tla,\tmu)>1$ by  \cref{squares} and if   $\gamma^{\rm rot}=(k-1,1)$, then the result follows by conjugating to the latter case, discussed above.
   \end{proof}

\begin{lem}
If either  $\gamma$ or $\delta$ is  a 2-line partition, then $g(\la,\mu)>1$.
\end{lem}

\begin{proof}
  By \cref{induction} and the previous results in this section, it will suffice to consider  $ \gamma$ and $\delta  $ such that up to conjugation
 \begin{itemize}[leftmargin=0pt,itemindent=1.5em]
\item  one is equal to $(k,k)$ and the other is $(k+1,k)/(1)$;
 \item   the pair is equal to one of the special pairs  $((3^3),(6,3))$ or $((3^3),(5,4))$;
 \item  one is equal to $(k-2,2)$ and the other is a rectangle;
\item
the pair is equal to one of   $\{(k+1,k),(k+1,k)\}$,  $\{(k^2),(k+1,k-1)\}$,
or $\{(k^2),(2k-3,3)\}$.
   \end{itemize}

 We  consider the proper skew partition case first.
 We assume without loss of generality that we are in case $(1a)$ (case $(1b)$ is identical and  such a pair $\gamma$ and $\delta$ cannot occur in cases $(1c,d,e)$).
 Remove all rows common to $\la$ and $\mu$ to obtain $(\tla,\tmu)$ equal to one of
 $$
( ((k+3)^k,1)
 ,
 ((k+1)^{k+1},k))
\;,\;
(
 ((2k+1)^2,1)
 ,
 ((k+1)^3,k)
 )\;,\;
(
 ((k+2)^2,1)
 ,
 (2^{k+2},1 )
 )
 \;,\;
 ((4^k,1), (2^{2k},1)).
 $$
 In the first   case, we have that
  $\tla^t / \tla^t\cap \tmu$ and
$\tmu / \tla^t\cap \tmu$ are both linear and so the claim
follows from \cref{7.2}.
 In the second  case, we have that
  $\tla^t / \tla^t\cap \tmu=(2^{(k-1)})$  and
$\tmu / \tla^t\cap \tmu = ((k-1)^3,k-2)/(1)$ and so the result follows for $k>3$ by \cref{induction};
at $k=2$ we have the seed $((5^2,1),(3^3,2))$, to which we also
easily reduce in the case $k=3$.
  The third and fourth cases follow from \cref{section2parter}.
For the remainder of the proof, we assume that $\gamma$ and $\delta$ are both proper partitions (up to rotation)
 and proceed case-by-case through $(1a)$ to $(1e)$.

 We first consider  case $(1a)$ depicted in \cref{singlefatty}.
 By the above,   $\gamma$ is a (non-rectangular) 2-line partition.
  If $\gamma$   is equal to $(k,k-1)$ up to conjugation, then we can remove rows and columns common to $\mu$ and $\lambda$ until we   are in one of the four  cases in \cref{hellodleo} or at one of $((7^3),(4^4,3))$
  or $((5^3),(2^5,1))$.    In the first three cases in \cref{hellodleo}, the result follows by \cref{squares,section2parter}.
 In the fourth case  in \cref{hellodleo}, we remove the final row of $\tmu$ and the penultimate column of $\tla$ to obtain $\hat\mu$ a fat rectangle and $\tla$ with $|\Remm(\tla)|=3$; so this case and the special cases follow by \cref{sec:rectangle}.

 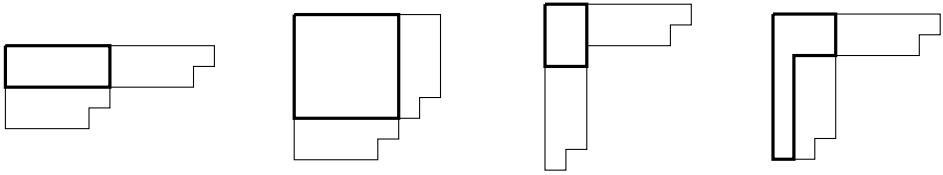
\begin{figure}[ht!]
 $$
  \begin{minipage}{38mm}\begin{tikzpicture}[scale=0.55]
  \draw[very thick]
 (0,0)--(2.5,0)--(2.5,-1)--(0,-1)--(0,0);
     \draw
       (2.5,0)--(5,0)--(5,-0.5)--(4.5,-0.5)--(4.5,-1)--(2.5,-1);
       \draw(2.5,-1)--(2.5,-1.5)--(2,-1.5)--(2,-2)--(0,-2)--(0,-1);
           \end{tikzpicture}\end{minipage}
             \begin{minipage}{33mm}\begin{tikzpicture}[scale=0.55]
  \draw[very thick]
 (0,1.5)--(2.5,1.5)--(2.5,-1)--(0,-1)--(0,1.5);
     \draw
       (2.5,1.5)--(3.5,1.5)--(3.5,-0.5)--(3,-0.5)--(3,-1)--(2.5,-1);
       \draw (2.5,-1)--(2.5,-1.5)--(2,-1.5)--(2,-2)--(0,-2)--(0,-1);
           \end{tikzpicture}\end{minipage}
  \begin{minipage}{30mm}\begin{tikzpicture}[scale=0.55]
  \draw[very thick]
 (1.5,0)--(2.5,0)--(2.5,-1.5)--(1.5,-1.5)--(1.5,0);
     \draw
       (2.5,0)--(5,0)--(5,-0.5)--(4.5,-0.5)--(4.5,-1)--(2.5,-1);
       \draw(2.5,-1.5)--(2.5,-3.5)--(2,-3.5)--(2,-4)--(1.5,-4)--(1.5,-1);
           \end{tikzpicture}\end{minipage}
   \begin{minipage}{33mm}\begin{tikzpicture}[scale=0.55]
  \draw[very thick]
 (1,0)--(2.5,0)--(2.5,-1)--(1.5,-1)--(1.5,-3.5)--(1,-3.5)--(1,0);
     \draw
       (2.5,0)--(5,0)--(5,-0.5)--(4.5,-0.5)--(4.5,-1)--(2.5,-1);
       \draw(2.5,-1.5+.5)--(2.5,-3.5+.5)--(2,-3.5+.5)--(2,-4+.5)--(1.5,-4+.5)--(1.5,-1);
           \end{tikzpicture}\end{minipage}
       $$
\caption{
The four families,  up to conjugation,   for $\gamma=(k,k-1)$ and $k\geq 3$   in case $(1a)$. }
\label{hellodleo}
\end{figure}

Continuing with case $(1a)$,
suppose that  $\gamma$ is equal to one of $(k+1,k-1)$,
$(k-3,3)$ or $(k-2,2) $   and  $\delta$ is a non-linear rectangle.
Remove all rows and columns common to both $\mu$ and $\la$  to obtain $\tmu$ a proper fat hook
and $\tla$ a non-linear rectangle; the result follows by \cref{sec:rectangle}.

It remains to consider the cases where
 $\gamma$  is equal to one of $(2^{k-1},1^2)$,
$(2^3,1^{k-6})$ or $(2^2,1^{k-4}),$ or $ (2^4,1)$  and  $\delta$ is a non-linear rectangle.

 Suppose $\gamma=(2^{k-1},1^2)$.
  If
    $\ell(\delta)>2$,   remove all rows and columns  common to $\la$ and $\mu$
to obtain $\tmu=(2^{k+\ell(\delta)-1},1^2)$ and $\tla$ a fat rectangle.
 If
   $\ell(\delta)=2$,  remove all rows and columns
   common to $\la$ and $\mu$
  with the exception of one column in the leg
 to obtain $\tmu=(3^{k+1},2^2)$ and $\tla=((k+3)^2,1^{k+1})$.
 Remove the final two rows of $\tmu$ and the final two columns of $\tla$
to obtain $\hat\mu$ a fat rectangle and $\hat{\la}$ a proper fat hook.    The result follows by \cref{sec:rectangle}.

 Now suppose that $\gamma=
 (2^3,1^{k-6})$ or $(2^2,1^{k-4}),$ or $ (2^4,1)$ and that $\gamma\supseteq (2^2,1^3)$ (as the other cases were  handled above).
Remove all rows and  columns common to both $\mu$ and $\lambda$ to obtain
$\tmu\supseteq (2^4,1^3)$   a 2-line partition not of the form $
 (2^{k-1},1^2)$ or $  (2^k)  $
 and $\tla$ is a non-linear rectangle.   The result follows by \cref{section2parter}.

Now consider the cases where both $\gamma$ and $\delta$ are
equal to $(k+1,k)$ (up to conjugation) for $(1b,c,d,e)$, where $k>1$.
  Remove all rows and columns common to both $\mu$ and $\lambda$ and arrive at
  twelve distinct cases (as $(1c)$ and $(1e)$ produce the same set of cases).
Eleven of the twelve  cases follow by \cref{squares,section2parter,sec:rectangle}.  The final case  is
$
 ((k+2,(k+1)^{3}),((2k+2)^2,1))
 $.
For $k>2$,  remove   the final  column of  $\tmu$ and the final row of  $\tla$ to obtain a pair of  rectangular partitions.
  The result follows from \cref{sec:rectangle}.
 For  $k=2$   we obtain  the seed $((6^2,1),(4,3^3))$.

It remains to consider cases $(1b,c,d,e)$ in which precisely one of  $\gamma$ and $\delta$ is a rectangle.

 In case $(1b)$, where $\gamma$ has to be a rectangle (respectively in  cases  $(1c,d,e)$ when $\delta$ is a rectangle)
  remove all rows and columns common to $\mu$ and $\lambda$ to obtain
    $\tla$ a non-linear rectangle and $\tmu$ a proper fat hook
(respectively   $\tmu$ a rectangle and $\tla$ a proper fat hook).
  The result follows from \cref{sec:rectangle}.

It remains to consider cases $(1c,d,e)$ for $\gamma$ a rectangle.
In cases $(1c)$ (respectively $(1e)$) with $\gamma=(2^k)$ and
$\delta=(k+1,k-1)$,
remove all rows and columns common to $\la$ and $\mu$ with the exception of one column in the leg (respectively  row in the arm)   to obtain $\tla$ and $\tmu$.
 In case $(1e)$, remove the final two columns of $\tla$ and the final two rows of $\tmu$ to obtain $\tla$ a fat rectangle and $\tmu$ a proper fat hook. In both cases the result follows by \cref{sec:rectangle}.  For  a case  not of the above form,  remove all rows and columns common to $\la$ and $\mu$ to obtain $\tla$ and $\tmu$.  The result follows by \cref{sec:rectangle}.
 \end{proof}

\begin{lem}
If either  of $\gamma$ or $\delta$ is linear and the other has two connected components, then $g(\la,\mu)>1$.

\end{lem}

\begin{proof}

By \cref{theo:ThDvir1,thm:mf-outer,thm:mf-basic-skew},
the non-linear   diagram (of the pair $\gamma$ and $\delta$) belongs to the list of skew partitions in
\cref{thm:mf-outer}.
By   \cref{42879342578837495278354278354}, we can assume that neither of $[\gamma]$ or $[\delta]$ is $[k-1] \boxtimes [1]$, up to conjugation.
  We first consider the exceptional cases where we cannot remove all rows and columns common to $\mu$ and $\lambda$.  These only happen in a few  cases in which all three external components in the diagram are linear.
 Up to conjugation of both $\la$ and $\mu$,  our exceptional cases   are listed below.
By aggressive application of \cref{animportantremark},  we can remove all rows and columns common to both partitions with the exception of a single row, $R$, or column, $C$,  to obtain $\tla,\tmu$.  These rows and columns are also listed.
\begin{itemize}[leftmargin=0pt,itemindent=2em]
\item[$(2a)$]
 $\gamma$ linear, $\delta'=(l)$, $\delta''=(1^m)$ and $C$ a single column in the arm;
\item[$(2b)$]
  $\gamma $ linear, $\delta'=(l)$, $\delta'' =(1^m)$ and $R$ a single row in the arm;
\item[$(2c )$] $\gamma=(1^{l+m})$, $\delta'=(l)$ and $\delta''=(1^m)$ and $C$ a single column in the leg;
\item[$(2d )$]  no exceptions;
\item[$(2e)$] $\gamma'=(1^l),  \gamma''=(m)$, $\delta=({l+m})$ and $C$ a single column in the arm or the leg;
\item[$(2f)$]  $\gamma'=(1^l)$ $\gamma''=(m)$ and $\delta=({l+m})$ or $(1^{l+m})$
 and $R$ a single row in the arm (the resulting partitions are the same as in the $(2b)$ case).
  \end{itemize}
  For any $\gamma$ and $\delta$ and any case  $(2a$--$e)$  not on the above list, remove all columns and rows common to $\mu$ and $\lambda$
 to obtain $\tmu$ and $\tla$.
 For an example of how \cref{animportantremark} is used, we compare $(2b)$ and $(2f)$; here we have reduced to the same set of
  exceptional cases, but using different arguments.   For $(2b)$, we know there must exist a row in the arm as $\mu$ is non-rectangular.
 For $(2f)$, we know that $\la$ is not a hook, and so there must be an extra row in the arm or column in the leg.
 However, case $(2f)$ is symmetric under conjugation (note that case $(2b)$ is not) and so we can assume there is an extra row in the arm.

For the exceptional cases of type $(2a)$, we have that $\tmu$ is non-rectangular (up to conjugation, $\tmu$ is obtained by adding a single node to the partition $((l+m+1)^2)$) and $\tla\supset(3,1^2)$ is a hook partition. Therefore the result follows from \cref{sectionhook!}.
   The generic case   follows from \cref{sec:rectangle}, as $\tmu$ is   a fat rectangle   and $\tla\supset (3,1^2)$.

In case $(2d)$, we know by  \cref{induction} and \cref{thm:mf-outer} that
$\delta'=(t^u,v^w)$ is a fat hook and therefore $\delta''=(r^s)$ is a rectangle.   If $\gamma=(1^{l+m})$ we remove the final $rs$ rows (each of width $r+1$) from $\tmu$
and the  final $s(r+1)$ rows (each of width $r$) from $\tla$ and hence obtain a pair $(\hat\la,\hat\mu)$ as in~$(1b)$.
 If $\gamma=({l+m})$,  then  the shortest row of $ \tmu$ (equal to $l+m+r$)
 is longer than the longest column of $\tla$ (equal to $s+u+w+1$) and therefore
 $\tla  / \tla\cap\tmu^t$ and $\tmu^t / \tla  \cap\tmu^t$ are both connected  (in fact  $(\tla  ,\tmu^t)$   are as in case $(1a)$). In both cases, the result follows by earlier results in this section.

We now consider case $(2e)$.  In the exceptional case with $C$ a single column in the arm, we have that $\tla=(2m+l+2,2)$ and $ \tmu  $ is a proper fat hook;
the result holds by  \cref{short_second_row} and \cref{induction}.
 In the exceptional case with $C$ a single column in the leg, we remove the final two columns of $\tla$ and the final row of $\tmu$ to obtain  $(\hat\la,\hat\mu)$ such that  $\hat\la/\hat\la\cap \hat\mu  $ and
 $\hat\mu/\hat\la\cap \hat\mu $   are both proper skew partitions with two  components each and the result follows by \cref{induction}.

We now consider the generic case of $(2e)$ with $\delta=(l+m)$.
We first consider the case where $w(\gamma')$ or $\ell(\gamma'')$
is equal to~1.
  If  $w(\gamma')=1$  and $\gamma''$ is  a rectangle,
  then $\tla$ is a
  hook  and  the result follows from   \cref{sectionhook!}.
  If  $w(\gamma')=1$  and $\gamma''$ is not  a rectangle,
  then remove $\gamma'$ from $\tmu$ and $|\gamma'|$ nodes from $\tla_1$ to obtain
  $\hat\mu$ a fat rectangle and $\hat\la$ a partition with at least three removable nodes; the result follows by \cref{sec:rectangle}.
We now assume that   $\ell(\gamma'')=1$ and  $w(\gamma'')>1$.
 If $\gamma'$ is a rectangle, then  the result follows by \cref{section2parter}.
  If    $\gamma'$ is not  a rectangle,  remove the final $w(\gamma'')$ columns from $\tmu$ and
$2w(\gamma'')$ nodes from $\tla_1$ to obtain $\hat\mu$ a non-linear rectangle and $\hat\la$ such that $|\Remm(\hat\la)|\geq3$.

By \cref{animportantremark} (see \cref{thm:mf-outer} in particular) we may assume   that  at least one of   $\gamma'$ or $\gamma''$ is a  rectangle and that $w(\gamma'),\ell(\gamma'')>1$.
 If  $\gamma'$ is a rectangle,   remove $\gamma'$
   from the bottom of $\tmu$
 and $|\gamma'|$ nodes from the   $\tla_1$ to obtain
   $\hat\mu $ a  fat rectangle and  $ \hat\la\supset  (3,2^2)$;
 the result follows by \cref{sec:rectangle}.
   We now assume that   $\gamma''$ is a rectangle.
 Remove the final $w(\gamma'')$ columns of $\tmu$
(each of length $  \ell(\gamma'') +1  $)
   and $w(\gamma'')(\ell(\gamma'')+1)$  nodes from   $\tla_1$ to obtain
   $\hat\mu $ a non-linear rectangle and  $ \hat\la\supset (3,2^2)$.
  The result follows by \cref{sec:rectangle}.

 We now consider the case  $\delta=(1^{l+m})$.
 If $\gamma'=(l)$, $\gamma''=(m)$ (with $l,m\neq1$ by our assumptions), remove the first  row of $\tmu$ and the final column of $\tla$ to obtain $\hat\la=\hat\mu\supset (4^4)$; the result follows by \cref{squares}.
Now assume that  $\delta=(1^{l+m})$ and $\gamma', \gamma''$ are not of the above form.
The shortest column of $ \tla$ (of length $l+m$)
 is longer than the longest row of $\tmu$
 (of length $w(\gamma')+w(\gamma'')$) and therefore
 $\tla^t / \tla^t\cap\tmu$ and $\tmu / \tla^t\cap\tmu$ are both connected
   and the result follows by earlier results in this section.

We now consider the  cases $(2f)$ and $(2b)$.
We first consider the generic case of $(2f)$.   If  $\gamma'$ and $\gamma''$ are  both rectangles,  then   $(\tla,\tmu)$
are as in case $(2a)$ considered above.   If    one of $\gamma'$ and $\gamma''$ is a rectangle and the other is a non-rectangular fat hook, then  the pair $(\tla,\tmu)$ are as in case  $(2d)$ considered above.
 Up to conjugation,  it remains to  consider  the case in which   $\gamma' $   is linear and $(\gamma'')^{\rm rot}$  is such that $\Remm( (\gamma'')^{\rm rot})\geq3$ ; in particular $ (\gamma'')^{\rm rot} \supseteq (3,2,1)$.
 If $\delta=(l+m)$, then the shortest  row   of $\tla$ is of length $l+m+w(\gamma')$,
 and the longest column of
$\tmu$ is less than or equal to $l+m-2$.
Therefore
$\tla^t / \tla^t\cap\tmu$ and $\tmu / \tla^t\cap\tmu$ are both connected
   and the result follows by earlier results in this section.
If $\delta=(1^{l+m})$ and $\gamma '=(l)$,
 remove the final  $(l+1)$ rows (of width $l$) from $\tmu$ and the final $l$ rows (of width $l+1$) from $\tla$
 to obtain $(\hat\la,\hat\mu)$.
 If $\delta=(1^{l+m})$ and $\gamma''=(1^l)$,
 remove the final $ 2l$ rows (of width 1) from $\tmu$ and the final $l$ rows (of width 2)  from $\tla$ to obtain $(\hat\la,\hat\mu)$.
 In either case,  $\hat\la^t / \hat\la^t\cap\hat\mu$ and $\hat\mu / \hat\la^t\cap\hat\mu$ are both connected
   and the result follows by earlier results in this section.

 The generic case for $(2b)$  follows from \cref{sec:rectangle} as $\tmu$ is a rectangle.
We now argue for the exceptional case for $(2b)$ (the exceptional case for $(2f)$ is identical but with the roles of $\gamma$ and $\delta$ switched).
For $\gamma=(1^{l+m})$ (respectively $({l+m})$)
    remove the final row of $\tmu$
    (respectively final two columns of~$\tmu$)
   and   the final column of $\tla$
 to obtain
$(\hat\mu ,\hat\la)$  such that $\hat\mu/\hat\la\cap \hat\mu$ and $\hat\la/\hat\la\cap \hat\mu$ both having two connected components
 (respectively  $(\tla,\tmu)$ are as in the generic case of $(2f)$).

 For  the exceptional case of $(2c)$,  remove the final row of $\tmu$ and the
final two columns  of $\tla$ to obtain $(\hat\la,\hat\mu)$.
If $l=2$, then
$(\hat\la,\hat\mu)$ are as in the exceptional case for $(2b)$.  If $l>2$, then
$\hat\la/\hat\la\cap \hat\mu  $ has three connected components and so
  the result follows by \cref{induction}.
  Now assume that we are in the generic  case with  $\gamma=(1^k)$.
  If $\tmu$ is a hook or 2-line partition, the result follows by \cref{sectionhook!,section2parter}.  Otherwise,
  we remove $\delta'$ from $\tla$ and $|\delta'|$ nodes from
  the first column of $\tmu$ to obtain a pair as in case $(1e)$ with $\tmu$ a proper fat hook.
For $\gamma=(k)$, if $\ell(\delta')+\ell(\delta'')=2$ then the result follows by \cref{section2parter}.  Otherwise,
 $ \tla\cap \tmu^t  =((\ell(\delta')+\ell(\delta'')+1)^{\ell(\delta')+\ell(\delta'')})\supseteq (4^3)$ and the result follows by \cref{squares}.   \end{proof}

\section{The general case}\label{sec:finale}

In this section, we continue to assume that \cref{thm:classification} has been proven for all Kronecker products labelled by pairs of  partitions of degree   less than or equal to $n-1$.
Armed with the  proof of \cref{thm:classification} for the case where one partition is a fat hook of degree $n$,
 we now embark on proving the general case for arbitrary pairs of partitions of degree $n$.

 We shall assume throughout that
$\lambda,\mu \vdash n$  are a pair of partitions
such that $\la\ne \mu$ and neither $\lambda$ nor  $\mu$
is a
 fat hook.
Furthermore, by \cref{theo:ThDvir1,animportantremark} we may also assume that the
pair of characters associated with the skew diagrams
$\gamma=\mu / (\la\cap \mu)$ and $\delta=\la / (\la\cap \mu)$
belongs to the lists in \cref{thm:classification} and \cref{thm:classification-skew}.
 In particular, we may (and will)
assume without loss of generality that $\gamma$ has one connected component and that $\delta$ has either one or two connected components.

 We shall  systematically work through the list
 of possible pairs of shapes
 $\la / (\la\cap \mu)$
and
$\mu / (\la\cap \mu)$
and reduce the corresponding  pairs of partitions $\la$ and $\mu$
to pairs of partitions $\tla$, $\tmu$ such that $g(\tla, \tmu)>1$ and  the semigroup property implies $g(\la,\mu)>1$.
 Our typical approach will be to  reduce to the case that one of $\tla$ or $\tmu$ is  a 2-line, rectangle, or fat hook partition and then appeal to the results of  \cref{sec:warm-up,sec:rectangle,sec:fatty}.

\begin{lem}\label{firstlinear}
Suppose $\gamma=\delta=(1)$, then $g(\la,\mu)>1$.
\end{lem}

\begin{proof}
 We let  $\gamma=(r_1,c_1)$, $\delta=(r_2,c_2)$ and
 we suppose, without loss of generality,
  that $r_1<r_2$ and $c_1>c_2$.
  Our general strategy shall be to remove all rows and columns
  outside of the region labelled by
   $[r_1, \dots, r_2] \times [c_2 , \dots, c_1]$, an example is  depicted in \cref{ASDFASFSDAFAS}, below.
We first consider the exceptional cases in which
$$
\zeta= ([r_1, \dots, r_2] \times [c_2 , \dots, c_1]) \cap  \lambda  \cap  \mu
$$
is equal to the Young diagram of a partition
of the form $(k,k)$, $(k,k-1)$, $(k-1,1)$, or $(k)$ up to conjugation.

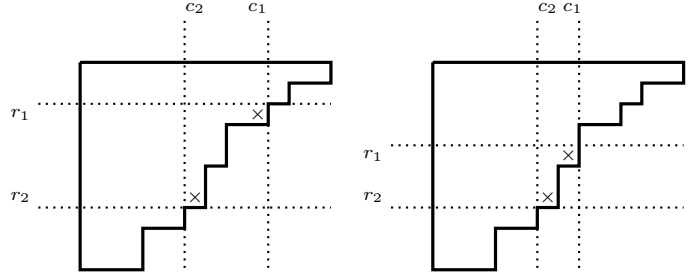
\begin{figure}[ht!]
$$
\scalefont{0.6}
     \begin{minipage}{38mm}\begin{tikzpicture}[scale=0.55]
  \draw[very thick]
      (-1,1.5)--(5,1.5)--(5,1)--(4,1)--(4,0.5)--(3.5,0.5)--(3.5,0)--(2.5,0)
      --(2.5,-1)--(2,-1)--(2,-2)--(1.5,-2)--(1.5,-2.5)--(0.5,-2.5)--
       (0.5,-3)-- (0.5,-3.5)-- (-1,-3.5)--
      (-1,1.5) ;
    \draw[thick, dotted]
      (-2,-2)--(5,-2)
      (-2,0.5)--(5,0.5)
      (3.5,2.5)--(3.5,-3.5)
            (1.5,2.5)--(1.5,-3.5)
      ;
      \draw(-2,-1.75) node[left] {$r_2$};
          \draw(-2,0.25) node[left] {$r_1$};
                    \draw(1.75,2.5) node[above] {$c_2$};                \draw(3.25,2.5) node[above] {$c_1$};
\draw(1.75,-1.75) node  {$\times$};                                        \draw(3.25,0.25) node  {$\times$};

           \end{tikzpicture}\end{minipage}
\quad\quad\quad
   \begin{minipage}{38mm}\begin{tikzpicture}[scale=0.55]
  \draw[very thick]
      (-1,1.5)--(5,1.5)--(5,1)--(4,1)--(4,0.5)--(3.5,0.5)--(3.5,0)--(2.5,0)
      --(2.5,-1)--(2,-1)--(2,-2)--(1.5,-2)--(1.5,-2.5)--(0.5,-2.5)--
       (0.5,-3)-- (0.5,-3.5)-- (-1,-3.5)--
      (-1,1.5) ;
    \draw[thick, dotted]
      (-2,-2)--(5,-2)
      (-2,-0.5)--(5,-0.5)
      (2.5,2.5)--(2.5,-3.5)
            (1.5,2.5)--(1.5,-3.5)
      ;
      \draw(-2,-1.75) node[left] {$r_2$};
          \draw(-2,-0.75) node[left] {$r_1$};
                    \draw(1.75,2.5) node[above] {$c_2$};
                       \draw(2.35,2.5) node[above] {$c_1$};
\draw(1.75,-1.75) node  {$\times$};
               \draw(2.25,-0.75) node  {$\times$};
        \end{tikzpicture}\end{minipage}
$$
\caption{An example of a generic and an exceptional pair of partitions  $\la$ and $\mu$ such that
 $\gamma=\delta=(1)$.  We have decorated the diagram with the region
   $[r_1, \ldots , r_2]\times [c_2, \dots, c_1]$. In the former  case the partition $\zeta$ has 3 removable nodes, in the latter case $\zeta$ is linear.
   }
\label{ASDFASFSDAFAS}
\end{figure}

We may assume that  $\zeta=(k)$ (the case $\zeta=(1^k)$ is similar);  we remove most rows common to $\la$ and $\mu$ to obtain three distinct cases.
If both  $r_1, c_2 \neq1$, we
can remove all but one column to the left of the region
 and all but one row above the region from $\la,\mu$
   to obtain partitions $\tla,\tmu$ such that
$$  \tla \cap \tmu  =
  (k+2,k+1,1)\:.
$$
In this case $g(\tla,\tmu)= g(((k+2)^2,1),(k+2,k+1,2))>1$, by \cref{sec:fatty}.

Now suppose that $c_2=1$, i.e.,
$\delta=(r_2,1)$.  By assumption, we have that
$\mu $ is not a rectangle and so   $\mu_1>\mu_{r_1}$.
 We   remove all
but the longest row  (of width $\la_1=\mu_1>\mu_{r_1}$)
   above the region;     we then truncate this row to be of length
   $k+2$; we hence
     obtain $\tla$ and $\tmu$ such that
$$
  \tla \cap \tmu
   =(k+2,k)\:.
 $$
In this case $g(\tla,\tmu)=g((k+2,k,1),(k+2,k+1))>1$, by \cref{section2parter}.

Now suppose that
$\gamma=(1,c_1)$, in which case we can remove all
but the longest column
  to the left of the region (which is of length   greater than or equal to 3, by assumption that neither of $\la$ or $\mu$  is a 2-line partition), we then truncate this column
   to be of length
   3 and hence
     obtain $\tla$ and $\tmu$ such that
$$
  \tla \cap \tmu
   =( k+1,1^2)\:.
 $$
In this case $g(\tla,\tmu)=g((k+1,2,1),(k+2,1,1))>1$, by \cref{sectionhook!}.

We now assume that
 $ \zeta $ is
  of the form
   $(k,k)$, $(k+1,k)$, or  $(k-1,1)$
up to conjugation.
 In all of these cases,
 we know that there is at least one extra column or row common to $\la$ and $\mu$ which we may consider; this follows from
 our assumption that neither
 $\la$ nor  $\mu$ is    a 2-line partition.
  This leads us to define
$\tla,\tmu$ as the intersections of $\la,\mu$ with the region
$[r_1-1,r_1, \dots, r_2] \times [c_2 , \dots, c_1]$ or
$[r_1, \dots, r_2] \times [c_2-1,c_2 , \dots, c_1]$,
so that $\tla\cap\tmu$ is equal to one of
 $$
 ([r_1-1,r_1, \dots, r_2] \times [c_2 , \dots, c_1]) \cap  \lambda  \cap  \mu
\qquad \text{or} \qquad
 ([r_1, \dots, r_2] \times [c_2-1,c_2 , \dots, c_1]) \cap  \lambda  \cap  \mu\:.
 $$
It will then suffice to show that $g(\tla,\tmu)>1$ in both cases for all three possible partitions, $\zeta$.
In the  latter case, for  $\zeta=(k+1,k)$
we have that $(\tla,\tmu)=((k+3,k+1,1), (k+2,k+1,2))$;
removing $(2)$ from $\tla_1$ and $(2)$ from $\tmu_3$, the result follows from \cref{section2parter}.
In the other five cases, the result follows as
$(\tla,\tmu)$ is not on the list of \cref{thm:classification} and
  one of the two partitions is a fat hook and so the result follows by \cref{sec:fatty}.

We now deal with the generic case  (in which  $\zeta\neq (k)$,   $(k,k)$, $(k+1,k)$, or  $(k-1,1)$, up to conjugation);
  remove all rows and columns
 outside of the region labelled by
   $[r_1, \dots, r_2] \times [c_2 , \dots, c_1]$ from $\la,\mu$,
   to obtain $\tla$ and $\tmu$ such that
$$
  \tla \cap \tmu  =[r_1, \dots, r_2] \times [c_2 , \dots, c_1]
  \cap  \lambda  \cap  \mu.
 $$
 We note that the node $\gamma$ (respectively $\delta)$ is
 $(SG)$-removable from  $\tmu$ (respectively  $\tla$); the result follows as $g( \tla \cap \tmu, \tla \cap \tmu)>1$ by \cref{selftensor}.
\end{proof}

\begin{lem}\label{8.2}
If $\gamma$ and $\delta$ are both  linear, then $g(\la,\mu)>1$.
\end{lem}

\begin{proof}
We assume, without loss of generality,
 that $\gamma$ appears
 higher than $\delta$ in the diagram
 and $(\gamma,\delta)=((1^k),(1^k))$,
 $((k),(1^k))$ or $((1^k),(k))$.  By \cref{firstlinear}, we may assume that  $k\geq 2$.
 The case $((k),(k))$ can be obtained by conjugation.

\textbf{Case 1:}   $(\gamma,\delta)=  ((1^k),(1^k))$.
Assume  there is a column, $C$,   to the left
of $\delta$ (respectively  to the right of $\gamma$).
   Remove from the intersection all
rows and columns  excluding  column $C$
 (respectively all columns excluding $C$, and all but one of the rows of width $c>c_1$
  above $\gamma$) to obtain $(\tla,\tmu)$ equal to either of
  $$
  ((3^k,1^k),
  (2^{2k}))\:, \;
  ((3,2^{k}),
  (3,1^{2k}))
  $$
  and the result follows from \cref{sec:fatty}.
  Now assume there is no such column to the left or right and recall our assumption that neither $\lambda$ nor $\mu$ is a 2-line partition.  There are two distinct cases to consider, namely
 \begin{itemize}[leftmargin=0pt,itemindent=1.5em]
 \item $k\geq 2$ and there is a single column, $C$, in between $\delta$  and $\gamma$
 and  a row, $R$, above $\gamma$;
 \item $k\geq 3$ and
there are at  least two columns, $C_1, C_2$  in between $\gamma$ and $\delta$ and no rows above $\gamma$.
 \end{itemize}
 In the former case,
 we remove from the intersection all
rows and columns  excluding $R$ and $C$ to obtain
$(\tla,\tmu)=(  (3,2^k,1^k), (3^{k+1}))$.
     In the latter  case, we remove from the intersection
  all rows and all columns except $C_1$ and $C_2$ to obtain
  $(\tla,\tmu)=((3^k,1^k),(4^k))$. In both cases the result follows from \cref{sec:rectangle} as
 $\tmu$ is a rectangle.

\textbf{Case 2:}
 $(\gamma,\delta)=  ((k),(1^k))$ for $k\geq 2$.
    If there is both a column and a row between $\gamma$ and $\delta$,  then we  reduce to the case $(\tla,\tmu)=((2^2,1^k),(k+2,2))$ and the result follows \cref{section2parter}.
 We may now assume that there is not both a column and a row between $\gamma$ and $\delta$.
 Conjugating if necessary, we may assume that there is no    column between $\gamma$ and $\delta$.
 Suppose that there are no rows above $\gamma$.
Then by our assumption that $w(\tla)>2$,
 there are two columns $C$ and $C'$ to the left of $\delta$.
 We  remove
 from the intersection
  all
 rows and columns except for $C$ and $C'$ to obtain
 $(\tla,\tmu)=((3^{k+1}),(3+k,2^{k})$
 and the result follows from \cref{sec:rectangle}.
 We may now suppose that there is a row, $R$, above $\gamma$.
 By assumption, $\tla$ is not a hook partition and so
there is either $(i)$ a single column, $C$, to the left of $\delta$
or $(ii)$ an extra row $R'$ above $\gamma$.
    In the former case, we remove all rows except $R$ and
    all columns except $C$ and hence obtain $(\tla,\tmu)=((k+2,2^{k+1}),((k+2)^2,1^k))$ with $k\geq 2$ and so the result follows from   \cref{sec:fatty}.
In the latter case, we remove
    from the intersection
     all  rows and columns with the exception of
        $R$ and $R'$
    to obtain $(\tla,\tmu)=(((k+1)^2,1^{k+1}),((k+1)^3))$; the result follows
    from \cref{sec:rectangle}.

\textbf{Case 3:} $(\gamma,\delta)=((1^k),(k))$.
For $k=2$, we can remove all but one column to the left of $\delta$ or all but one column between $\gamma$ and $\delta$ (up to conjugation)
to obtain $(\tla,\tmu)$ equal to either of the small seeds
$
((3^3),(4^2,1))
\:,\;
((3^2,2),(4^2))
$.  Otherwise, we may remove all rows and columns common to both partitions, and the result follows by
\cref{selftensor}.
\end{proof}

\begin{lem}
If one of $\gamma$, $\delta$ is linear  and the other
is a proper partition up to rotation, then $g(\la,\mu)>1$.
\end{lem}

\begin{proof}
We assume, without loss of generality,
 that $\gamma $ is linear and that it appears
 higher than $\delta$ in the diagram. By \cref{8.2}, we can assume that $\delta$ is non-linear.
We start with the discussion of
the cases where $\delta$ is a proper partition.

\textbf{Case 1:} $\gamma=(k)$ and $\delta$ is a proper partition.  Suppose there are no rows either above $\gamma$ or between $\gamma$ and $\delta$.
In which case (by our assumption that $\mu$ is neither linear, nor a hook) there exist two columns $C$ and $C'$ to the left of $\delta$.  We remove from the intersection all rows and all columns with the exception of $C$ and $C'$.  The result follows as $\tmu$ is a fat hook.

 Suppose that there is a row, $R$,  above $\gamma$.  Remove all rows and columns common to both $\la$ and $\mu$ with the exception of $R$, to obtain
 $\tmu=((k+w(\delta))^2)$ and
 $\tla = (k+w(\delta), w(\delta), \delta)$;
 we have that $\tla$ is either a proper fat hook or
 $|\Remm(\tla)|\geq3$ (our assumptions imply that
  $w(\delta)\geq 2$, $k\geq 3$).
  The result then follows from \cref{section2parter}.

 Now assume that there is a row, $R$, between $\gamma$ and
 $\delta$ and no row above $\gamma$.  If $\delta$ is not a fat hook, then remove all common rows and columns from $\la$ and $\mu$ with the exception of $R$ to obtain
 $\tmu$    a 2-line partition and  $\tla  $ a partition such that
  $|\Remm (\tla   )|\geq 3$.
We now assume  that  $\delta$ is a fat hook.
Now  (by assumption that $\mu$ is not a 2-line partition)
 there is either a second row $R'$ between $\gamma$ and $\delta$ or an extra column, $C$, to the left of $\delta$.
  In either case remove  all rows and columns from $\la$ and $\mu$ with the exception of $R$ and $R'$ or $R$ and $C$  to obtain a pair $(\tla,\tmu)$.
  In the former case, $\tmu$ is a proper fat hook and $[\tla]$ is neither a linear character nor the natural character or its dual.
 In the latter case, $\tla$ is a proper fat hook and $\tmu$ has three removable nodes.
  In either case, the result follows from \cref{sec:fatty}.

\textbf{Case 2:}
$\gamma=(1^k)$ and $\delta$ is a proper partition.   We have two exceptional cases to consider, in which $\delta=(2,1)$ or $(2,2)$.
 In either case, we remove all but a single row or column
 from   $\la$ and $\mu$ to obtain 12 seeds $(\tla,\tmu)$ of degree less than or equal to 18.
Assume  $\delta\neq (2,1)$, or $(2,2)$;  remove all rows and columns common to $\la$ and $\mu$ to obtain $\tla$ and $\tmu$.
 If $w(\delta)=2$, the result follows from \cref{section2parter}.
Otherwise, $\tla\cap \tmu= (w(\gamma)^k) $ with $w(\gamma), k \geq 3$ and
 so $g(\lambda,\mu) \geq g(\tla\cap \tmu,\tla\cap \tmu)>1$ by \cref{squares}.

 \textbf{Case 3:} $\gamma=(k)$
  and $\delta^{\rm rot}$ is a proper partition.  By assumption, $\la$  is not a fat hook  and so there exists at least two
 columns, $C$ and $C'$,  of {\em distinct lengths} belonging to one or two of the regions:
  to the left of $\delta$, between $\gamma$ and $\delta$, or to the right of $\gamma$.
 We can assume that the final node in column, $C$ say, does not   belong to
   the same row as the nodes in the partition $\gamma$.
 Remove all rows and columns except for  $C$ to obtain $\tla$ a proper fat hook and $\tmu$  such that $\ell(\tmu ), w(\tmu)> 2$; the result follows from \cref{sec:fatty}.

 \textbf{Case 4:} $\gamma=(1^k)$ and $\delta^{\rm rot}$ is a proper partition.   We remove all rows and columns common to $\la$ and $\mu$, to obtain $\tla$ a non-linear rectangle and $\tmu$ a non-rectangular  partition such that    $(3^3)\subseteq \tmu$;
  the result follows from \cref{sec:rectangle}.
\end{proof}

We fix some notation which will be used throughout the remainder of this section.
 If  $\delta$ and $\gamma$ each have exactly one connected component, then we can assume without loss of generality that
 $\delta$ lies below $\gamma$ on the diagram, as depicted in the leftmost diagram in \cref{Extra rows and columns}.
 We shall let $R_1$ (respectively $R_2$) denote the longest row in $\la\cap\mu$ which appears above $\gamma$
 (respectively between $\delta$ and $\gamma$) if such a row exists, and let $R_1$ (respectively $R_2$) be undefined otherwise.
Similarly, we shall let $C_1$ (respectively $C_2$) denote the longest column in $\la\cap\mu$ which appears to the left of  $\delta$
 (respectively between $\delta$ and $\gamma$) if such a column exists, and let $C_1$ (respectively $C_2$) be undefined otherwise.  This is depicted in \cref{Extra rows and columns}.

 If  $\gamma$ has exactly one connected component and  $\delta$ has exactly two connected components, then we can assume without loss of generality that either
 \begin{itemize}[leftmargin=0pt,itemindent=1.5em]
\item $\gamma$ lies below $\delta'$ and $\delta''$ on the diagram, as depicted in the middle diagram in \cref{Extra rows and columns};
\item $\gamma$ lies between $\delta'$ and $\delta''$ on the diagram, as depicted in the rightmost  diagram in \cref{Extra rows and columns}.
\end{itemize}
We define the rows $R_1, R_2, R_3$ and $C_1, C_2, C_3$ by the obvious extension of the definition above, which is depicted in the   two rightmost  diagrams in \cref{Extra rows and columns}, below.

\begin{figure}[ht!]
$$
\scalefont{0.6}
     \begin{minipage}{38mm}\begin{tikzpicture}[scale=0.5,x={(0,-1cm)},y={(-1cm,0)}]
  \draw[very thick]
      (-2,1.5)--(4.5,1.5)--(4.5,1)--(4,1)--(4,0.5)--(3.5,0.5)--(3.5,0)--(2.5,0)
      --(2.5,-1)--(2,-1)--(2,-2)--(1.5,-2)--(1.5,-2.5)--(0.5,-2.5)--(0.5,-4)--(0,-4)--(0,-4.5)--(-0.5,-4.5)--(-0.5,-5)
     --(-2,-5)--(-2,1.5);
   \draw[very thick]
     (3.5,0.5)--(3.5,-1.5)--(2,-1.5);
   \draw[very thick]
  (1.5,-2.5)--(-0.5,-2.5)--(-0.5,-4.5);
         \draw(0,-3.25) node {$\gamma$};
          \draw(3,-0.8) node {$\delta$};
    \draw(-2.5,1.2) node {$C_1$};        \draw(-1.75,1.25)--(4.25,1.25);
                 \draw(-2.5,-1.75) node {$C_2$};      \draw(-1.75,-1.75)--(1.75,-1.75);
      \draw(-1.65,1.9) node {$R_1$};     \draw(-1.75,1.25)--(-1.75,-4.75);
      \draw(1.25,1.9) node {$R_2$};                             \draw(1.25,-2.25)--(1.25,1.25);
       \end{tikzpicture}\end{minipage}
\quad
      \begin{minipage}{48mm}\begin{tikzpicture}[scale=0.5]
  \draw[very thick]
      (-2,3.5)--(6,3.5)--(6,3.5)--(6,3)--(5.5,3)--(5.5,2.5)--(4.5,2.5)--(4.5,1)--(4,1)--(4,0.5)--(3.5,0.5)--(3.5,0)--(2.5,0)
      --(2.5,-1)--(2,-1)--(2,-2)--(1.5,-2)--(1.5,-2.5)--(0.5,-2.5)--(0.5,-4)--(0,-4)--(0,-4.5)--(-0.5,-4.5)--(-0.5,-5)
     --(-2,-5)--(-2,3.5);
   \draw[very thick]
     (3.5,0.5)--(3.5,-1.5)--(2,-1.5);
   \draw[very thick]
  (1.5,-2.5)--(-1,-2.5)--(-1,-5);
         \draw(-.25,-3.25) node {$\gamma$};
          \draw(3,-0.8) node {$\delta''$};
           \draw[very thick]
     (5.5,3)--(5.5,1.5)--(4.5,1.5);
          \draw(5,2) node {$\delta'$};
    \draw(-2.5,1.2) node {$R_2$};        \draw(-1.75,1.25)--(4.25,1.25);
                 \draw(-2.5,-1.75) node {$R_3$};      \draw(-1.75,-1.75)--(1.75,-1.75);
      \draw(-1.65,4) node {$C_1$};     \draw(-1.75,3.25)--(-1.75,-4.75);
      \draw(1.25,4) node {$C_2$};                             \draw(1.25,-2.25)--(1.25,3.25);
      \draw(-2.5,3.25) node {$R_1$};     \draw(-1.75,3.25)--(5.75,3.25);
            \draw(3.75,4) node {$C_3$};                             \draw(3.75,0.75)--(3.75,3.25);

       \end{tikzpicture}\end{minipage}
\quad
     \begin{minipage}{38mm}\begin{tikzpicture}[scale=0.5]
  \draw[very thick]
      (-2,3.5)--(6,3.5)--(6,3.5)--(6,3)--(5,3)--(5,2.5)--(4.5,2.5)--(4.5,1)--(4,1)--(4,0.5)--(3,0.5)--(3,0.5)--(2.5,0.5)
      --(2.5,-0.5)--(2,-0.5)--(2,-2)--(1.5,-2)--(1.5,-2.5)--(0.5,-2.5)--(0.5,-4)--(0,-4)--(0,-4.5)--(-0.5,-4.5)--(-0.5,-5)
     --(-2,-5)--(-2,3.5);
   \draw[very thick]
     (3.5,0.5)--(3.5,-1.5)--(2,-1.5);
   \draw[very thick]
  (1.5,-2.5)--(0.5,-2.5)--(0.5,-3)--(-0.5,-3)--(-0.5,-4.5);
         \draw(0,-3.5) node {$\delta'' $};
          \draw(3,-0.8) node {$\gamma$};       \draw[very thick]
     (5.5,3)--(5.5,1.5)--(4.5,1.5);
          \draw(5,2) node {$\delta'$};
    \draw(-2.5,1.2) node {$R_2$};        \draw(-1.75,1.25)--(4.25,1.25);
                 \draw(-2.5,-1.75) node {$R_3$};      \draw(-1.75,-1.75)--(1.75,-1.75);
      \draw(-1.65,4) node {$C_1$};     \draw(-1.75,3.25)--(-1.75,-4.75);
      \draw(1.25,4) node {$C_2$};                             \draw(1.25,-2.25)--(1.25,3.25);
      \draw(-2.5,3.25) node {$R_1$};     \draw(-1.75,3.25)--(5.75,3.25);
            \draw(3.75,4) node {$C_3$};                             \draw(3.75,0.75)--(3.75,3.25);

       \end{tikzpicture}\end{minipage}
$$
\caption{Extra rows and columns}
\label{Extra rows and columns}
\end{figure}
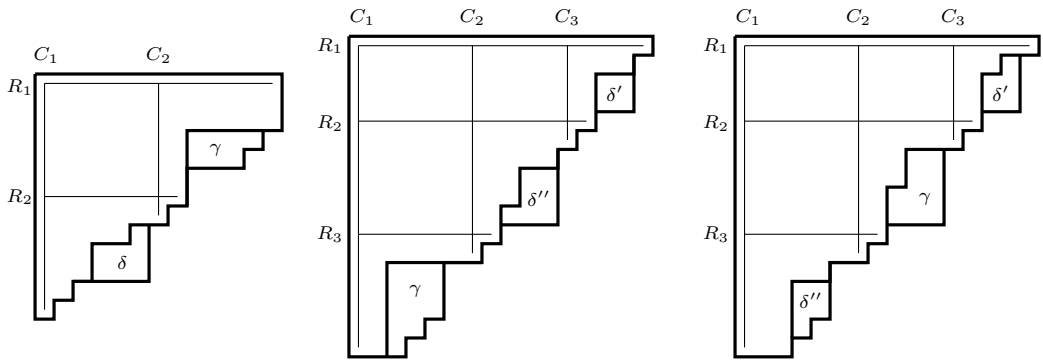

\begin{lem}\label{8.4}
If
 $\delta$ is a proper skew partition with one connected component, then $g(\la,\mu)>1$.
\end{lem}
\begin{proof}
We assume, without loss of generality,
 that $\gamma$ appears
 higher than $\delta$ in the diagram.
    By \cref{thm:mf-basic-skew} and \cref{induction},
we know that $\delta$ is of the form  $\delta=\sigma / \rho$ where $\sigma$ is a partition, and $\rho$ is a rectangle, or $\sigma,\rho$ satisfy $|\Remm(\rho)|\ge 2$ and $|\Remm(\sigma)|=2$.

 We  consider the exceptional case in which   $\rho = (1)$ and $|\Remm(\sigma)|=2  $ and $\gamma=(k)$.
 By assumption, neither $\la$ nor $\mu$ is a fat hook and so there exists at least one extra row or column $R_1,R_2, C_1$ or $C_2$ as in \cref{Extra rows and columns}.
 We remove all rows and columns
 common to both $\la$ and $\mu$ with the exception of  one of $R_1$ or $R_2$ or $C_1$ or $C_2$
to obtain $(\tmu,\tla)  $.
In all other cases, we remove all rows and columns common to both $\la$ and $\mu$ to obtain a pair of partition $(\la,\mu)$.
The resulting pair $(\tla,\tmu)$
 are such that $(i)$ $\tla\neq \tmu$ $(ii)$  both $\tla$ and $\tmu$ are non-rectangular $(iii)$ neither $[\tla],[\tmu]$ is equal to the standard character or its dual.
 Therefore $g(\tla,\tmu)>1$ as required.
 \end{proof}

\begin{lem}\label{finallinear}
If $\gamma$ is linear  and $\delta $ has two connected components, then $g(\la,\mu)>1$.
\end{lem}

\begin{proof}
We first consider the two exceptional cases, in which  $\delta'$ and $\delta''$ are both linear.

Suppose that $\delta''$ is below $\gamma$ and $\gamma$ is below $\delta'$  as depicted in the rightmost diagram in \cref{Extra rows and columns}.
We assume without loss of generality that $\gamma=(k_1+k_2)$ and $\delta'$ and $\delta''$ are partitions of $k_1$ and $k_2$, respectively.
 The only exceptional case for such a shape is given by
  $\delta''=(1^{k_2})$ and $\delta'=({k_1})$.
 We want to remove all but a single row or column from $\la$ and $\mu$ depending on having a suitable row or column in one of the six cases illustrated in \cref{Extra rows and columns}; however, as we assume that $\mu$ is not a 2-row partition, we can ignore the two cases $C_2$ and $C_3$.
It therefore remains to consider the cases where one of the columns or rows
$C_1, R_1, R_2,$ and $R_3$ exist, and we have reduced all other rows and columns
common to both $\la$ and $\mu$ to obtain $\tla$, $\tmu$.
In each of these four cases, the partition $\tmu$ is either a proper fat hook or a fat rectangle
 and $\tla$ is a partition with $w(\tla)\geq 4$, $\ell(\tla)\geq 3$ and $|\Remm(\tla)|\geq 2$;  the assertion follows from the result for fat hooks.

 Suppose that   $\gamma$   is below $\delta'' $, and $\delta''$ is below $\delta'$  as depicted in the central diagram in \cref{Extra rows and columns}.
 The only exceptional case for such a shape is given by
 $\gamma=(1^{k_1+k_2})$,  $\delta''=(1^{k_2})$ and $\delta'=({k_1})$.
 In this case, we need to consider each of the  six possible cases given by removing all rows and columns common to $\la$ and $\mu$ with the exception of one of
 $R_1, R_2, R_3, C_1, C_2,$ or $C_3$ to hence obtain partititons $\tla$ and $\tmu$.
  In the case of $R_1$,  $R_3$, $C_1$ or $C_2$, we have that  one of the partitions $\tla, \tmu$  is a fat hook and the other has 3 removable nodes.
 In the  case of $R_2$ or $C_3$, we have that $|\Remm(\tla)|, |\Remm(\tmu)| = 2$
 and either $\tla$ or $\tmu$ has width and length at least~3.  Therefore the claim follows from the result for fat hooks.

 Having taken care of the exceptional cases, we now turn our attention to the generic case.
By our inductive assumption, we have that   one of $\delta'$ and $\delta''$ is a rectangle and the other is a
 proper  partition, up to rotation.
  Note that  this covers   all the pairs $\delta'$ and $\delta''$   in \cref{thm:classification-skew,thm:mf-outer}.
   We let $\tla$ and $\tmu$ denote the partitions obtained by removing all row and columns common to both $\la$ and $\mu$.

   We first cover the simplest case in which $\delta'$ and $\delta''$ are both
rectangles (and one may be linear).  In this case, we remove all rows and columns
common to $\la$ and $\mu$ to obtain a pair of partitions $\tla \neq \tmu$ which are
both fat hooks and do not give a pair on our list;
 the result follows.

We now assume that one of  $\delta'$ and $\delta''$ is a rectangle
and the other is a proper non-rectangular partition up to rotation.
 If one of $\delta'$ and $\delta''$ is a rectangle and the other is obtained by
rotating a  proper  non-rectangular  partition, then  $\tla$ is necessarily a proper
fat hook and $\tmu$ is either a proper fat hook or  $|\Remm(\tmu)|>2$,    and the
result follows.
 In the non-rotated case,  $\tmu$  is necessarily a fat hook and
$|\Remm(\tla)|=|\Remm(\delta')|+|\Remm(\delta'')| \geq 2 + 1=3$, and the result
follows.
 \end{proof}

 \begin{lem}
If   $\delta$ is  a proper skew partition,
 then $g(\la,\mu)>1$.
\end{lem}

\begin{proof}
By \cref{thm:classification-skew} and \cref{8.4,finallinear}, it only
 remains to check the
case where
  $ \gamma = (a^b) $, for $a,b>1$, and $[\delta]=[\delta']\boxtimes [\delta'']$ with one of $\delta'$, $\delta''$ being $(1)$ and the other linear.
    We remove all   rows and columns common to both $\la$ and $\mu$
to obtain a pair  or partitions $(\tla,\tmu)$.

We can assume without loss of generality
  that $\gamma$ appears below $\delta'$ and $\delta''$  or between $\delta'$ and $\delta''$.
In the former case,  $\tmu$ is a proper fat hook and   $  \tla   \supseteq(2^2)$;   the result follows.
In the latter case  $\tmu$ is   a fat rectangle and
  $\tla$  is a partition satisfying $|\Remm(\tla)|=2$ and $\ell(\tla),w(\tla) \geq 4$;
therefore the result holds.
 \end{proof}

\begin{lem}
If either  $\gamma$ or $\delta$ is a rectangular partition, then $g(\la,\mu)>1$.
\end{lem}

\begin{proof}
Given the previous results, we suppose without loss of generality that $\gamma$ is a non-linear rectangle and $\delta$ is a non-linear fat hook up to rotation.
 We assume without loss of generality that $\gamma$ appears above $\delta$, as in \cref{Extra rows and columns}.

There are numerous  exceptional small cases, however we do not need to list them all. Instead, we shall show that if there is a row or column $R_1, R_2, C_1$, or $C_2$ as in \cref{Extra rows and columns} then the product contains multiplicities.   If  there is no such  row or column in the diagram for $\la$ and $\mu$ then if  $\delta$ (respectively $\delta^{\rm rot}$) is a proper partition, then it follows that   $\mu$ (respectively $\la$) is a rectangular partition (recall that $\la,\mu$ are not 2-line partitions, so the pair $(3^4),(6^2)$ does not occur), and so we are done.

 Now suppose that
the diagram has a row or column $R_1, R_2, C_1$, or $C_2$ and
we let $(\tla,\tmu)$ denote the pair obtained by removing all common rows and columns
except this single row or column  (in each of the four cases);  we  now
show that the product contains multiplicities.

If there is a row, $R_1$, in the diagram for $\la$  and $\mu$, then  $\tmu$ is either  a proper fat hook or a fat rectangle and $\tla$ is such that $w(\tla), \ell(\tla) \geq 3$.
 If there is a  column, $C_1$, then
  $\tmu$ is  either a proper fat hook or $|\Remm(\tmu)|=3$, and
   $\tla$ is either a proper fat hook or a fat rectangle.
   If there is a row, $R_2$, and $\delta$ is a proper partition (respectively $\delta^{\rm rot}$ is a proper partition and $\delta$ is not),  then
  $\tmu$ is   a proper fat hook (respectively   $|\Remm(\tmu)|=3$) and
   $\tla\supseteq (2^2)$   (respectively   $\tla$ is a rectangular partition).
If there is a column, $C_2$,  and $\delta$ is a proper partition (respectively $\delta^{\rm rot}$ is a proper partition and $\delta$ is not)  then
  $\tmu$ is a non-linear rectangle   (respectively is a proper fat hook) and in either case $\tla$ is either
   a proper   fat hook or $|\Remm(\tla)|\geq 3$.   In each of these cases, the result follows by the result for rectangles and  fat hooks.
\end{proof}

\begin{lem}
If   $\gamma$ and $\delta$ are both equal to $(k+1,k)$ up  to conjugation and rotation, then $g(\la,\mu)>1$.
\end{lem}

\begin{proof}
Without loss of generality we can reduce to three cases:
\begin{itemize}
\item[$(i)$]  $\gamma^{\rm rot}$ and $\delta^{\rm rot}$ are both proper partitions;
\item[$(ii)$]  $\gamma$ and $\delta^{\rm rot}$ are both proper partitions;
\item[$(iii)$]  $\gamma$ and $\delta$ are both proper partitions.
\end{itemize}
In case $(i)$, we remove all rows and columns to obtain $\tla$ and $\tmu$ a pair of proper fat hooks and the result follows.
In case $(ii)$, we remove all rows and columns to obtain $\tla$ a rectangular partition and $\tmu$ such that $|\Remm(\tmu)|=3$, the result follows.

In case $(iii)$,   we first  deal with the exceptional  case, in which $\ell(\gamma) = w(\delta)=2$.  We remove all but one row or column $R_1, R_2, C_1, C_2$ to obtain a pair $(\tla,\tmu)$.
In the case of $C_1$ (respectively  $R_1$) the partition $\tla$ (respectively $\tmu)$ is a proper fat hook and $\tmu$ (respectively $\tla$) has 3 removable nodes, the result follows.
In the case of $C_2$ (respectively  $R_2$) the partition $\tla$ (respectively $\tmu)$ is a 2-line partition and $\tmu$ (respectively $\tla$) has 3 removable nodes, the result follows.
 Now suppose $\ell(\gamma) > 2$ and $w(\delta)=2$, in which case
  $\tla=(2^{2k+1},1)$  and $\tmu=(4^k,3)$; the result follows  from the case for 2-line partitions.
If  $\ell(\gamma), w(\delta) > 2$, then    $(\tla,\tmu)$  are a  pair of proper fat hooks and the result follows.
\end{proof}

\begin{lem}
If up to rotation and conjugation, one of  $\gamma$ and $\delta$ is equal to $(k-1,1)$ and the other is a fat hook, then $g(\la,\mu)>1$.
\end{lem}

\begin{proof}
By the previous results, we can assume $k>3$ and that  neither $\gamma$ nor $\delta$  is a rectangle.
There are three cases to consider
\begin{itemize} \item[$(i)$]   $\gamma$ and $\delta^{\rm rot}$ are both proper partitions;
\item[$(ii)$] $\gamma^{\rm rot}$ and $\delta^{\rm rot}$ are both proper partitions;
\item[$(iii)$]  $\gamma$ and $\delta$ are both proper partitions.
\end{itemize}
In all three cases, let $(\tla,\tmu)$ denote the pair of partitions obtained by removing all rows and columns common to both $\la$ and $\mu$.
In case $(i)$, we have that $\tla$ is a rectangular partition and $\tmu$ is such that $|\Remm(\tla)|\geq 3$.
In case $(ii)$  we have that $(\tla,\tmu)$ is a pair of proper  fat hooks.
In case $(iii)$, we have that $|\Remm(\tla)|=|\Remm(\tmu)|=2$ and the pair is not on the list of \cref{thm:classification}.   The result follows.
\end{proof}

In summary, we have now proved
\begin{cor}
If $\la$ and $\mu$ is a pair of partitions which does not belong to the list in \cref{thm:classification}, then $g(\la,\mu)>1$.
\end{cor}

Hence the proof of \cref{thm:classification} and thus also the proofs of \cref{thm:3classification} and \cref{thm:classification-skew} are now complete.

\bigskip

\noindent
\textbf{Acknowledgements.}
 We  would like to thank the American Institute of Mathematics and the organisers of  ``Combinatorics and complexity of Kronecker coefficients''
and Banff International Research Station
  for their hospitality.
  We would also like to thank
   Leibniz University Hannover,  EPSRC grant EP/L01078X/1, and the Royal Commission for the Exhibition of 1851 for their financial support towards mutual visits in   Hannover and London.
We  would like to thank
   S.~Assaf,
A.~Hicks,
J.~Remmel,
V.~Tewari,
and S.~van Willigenburg for their interest at some stages of this project.
Finally, we would like to thank M.~Wildon for asking us about \cref{thm:3classification} during the conference
``Kronecker Coefficients 2016''.


\begin{thebibliography}{10}


\bibitem[BO06]{BaOr}
C.~M.~Ballantine, R.~C.~Orellana,
A combinatorial interpretation for the coefficients in the Kronecker product $s_{(n-p,p)}*s_{\lambda}$,
S\'em. Lotharingien de Combinatoire, B54Af (2006), 29 pp. (electronic).

 \bibitem[BK99]{BK} C.\ Bessenrodt, A.\ Kleshchev, On Kronecker products of complex representations of the symmetric and alternating groups, Pacific J. of Math. 190 (1999), no. 2, 536--550.

 \bibitem[BW14]{BW14} C.\ Bessenrodt, S.\ van Willigenburg, On (almost) extreme components in Kronecker products of characters of the symmetric groups, J. Algebra  410 (2014), 460--500.

 \bibitem[BB04]{BB04} C.\ Bessenrodt; C.\ Behns, On the Durfee size of Kronecker products of characters of
              the symmetric group and its double covers,
               J. Algebra  280 (2004), 132--144.

\bibitem[Bla14]{Bla14}
J.\ Blasiak,   Kronecker coefficients for one hook shape,  http://arxiv.org/abs/1209.2018.

\bibitem[BWZ10]{BWZ10}
A.~A.~H.\ Brown, S.\ van Willigenburg, M.\ Zabrocki,
Expressions for Catalan Kronecker products,
Pacific J. Math. 248 (2010), 31--48.

 \bibitem[CHM07]{CHM07} M.\ Christandl, A.\ Harrow, G.\ Mitchison, Nonzero Kronecker coefficients and what they tell us about spectra,
Comm. Math. Phys. 270 (2007), no. 3, 575--585.

\bibitem[CM93]{CM}
M.\ Clausen, H.\ Meier,
Extreme irreduzible Konstituenten in Tensordarstellungen
symmetrischer Gruppen,
{\em Bayreuther Math. Schr.} 45 (1993), 1--17.

\bibitem[Dvi93]{Dvir}
Y.\ Dvir,
 On the Kronecker product of $S_n$ characters,
J.\ Algebra 154 (1993), 125--140.

\bibitem[GWXZ]{GWXZ}
A.\ Garsia, N.\ Wallach, G.\ Xin, M.\ Zabrocki,
Kronecker coefficients via symmetric functions and constant term identities,
Internat. J. Algebra Comp. 22 (2012), 1250022, 44 pp.

\bibitem[Gut10a]{CG1}
C.\ Gutschwager,
Reduced Kronecker products which are multiplicity free or contain only few components,  European J. Combin. 31 (2010), 1996--2005.

\bibitem[Gut10b]{CG}
C.\ Gutschwager,
On multiplicity-free skew characters and the Schubert Calculus,
Annals of Combinatorics  14 (2010), 339--353.

\bibitem[Gut]{CGposet}
C.\ Gutschwager,
The skew diagram poset and components of skew characters,
arXiv:1104.0008

 \bibitem[Jam77]{JAM}
G.~D.~James,  A characteristic-free approach to the representation theory of $S_n$,
 J. Algebra 46 (1977),  430--450.

\bibitem[JK81]{JK}
G.\ James, A.\ Kerber,
{\em The representation theory of the symmetric group},
Addison-Wesley, London, 1981.


\bibitem[LR34]{LR}
 D.\  Littlewood, A.\ Richardson,  Group characters and algebra, Phil. Trans. A 233, (1934), 99--141.



\bibitem[Man10]{Man10}
L.\ Manivel, A note on certain Kronecker coefficients,
Proc. Amer. Math. Soc. 138 (2010), 1--7.


\bibitem[Man11]{Man11}
L.\ Manivel, On rectangular Kronecker coefficients,
J. Algebr. Comb. 33 (2011),  153--162.



\bibitem[Mul07]{GCT6} K. Mulmuley, Geometric complexity theory VI: the flip via saturated and positive integer programming in representation theory and algebraic geometry, Tech. Report TR--2007-04, Computer Science Department, The University of Chicago, May 2007.

\bibitem[Mur38]{Mur38}
F.\  Murnaghan, The analysis of the Kronecker product of irreducible representations of the symmetric group,
Amer. J. Math.   (1938), 761--784.


\bibitem[PP]{PP1}
I.\ Pak, G.\ Panova,
Bounds on Kronecker and $q$-binomial coefficients,
arXiv:1410.7087v2.


\bibitem[Rem89]{R}
J.\ Remmel,
A formula for the Kronecker products of Schur functions of hook shapes,
J.\ Algebra 120 (1989), 100--118.

\bibitem[Rem92]{R92}
J.\ Remmel,
Formulas for the expansion of the Kronecker products
$S_{(m,n)}\otimes S_{(1^{p-r},r)}$ and
$S_{(1^k 2^l)}\otimes S_{(1^{p-r},r)}$,
Discrete Math. 99 (1992), 265--287.

\bibitem[RW94]{RW}
J.\ Remmel, T.\ Whitehead,
On the Kronecker product of Schur functions of two row shapes,
Bull.\ Belg.\ Math.\ Soc. 1 (1994), 649--683.


\bibitem[Ro01]{Rosas}
M.~H.\ Rosas,
The Kronecker product of Schur functions indexed by 2-row
shapes or hook shapes,
J.\ Algebraic Combin. 14 (2001), 153--173.


\bibitem[Sax87]{Saxl}
J.\ Saxl,
The complex characters of the symmetric groups
that remain irreducible in subgroups,
J.\ Algebra 111 (1987), 210--219.
				
		
\bibitem[Sta99]{Sta99}
R. P. Stanley, Enumerative Combinatorics, Vol. 2, Cambridge U. Press, Cambridge, 1999.

\bibitem[Sta00]{Sta00}
 R. P. Stanley. Positivity problems and conjectures in algebraic combinatorics, in: Mathematics: Frontiers and Perspectives, 295--319. Amer. Math. Soc., 2000.

\bibitem[Ste01]{Stembridge01}
J.\ Stembridge,
On multiplicity-free products of Schur functions,
Ann.\ Comb. 5 (2001), 113--121.



\bibitem[TY10]{TY}
H.\ Thomas and A.\ Yong,
Multiplicity-free Schubert calculus,
Canad.\ Math.\ Bull. 53 (2010), 171--186.

\bibitem[Val97]{V}
E.\ Vallejo,
On the Kronecker product of the irreducible characters of the symmetric group,
{\em Pub.\ Prel.\ Inst.\ Mat.\ UNAM} 526 (1997).

\bibitem[Val14]{V14}
E.\ Vallejo,
A diagrammatic approach to Kronecker squares,
J. Comb. Theory A 127 (2014), 243--285.

\bibitem[Will05]{vW_schur}
S.\ van Willigenburg,
Equality of Schur and skew Schur functions,
Ann.\ Comb. 9 (2005), 355--362.


\bibitem[Zis92]{Z}
I.\ Zisser,
The character covering numbers  of the alternating groups,
J.\ Algebra 153 (1992), 357--372.
\end{thebibliography}
\end{document}